\documentclass[12pt, twoside]{Thesis} 

\graphicspath{{Pictures/}} 
\usepackage{comment}
\usepackage{xypic}
\usepackage{lscape}
\usepackage[square, numbers, comma, sort&compress]{natbib}
\hypersetup{urlcolor=blue, colorlinks=true} 
\title{\ttitle} 

\usepackage{amsmath}
\usepackage{amssymb}
\usepackage{amsthm}
\usepackage{amsfonts}
\usepackage{dsfont}
\usepackage{tkz-graph} 
\usepackage{caption}
\newcommand{\para}[1]{#1}
\newcommand{\adele}{\mathbb}
\newcommand{\places}{\mathds}
\newcommand{\adelegroup}[2]{{#1}(\mathbb{#2} ) }
\newcommand{\weyl}[{1}]{\mathtt{W}_{#1}}

\newcommand{\weylelement}[1]{w_{#1}}
\newcommand{\fundamental}[1]{\bar{\omega}_{#1}}
\newcommand{\automorphicspace}[2]{\mathcal{A}(#1(#2)\backslash #1(\adele{A}))}
\newcommand{\quoientspace}[3]{#1(#2)\backslash #1(\adele{#3})}
\newcommand{\leadingterm}[1]{\Lambda_{#1}}
\newcommand{\lie}{\mathfrak}
\newcommand{\inner}[1]{\langle #1 \rangle}
\theoremstyle{plain}
\newtheorem{thm}{Theorem}[chapter] 
\newtheorem{prop}{Proposition}[thm]
\newtheorem{lemm}{Lemma}[thm]
\theoremstyle{definition}

\newtheorem{defn}[thm]{Definition} 
\newtheorem{remk}{Remark}[thm]

\newtheorem{ques}{Question}
\newtheorem{asmp}{Assumption}

\newtheorem*{prop*}{Proposition}
\newtheorem*{thm*}{Theorem}
\newtheorem*{cor*}{Corollary}
\newtheorem*{obv}{Observation}
\newtheorem{cor}[thm]{Corollary}

\theoremstyle{Que}
\usepackage{tikz}
\usepackage{cleveref}
\usetikzlibrary{%
  arrows,%
  shapes.misc,
  shapes.arrows,%
  chains,%
  matrix,%
  positioning,
  scopes,%
  decorations.pathmorphing,
  shadows%
}
\usepackage{mathtools}

\usepackage{float}
 \allowdisplaybreaks[1]
\usepackage{multirow} 
\usepackage{colortbl}
\definecolor{kugray5}{RGB}{224,224,224}
\usepackage{longtable}
 \newcolumntype{H}{>{\setbox0=\hbox\bgroup}c<{\egroup}@{}}
 \usepackage{pifont}
 \newcommand{\cmark}{\ding{51}}
 \newcommand{\xmark}{\ding{55}}
 \newcolumntype{K}[1]{>{\centering\arraybackslash}p{2cm}}
\begin{document}
\frontmatter
\thispagestyle{empty}

\noindent
\begin{center}
    \Large
    Ben-Gurion University of the Negev\\
    The Faculty of Natural Sciences\\
    Department of Mathematics\\
\end{center}

\vfill
\begin{center}
    \Huge\bfseries
Poles of degenerate Eisenstein series and Siegel-Weil identities for exceptional split groups  
\end{center}

\vfill\vfill
\begin{center}
    \large
    Thesis Submitted in Partial Fulfillment of the Requirements 		for the Master of Sciences Degree\\
\end{center}

\vfill
\begin{center}
    \huge\bfseries
    By: Hezi Halawi
\end{center}

\vfill\vfill\vfill
\begin{center}
    \Large
    Under the Supervision of: Dr Nadya Gurevich
\end{center}

\vfill
\begin{center}
\large
    Beer Sheva, August 2016
\end{center}

\newpage
\thispagestyle{empty}
\vspace*{\fill}
{\hfill\sffamily\itshape $ $}
\cleardoublepage

\rmfamily
\normalfont

\newpage
\thispagestyle{empty}
\noindent
\begin{center}
    \Large
    Ben-Gurion University of the Negev\\
    The Faculty of Natural Sciences\\
    Department of Mathematics\\
\end{center}

\vfill
\begin{center}
    \Huge\bfseries
Poles of degenerate Eisenstein series and Siegel-Weil identities for exceptional split groups  
\end{center}

\vfill\vfill
\begin{center}
    \large
    Thesis Submitted in Partial Fulfillment of the Requirements 		for the Master of Sciences Degree\\
\end{center}

\vfill
\begin{center}
    \huge\bfseries
    By: Hezi Halawi
\end{center}

\vfill\vfill\vfill
\begin{center}
    \Large
    Under the Supervision of: Dr Nadya Gurevich
\end{center}

\vfill
\begin{center}
    \large
    Signature of Student: \underline{\qquad\qquad\qquad\qquad}  \qquad   Date: \underline{\qquad\qquad}\\
    Signature of Supervisor: \underline{\qquad\qquad\qquad\qquad} \quad       Date: \underline{\qquad\qquad} \\
    Signature of Chairperson 
of the Committee for Graduate Studies: \underline{\qquad\qquad\qquad\qquad} \quad Date: \underline{\qquad\qquad}
\end{center}

\vfill
\begin{center}
\large
    Beer Sheva, August 2016
\end{center}

\newpage

\thispagestyle{empty}

\newpage

\chapter*{\centering Abstract}
\addcontentsline{toc}{chapter}{Abstract}
Let $G$ be a linear split algebraic group.
The degenerate Eisenstein series associated to a maximal parabolic subgroup $E_{\para{P}}(f^{0},s,g)$  with the spherical section $f^{0}$  is studied in the first part of the thesis. In this part, we study the poles of $E_{\para{P}}(f^{0},s,g)$ in the region $\operatorname{Re}  s >0$. We  determine when the leading term in the Laurent expansion of 
 $E_{\para{P}}(f^{0},s,g)$ around $s=s_0$  is square integrable.
The second part is devoted to finding identities between the leading terms of various Eisenstein series at different points.
We present an algorithm to find this data and implement it on \textit{SAGE}.
While both parts can be applied to a general algebraic group, we restrict ourself to the case where $G$ is split exceptional group of type 
$F_4,E_6,E_7$, and obtain new results.
\newpage

\thispagestyle{empty}
$ $

\newpage

\thispagestyle{empty}

\tableofcontents
\addtocontents{toc}{\protect\setcounter{tocdepth}{1}}

\mainmatter 

\pagestyle{fancy} 

\label{ch:Preface} 

\lhead{Chapter 0 . \emph{Preface}} 
\chapter*{Preface}
\addcontentsline{toc}{chapter}{Preface}
\subsection*{Classical Eisenstein series}
Let $G$ be a linear Lie group with a  Lie algebra $\mathfrak{g}$. Let $Z(\mathfrak{g})$ be the center of  the universal enveloping algebra  $\mathcal U (\mathfrak{g})$.
Let $K$ be a  maximal compact subgroup of $G$.  There is a  faithful representation
 $G \hookrightarrow GL_{n}$ that allows to define a norm  $||\; ||$ on $G$. 
Finally let  $\Gamma$ be a discrete subgroup of $G$.

Let $\mathcal A(\Gamma\backslash G)$ be the space of
automorphic forms, i.e  all functions $f : G \rightarrow \mathbb{C}$ that  satisfy the following conditions:
\begin{enumerate}
\item
$f$ is smooth.
\item
$f$ is invariant under the action of $\Gamma$  i.e. for every $\gamma \in \Gamma  $ $ f(\gamma g)=f(g)$.
\item
$f$  is $K$ -- finite, i.e. the space spanned  by the translations of $f$ under elements of $K$ is finite dimensional.  
\item
$f$ is $Z(\mathfrak{g})$ finite.
\item
$f$ is of moderate growth i.e. there exists  $C>0$ and $n \in \mathbb{N}$ such that for all $g \in G$ it holds that 
$ ||f(g)|| \leq C ||g||^{n}$.
\end{enumerate}  
An automorphic form $f$ is called  \textbf{spherical} if it is in addition right invariant under the action of the maximal compact subgroup $K$.
In this case the automorphic function  can be consider as  function on the space of cosets
 $\Gamma\backslash G\slash K$.
The classical example is as follows:\\ 
Let $G=SL_2(\mathbb{R})$ whose maximal compact subgroup is   $K=SO_2({\mathbb{R}})$. 
The space $G/K$ can be identified with the upper half plane  
$\mathbb{H}=\{x+iy \in \mathbb{C} \: : \: y >0\}$ via $g\rightarrow g.i$ 
where $G$ acts on $ \mathbb H$ by Mobius transformations.  
One has $Z(\mathfrak{g}) \simeq \mathbb{C}[\Delta]$ where 
$\Delta=y^{2}(\frac{\partial^{2}}{\partial^{2}{x}} + \frac{\partial^{2}}{\partial^{2}{y}}).$ 
Finally let  $\Gamma =SL_{2}(\mathbb{Z}).$

For a complex number $s$ define $h_s: \mathbb{H} \rightarrow \mathbb{C}$  to be 
\begin{equation}\label{First_auto_form}
h_s(z)=\sum_{\tiny{
\begin{aligned}
(c,d)\in \mathbb{Z}^{2}\\
(0,0)\neq (c,d)
\end{aligned}}}\frac{y^{s+\frac{1}{2}}}{|cz+d|^{2s+1}}
\quad z=x+iy\in \mathbb{H}
\end{equation}
The sum converges absolutely for $Re \: (s) > \frac{1}{2}$. For fix $z \in \mathbb{H}$ the function $h_{s}(z)$ admits a meromorphic continuation to the whole complex plane.

 As a function of $z\in \mathbb{H}$ the series is non-holomorphic but satisfies a differential equation:
$$\Delta h_s(z)= (s-\frac{1}{2})(s+\frac{1}{2})h_s(z).$$ 

By re-arranging the sum in (\ref{First_auto_form}) we get 
$$h_s(z)=2 \zeta(2s+1) \sum_{\tiny{
\begin{aligned} 
(c,d)\in \mathbb{Z}^{2}\\
gcd(c,d)=1
\end{aligned}}}\frac{y^{s+\frac{1}{2}}}{|cz+d|^{2s+1}}=
2\zeta(2s+1) \sum_{\gamma \in \para{B}(\mathbb{Z})\backslash \Gamma}[{Im\:\gamma z}]^{s+\frac{1}{2}} $$

where $\para{B}(\mathbb{Z})$ is the group of  upper triangular matrices in $\Gamma$. Therefore by dropping the $\zeta$ factor we obtain the first example of an Eisenstein series on $G$
$$E(z,s)=\sum_{\gamma \in \para{B}(\mathbb{Z}) \backslash \Gamma} \chi_s(\gamma.z), \quad \quad Re \: s >\frac{1}{2}$$
when $\chi_s(z)=[Im\: z]^{s+\frac{1}{2}}$.
Notice that $K$ leaves the point $i\in \mathbb{H}$ invariant. By abuse of notation we denote by $\chi_s$ the pullback of $\chi_s$ to $G$ as well. Therefore,
the Eisenstein series $E(z,s)$ gives rise to 
an automorphic function 
$$E(\cdot,s):  \Gamma\backslash  SL_2(\mathbb{R})/SO(2,\mathbb{R})\rightarrow \mathbb{C}$$ defined by 
$$E(g,s)=\sum_{B(\mathbb{Z})\backslash \Gamma} \chi_s(\gamma g), \quad Re(s)>\frac{1}{2}.$$ 
 
Let us reinterpret this function from the perspective of representation theory. 
Recall the Iwasawa decomposition $G=N\cdot T \cdot K$ where 
$N$ is the group of  upper triangular unipotent matrices and $T$ is the group 
of invertible diagonal matrices of $G$.   
The function $\chi_s:G\rightarrow \mathbb C$ satisfies: 
$$\chi_s(ntk)=\delta_B(t)^{s+1/2}$$ and hence can be regarded as a spherical 
vector in $\operatorname{Ind}^{SL_2(\mathbb{R})}_{\para{B}(\mathbb{R})} \delta_\para{B}^s$ (see Definition (\ref{ind_def}) for more details). 

More generally we define $E(\cdot,\cdot,s)$   an operator
from  $\operatorname{Ind}^{SL_2(\mathbb{R})}_{\para{B}(\mathbb{R})} \delta_\para{B}^s$ to the space of automorphic forms 
$\mathcal A(\Gamma\backslash SL_2(\mathbb{R}))$  by 
$$E(g,\phi,s)=\sum_{\para{B}(\mathbb{Z})\backslash \Gamma} \phi(\gamma g,s)$$
where $\phi(g,s)$ is any flat section  in the induced representation
 i.e. section $\phi(g,s)$ whose restriction to  $K$ it independent on $s$.

\subsection*{Motivation}
The modern theory of automorphic forms concerns  functions on adelic groups. Let $F$ be a number field and $\adele{A}$ be its ring of adeles. 
Let $G$ be a split algebraic group. 
The group $G(F)$ is a discrete subgroup in $G(\adele{A})$.
\\
From perspective of
representation theory one can think of an Eisenstein series as $G(\adele{A})$-equivariant
map from induced representations  into the space of automorphic forms $\mathcal{A}(G(F) \backslash G(\adele{A}))$. \\
Precisely, let $\para{P}=MN$ be a standard parabolic subgroup of a reductive group $G$ and $\sigma$ be an automorphic representation of $M$, i.e. $\sigma$ is realized in the space of automorphic functions on $M$ 
(see section \cite{auoto_section} for more details). Let $X^\ast(M)=Hom(M,G_m)$ be the set of algebraic characters of $M$. We denote by 
$I_{\para{P}}(\lambda\otimes \sigma) $ the induced representation. The family of operators
$$ E(\cdot,\lambda): I_{\para{P}}(\sigma \otimes \lambda) \rightarrow \mathcal{A}(G(F) \backslash G(\adele{A})) \quad \lambda \in X^{\ast}(M) \otimes \mathbb{C}$$
is defined by 
$$E(f,g,\lambda)=\sum_{\gamma \in \para{P}(F) \backslash G(F)}f(\gamma g,\lambda)(1)$$
for $\lambda$ in a shifted dominant chamber where the series converges. By a fundamental theorem of Langlands the operator $E$ admits a meromorphic continuation to the space $X^{\ast}(M)\otimes \mathbb{C}$.        
\\
 One of the main goals of the theory of automorphic forms is 
to decompose $L^2(G(F)\backslash G(\adele{A}))$ into irreducible representations and for this,  Eisenstein series is an indispensable tool.
 
 The space $L^2(G(F)\backslash G(\adele{A}))$ can be written as sum of two orthogonal subspaces  
$$L^2(G(F)\backslash G(\adele{A}))=
L^2_{disc}(G(F)\backslash G(\adele{A}))
\oplus
L^2_{cont}(G(F) \backslash G(\adele{A}))
$$
corresponding respectively to the discrete and continuous parts of the spectrum.  Moreover, Langlands showed that $L^2_{disc}(G(F)\backslash G(\adele{A}))$ can be decomposed as 
$$L^2_{disc}(G(F)\backslash G(\adele{A}))=
L^2_{res}(G(F)\backslash G(\adele{A})) 
\oplus
L^2_{cusp}(G(F)\backslash G(\adele{A}))
$$
where $L^2_{res}$ is spanned by iterated residues of Eisenstein series associated to automorphic cuspidal representations and $L^2_{cusp}$ is the space of cusp forms.
In this thesis we  will concentrate of a different type 
of Eisenstein series, the ones associated to a degenerate principal series 
$I_\para{P}(s)=Ind^{G}_{\para{P}} (\delta_\para{P}^{s})$, 
where $\para{P}$ is a maximal parabolic subgroup.
Let us give a partial list of applications of 
degenerate Eisenstein series. 
 
\begin{enumerate}
\item 
The famous Siegel-Weil formula relates a period of 
the theta function to a special values of a degenerate 
Eisenstein series.  
\item 
The degenerate Eisenstein series occur in  
Rankin-Selberg integral representations of L-functions 
and hence its poles determine the poles of L-functions. 
See (\cite{GPSR}, \cite{PSR}, \cite{GS}). 
\item 
The minimal representation of $G$, when exists, 
can be realized as a residue of degenerate Eisenstein series. See \cite{Ginzburg1997}.
\end{enumerate}

In this thesis we study the poles of  degenerate
Eisenstein  series in the right half plane  and
the  relation between their leading terms. Previous studies on this area can be viewed for example in  \cite{G2_Muic},\cite{GL_n},\cite{Dih1}. 
We obtain new results for 
split exceptional groups where $G=F_4,E_6,E_7$. 

\subsection*{Strategy to studied poles of Eisenstein series}
By the fundamental theorem of Langlands see[],  
the poles of the degenerate Eisenstein series $E_{\para{P}}(f,g,s)$
coincide with  the poles of its constant term along the the Borel subgroup
$$E_{\para{P}}^{0}(f,g,s)=\int_{U(F)\backslash U(\adele{A})} E_{\para{P}}(f,ug,s)du $$
where $U$ is the unipotent radical of the Borel.
\begin{thm*}
The constant term $E_{\para{P}}^{0}(f,g,s)= \sum _{\weylelement{}\in \weyl{G}\slash \weyl{\para{P}}} M_{\weylelement{}}(s) f^{0}_{s}$
where $M_{\weylelement{}}$ is the intertwining operator
\end{thm*}
The definition and the properties of $M_{\weylelement{}}(s)$ are summarized in Theorem (\ref{intertwiningProp}).
Thus, the union set of poles of various intertwining operators contains
the set of poles of Eisenstein series, but by no means is equal to it. 
The reason is  possible cancellations of the poles from various 
summands. Moreover, even if a pole is not canceled its order 
can be lower then the order of the pole of each summand as in the case that it is canceled.

Although many examples have been worked out the  nature
of these cancellations remains a mystery.

 The degenerate induced representation associated to the trivial 
 character contains unique spherical section $f^0_{s}(g)$ 
 normalized such that $f^0_{s}(1)=1$.
 Using  Gindikin-Karpelevich formula, Theorem (\ref{gid}), for the
 action of $M_w(s)$ on the spherical function we can 
 evaluate the sum of intertwining operators applied
 to the vector $f^0_{s}(g)$ and witness the cancellations.

 The approach is computational. We have produced an algorithm
 that being realized in the {\sl SAGE}  computes 
 the  poles and their order of all spherical degenerate Eisenstein
 series for split groups of small rank.
 
 \subsection*{ Our Results:  for the spherical Eisenstein series $E_{\para{P}}(f^0,g,s)$}
 We have
 obtained new results for
 the split exceptional groups  of type $F_4, E_6,E_7$.
 \begin{thm*}
 Let $G=E_7$ and $\para{P}=\para{P}_4$ then the Eisenstein series $E_{\para{P}}(f^{0},s,g)$ admits a pole of order $4$ at $s=\frac{1}{8} $.
 \end{thm*} 
 To the best of our knowledge this is the first examole of Eisenstein series with order equal to $4$.
 The results we obtain for $G_2, GL_n, n\le 8$ and $Sp_n, n<8$ 
 coincide with the results obtained earlier in \cite{G2_Ginz},\cite{GL_n},\cite{Sp4}.  \\
 We find out that all the poles for $\operatorname{Re} s >0$ are real.\\
 \newpage
 \begin{thm*}(Theorem (\ref{Upper}))
 If $F=\mathbb{Q}$ and $s_0 \in \mathbb{R}$  then the order of a pole of 
 $E_{\para{P}}(f^0,g,s)$  at a point $s_0$ is bounded by
 $d_{\para{P}}(\chi_{s_0})$ where 
  $$d_{\para{P}}(\chi_{s_0})=\left|\alpha>0, \langle \alpha^\vee, \chi_{s_0} \rangle=1\right|-
 \left|\alpha>0, \langle \alpha^\vee, \chi_{s_0} \rangle=0\right|-(n-1)$$
 and $\chi_{s_0}=\delta_{\para{P}}^{s_0-\frac{1}{2}}\otimes \delta_{\para{B}}^{\frac{1}{2}}$.
 . 
 \end{thm*}
\begin{obv}
If $F=\mathbb{Q}$ then 
for all the cases that we studied the order of the spherical Eisenstein series
is actually equal to $d_{\para{P}}(\chi_{s_0})$. 
\end{obv}

 \subsection*{ The residual representation}
 Suppose that $E_{\para{P}}(\cdot,g,s)$ admits a pole of order $m$ at $s=s_0$ . Consider its  Laurent expansion 
 at $s=s_0$
 $$E_{\para{P}}(\cdot,s,g) = \sum_{i=-m}^{\infty} \leadingterm{-i}^{\para{P}}(\cdot,s_0,g)(s-s_0)^{i}.$$ 
 
\begin{prop*}
The leading term in the Laurent expansion 
 $\leadingterm{-m}^{\para{P}}(\cdot,s_0,g)$
 , i.e. the first term that it is not zero,
  defines an intertwining map between 
 $I_{\para{P}}(s_0)$ and  the space of automorphic
 forms.
\end{prop*}  The image is called \textbf{residual representation}
 and  is denoted by $\Pi$.
 
 The spherical vector in $I_{\para{P}}(s)$ 
 generates a subrepresentation $I^0_{\para{P}}(s)$. 
 If $E_{\para{P}}(f^0,g,s)$ admits a pole of order $m$ at $s=s_0$ 
 then the leading term in the Laurent expansion 
 $\leadingterm{-m}^{\para{P}}(f^{0},s_0,g) $
 defines an intertwining map between 
 $I^0_{\para{P}}(s_0)$ and  the space of automorphic
 forms. The image is called \textbf{spherical residual representation}
 and is denoted by $\Pi^0$.
 
 It is expected that for any $s$ such that $\operatorname{Re} s >0 $
 the highest order of the pole 
 $E_{\para{P}}(\cdot,g,s)$ will be attained by the spherical section. 
 While the analogous statement is known for the 
 intertwining operators $M_w(s)$ the expectation is not proven 
 in general. Indeed the cancellation occurring for the spherical 
 vector might not occur for an arbitrary section.
 
 \subsection*{Square-integrability}
 
 To describe the residual representation 
 $\Pi^0$ it is useful to know whether the function 
 $\leadingterm{-m}^{\para{P}}(f^0,s_0,g)$ is contained in  $L^2(G(F)\backslash G(\adele{A})).$
 To find this out we use  Langlands' criterion Theorem (\ref{square_init}) . 
  \\
  If this is the case the spherical residual representation 
  is necessary a direct sum of irreducible representations. In particular, 
  $\Pi^0$ is isomorphic to the unique spherical irreducible 
  quotient of $I_{\para{P}}(s_0)$.
  
  In case $I_{\para{P}}(s_0)$ has unique irreducible quotient and 
  $\leadingterm{-m}^{
  \para{P}}(f^{0},s_0,g)$ is square integrable the residual representation 
  is isomorphic to the unique spherical quotient of $I_{\para{P}}(s_0)$.

 All square integrable residual representations comming from degenerate Eisenstein series have cuspidal support $[T,1]$. In the recent  paper  \cite{op}, the space \\$L^{2} (G(F) \backslash G(\adele{A}))_{[T,1]}^{K}$,
 consisting of square integrable functions
  having cuspidal support $[T,1]$ and as representation of $G(\adele{A})$  generated by the spherical section, is studied. 
 It discrete summands correspond to distinguished nilpotent orbits of the Lie algebra $^{L}\lie{g}$ of the dual group.    We write explicitly the distinguished orbits 
 corresponding to the residual representations that are square integrable see subsections (\ref{G2Orbirts}),(\ref{F4Orbirts}),(\ref{E6Orbirts}),(\ref{E7Orbirts}).

\subsection*{Identities}

Some automorphic representations can be realized in several ways as leading terms 
of degenerate Eisenstein series. 
 For example the trivial function can be realized as the residue of $E_{\para{P}}(f^0,s,g)$ at $s = \frac{1}{2}$
for any maximal parabolic subgroup $P$. This is an  analogue of  the statement that a representation
can be a common quotient of several degenerate principal series. 

In the second part of the thesis we explore identities involving the leading terms
 of various Eisenstein series.
Generally speaking, we would like to realize a given 
representation as residues of various Eisenstein series
in a hope to gain additional information about it.

For example, 
the minimal representation of the exceptional group $E_7$ 
has been realized as a residue of a degenerate Eisenstein 
series associated to a non-Heisenberg parabolic subgroup in \cite{Ginzburg1997}
and later as  a residue at a different point for 
Eisenstein series associated to a Heisenberg parabolic subgroup. 
The second realization proved to be useful for certain exceptional 
theta lift \cite{NadiaTheis}.

A necessary condition for the identity between leading terms of Eisenstein series is the induced representations $I_{\para{P}}^{0}(s_0)$, $I_{\para{Q}}^{0}(t_0)$  having a common quotient.
These representations are subquotients of 
$I_{\para{B}}(\chi_{s_0})$ 
$I_{\para{B}}(\chi_{t_0})$ respectively
 where $\chi_{s_0}= \delta_{\para{P}}^{s_0-\frac{1}{2}} \otimes 
\delta_{\para{B}}^{\frac{1}{2}}$ and
$\chi_{t_0}= \delta_{\para{Q}}^{t_0-\frac{1}{2}} \otimes 
\delta_{\para{B}}^{\frac{1}{2}}$. This leads to the following definition
 \begin{defn}
 The quintuple $(\para{P},s_0,\para{Q},t_0,\weylelement{})$ where $\para{P},\para{Q}$ are maximal parabolic subgroups, 
 $s_0,t_0 \in \mathbb{R}$ and $\weylelement{}\in \weyl{G}$ is called  \textbf{admissible data} if 
$$\weylelement{} ( \delta_{\para{P}}^{s_0-\frac{1}{2}} \otimes 
 \delta_{\para{B}}^{\frac{1}{2}}) = \delta_{\para{Q}}^{t_0-\frac{1}{2}} \otimes 
 \delta_{\para{B}}^{\frac{1}{2}}.$$
 If also $s_0,t_0$ are non negative numbers it is called  \textbf{positive admissible data}.
 \end{defn}
The (positive) admissible  data  can be found by direct computation as explained in Chapter \ref{ch3:ids}. In this chapter we also list  all the positive admissible data 
for the exceptional groups $G_2, F_4, E_6, E_7$

Note the curious chains of the admissible data. We have pairs 
of the form $(\para{P}_1,s_1,\para{P}_2,s_2,w_1)$ and $(\para{P}_2,s_2,\para{P}_3,s_3,w_2)$ leading to 
the chains of identities. See section \ref{section:E6}.

Let us  state  our main theorem regarding the spherical 
Eisenstein series.

\begin{thm}\label{intomain}
 Let $F=\mathbb{Q}$ and  $(\para{P},s_0,\para{Q},t_0,\weylelement{})$ be an  admissible data.
Let $f_\para{P}^0, f^0_\para{Q}$ be  spherical sections in $I_{\para{P}}(s_0),
 I_{\para{Q}}(t_0)$
 respectively.

There exists a constant $C \in \mathbb{C}^{\ast}$ 
that depends on the admissible data
such that
$$\leadingterm{-d_P(\chi_{s_0})}^{\para{P}}(f_\para{P}^0,s_0,g)=
C \leadingterm{-d_Q(\chi_{t_0})}^{\para{Q}}(f_\para{Q}^0,t_0,g).$$
\end{thm} 
\begin{remk}
We assume that $F =\mathbb{Q}$ since in that case we know that $\zeta$ does not have any zeros in the real line.
\end{remk}
For $F=\mathbb{Q}$ we observe if $E_{\para{P}}(f^{0},s,g)$ admits a pole of order $d$ at $s=s_0$ for $\operatorname{Re} s_0>0$ then $d=d_{\para{P}}(\chi_{s_0})$.
Hence Theorem (\ref{intomain}) is an identity between the leading terms.

In Chapter \ref{ch:Technical lemmas} we give an explicit formula for the constant in Theorem (\ref{intomain}).
For the groups $G=F_4,E_6,E_7$ we list all the positive admissible data and determine  the constant. Using this we write the identities explicitly.

The identities of this type has been explored before by 
Dihua Jiang for the symplectic groups in \cite{Dih1}.
\newpage
\subsection*{Thesis structure}
Below is the outline of the thesis.

Chapter \ref{Ch:Introduction} is introductory. 
In section \ref{Ch1:section1}  we set all the notations for an algebraic
group $G$ defined over a field $F$, such as root datum and Weyl group. 

In section \ref{Ch1:section2} we recall the standard definitions and facts
from the representation theory of split reductive group $G(F)$,  where $F$ is a local field. 
In particular we define induced representations and intertwining operators
between them and recall their properties. 

Section \ref{Ch1:section3} refers  to the space of automorphic forms of $G$.
We define a degenerate Eisenstein  series, global intertwining operators
and recall their properties. 

In Chapter \ref{ch4:algo} we explain our basic algorithm that allows 
us to compute all the poles and their orders
of all degenerate spherical Eisenstein series $E_P(f^0,g,s)$  
in the right-half plane for the groups $G=G_2,F_4,E_6,E_7$ .
\\
Our algorithm works for any group whose rank is  not too big. 
In particular the results we obtain for $G_2$ and $GL_n, n\le 8$ 
agree with the results of \cite{G2_Ginz} and \cite{GL_n}.

In Chapter \ref{ch2:Poles} we present the results for the groups of type $G_2,F_4,E_6,E_7$.

In  Chapter \ref{ch2:NormalizedEise} 
we prove the preliminary version of our main theorem. 
Our approach follows Ikeda \cite{Ike1}. 

Using the $W$ invariance and entireness of the normalized Eisenstein series 
we prove the identity between two leading terms of degenerate Eisenstein series at an admissible data 
up to a non-zero constant. 

We compute this constant in Chapter \ref{ch:Technical lemmas} .

The Chapter \ref{ch3:ids} is devoted to the search of positive admissible data.  
We list all of them for the groups $G=F_4,E_6,E_7$.
Note  the curious chains of length three for the group $E_6$. 
This phenomenon seems to be new.  

The sample of computation
for certain parabolic of $G=F_4$ at a fixed point is attached
in the appendix \ref{Ap1:Code_example}

\chapter{Introduction} 

\label{Ch:Introduction} 

\lhead{Chapter \ref{Ch:Introduction}. \emph{Introduction}} 


\section{Algebraic groups and their structure}\label{Ch1:section1}
Let $F$ be a field, and let $G$ be an algebraic reductive split group over $F$. We denote by $G(F)$ the group of $F$ points of $G$. 
For every algebraic group we can define its Lie algebra by the following procedure.

Let $A=F[G]$, we consider the Lie algebra $Der(A)$ of all $F$-- derivations $f : A \rightarrow A$ with the Lie bracket 
$[f,g] = f \circ g -g \circ f$.
On $A$ we can define a left action of $G$ as follows: for every $g \in G$  we denote by $\lambda_g$ the action 
$\lambda_g(f)(h) =f(g^{-1}h)$ for every $h \in G$.

\begin{defn}
A derivation $\delta\in Der(A)$ is called a left--invariant if it commutes with left translations i.e. for every $g \in G$ it holds $\delta \circ \lambda_g=\lambda_g \circ \delta$.  
\end{defn}
The left invariant derivations of $A$ form a Lie subalgebra of $Der(A)$, called the Lie algebra of $G$ and denoted by $\mathfrak{g}$.

We fix a maximal split torus $T$ of $G$.
Let $X(T)=Hom(T,\mathbb{G}_m)$ be the character group of $T$, and let $Y(T)=\operatorname{Hom}(\mathbb{G}_m,T)$ be the cocharacter
group. Both $X(T)$ and $Y(T)$ are free abelian groups of rank $\operatorname{dim}\: T$.
Moreover since the $\operatorname{Aut}(\mathbb{G}_m)=\mathbb{Z}$ there is a pairing 
$$  \left \langle \; , \: \right \rangle: \: X(T) \times Y(T) \rightarrow \mathbb{Z}.$$

The maximal torus  $T$ acts on the Lie algebra $\mathfrak{g}$ of $G$ via the adjoint action. Since $G$ is reductive - the zero  eigenspace of $\mathfrak{g}$ with respect to $T$ is exactly the Lie algebra  $\mathfrak{t}$ of $T$ itself. Therefore we can decompose $\mathfrak{g}= \mathfrak{t} \oplus \bigoplus_{\alpha \in \Phi_{G}} \mathfrak{g}_{\alpha}$,
where $\Phi_{G}$ is the set of all $0 \neq \alpha \in X(T)$ such that $\mathfrak{g}_{\alpha}=\{x \in \mathfrak{g} \: | t.x=\alpha(t)x \:\: \forall t\in T\} \neq 0$. The set $\Phi_{G}$ is called the set of roots of $G$.

\begin{defn}
A subgroup $B$ of $G$ is called a \textbf{Borel subgroup} if it is a maximal Zariski closed, connected and solvable. 
\end{defn}
\begin{remk}
 All Borel subgroups of $G$ are conjugate in $G$.
\end{remk}
From now we shall assume that $G$ is a linear group. So we fix $n \in \mathbb{N}$ and an embedding $\iota : G \hookrightarrow GL_{n}$. We identify $G$ with $\iota(G)$.
\begin{defn}
An element $g\in G$ is called an \textbf{unipotent} element of $G$, if $g-I_{n}$ is nilpotent, where $I_{n}$ is the identity of $GL_{n} $ .
\end{defn}
\begin{defn}
Let $G$ be a connected algebraic group. The \textbf{radical} of $G$, denoted by
$R(G)$, is the maximal connected solvable normal subgroup of $G$. The \textbf{unipotent radical}
$R_u(G)$ of $G$ is the subgroup of unipotent elements of $R(G)$.
\end{defn}
We choose a Borel subgroup $\para{B}$ such that $T\subset \para{B}$, and denote its Lie algebra  by $\mathfrak{b}$.
Then 
there exists a closed connected normal subgroup of $\para{B}$ denoted by $U$ such that $\para{B}=TU$ is semi--direct product (see \cite{ParabolicSturctre}, section 6) . 
Therefore, since $T$ acts on $\para{B}$, it also acts on $\mathfrak{b}$.  Hence
by choosing $\para{B}$, we also get a choice of positive roots $\Phi_{G}^{+}$ defined by $\alpha \in \Phi_{G}^{+}$ iff $\mathfrak{g}_{\alpha} \subset \mathfrak{b}$. We define the negative roots to be $\Phi_{G}^{-}=\Phi_{G} \setminus \Phi_{G}^{+}$. 
\begin{defn}
A positive root $\alpha$ is called a \textbf{simple root} if it can not be written as the sum of two positive roots. We  denote the set of simple roots of $G$ by $\Delta_{G}$.
\end{defn}

For every root $\alpha \in \Phi_{G}$ we can associate  a homomorphism  $x_{\alpha} \: : \mathbb{G}_a \rightarrow G$ such that for every $c \in \mathbb{G}_a$ and $t \in T$ it holds $tx_{\alpha}(c)t^{-1}=x_{\alpha}(\alpha(t)c)$. Moreover, we can define a morphism $\phi_{\alpha} : SL_{2} \rightarrow G$ such that
$$\phi_{\alpha} \left(\begin{array}{rr}
1 & c \\
0 & 1
\end{array}\right) =x_{\alpha}(c), \quad \phi_{\alpha} \left(\begin{array}{rr}
1 & 0 \\
c & 1
\end{array}\right) =x_{-\alpha}(c).$$

Define $\check{\alpha} :\: \mathbb{G}_m \rightarrow T$ by $\check{\alpha}(c)=\phi_{\alpha}\left(\begin{array}{rr}
c & 0 \\
0 & \frac{1}{c}
\end{array}\right)$. The element $\check{\alpha} \in Y(T)$ and is called the coroot associated to the root $\alpha$. We denote by $\Phi_{G}^{\check{}}$ the set of all coroots of $G$.

By the above construction we associate for every algebraic split reductive group a quadruple $(X(T),\Phi_{G},Y(T),\Phi_{G}^{\check{}})$ with respect to the torus $T$, that is called a \textbf{root datum} . 

 Let $\Omega=\{ \fundamental{i}\}_{i=1}^{n}$ be elements of $\mathfrak{t}^{*}$ such that $\left \langle \fundamental{j},\check{\alpha_i}\right \rangle =\delta_{ij}$. These elements are called fundamental weights. Obviously, $X(T) \subset \operatorname{Span}_{\mathbb{Z}}\{\Omega\}$. Moreover, the quotient is finite. 
 
 For every root $\alpha \in \Phi_{G}$ we define the simple reflection $\weylelement{\alpha}\in Aut(X(T)) , \weylelement{\check{\alpha}}\in Aut(Y(T))$) to be 
$$\weylelement{\alpha}(x)= x- \left \langle \; x , \check{\alpha} \: \right \rangle \alpha \quad \weylelement{\check{\alpha}}(y)= y - \left \langle \; \alpha, y\: \right \rangle \check{\alpha} .$$
 
 Let $Norm_{G}(T)$ be the normalizer of $T$ in $G$. The \textbf{Weyl group} of $G$ is defined to be $\weyl{G}=: Norm_{G}(T) \slash T$.
 
 \begin{remk}
 Since $\weyl{G}$  acts on $T$ it also acts on $X(T)$. 
 \end{remk}
 
 \begin{prop} \cite{ParabolicSturctre}
 $\weyl{G} \simeq <w_{\alpha} \: | \: \alpha \in \Delta_{G}>$, and this is a finite group.
 \end{prop}

\subsection*{Parabolic subgroups}
Let $G,\para{B},T,\Delta_{G}$ be as above. For simplicity we denote by $\Delta=\Delta_{G}$.
\begin{defn}
A subgroup $P$ of $G$ is called a \textbf{parabolic subgroup} (resp. standard) if it contains a Borel subgroup (resp. $\para{B}$).
\end{defn}
\begin{prop}\cite{ParabolicSturctre}
Let $I \subset \Delta$, and let $\weyl{I}$ be the subgroup of $\weyl{G}$ generated by the subset $S_{I}=\{ \weylelement{\alpha} \: | \alpha \in I \}$. Denote by $\para{P}_{I}= \para{B}\weyl{I}\para{B}$. Then $\para{P}_{I}$ is a standard parabolic subgroup of $G$.
\end{prop}
\begin{prop}\cite{ParabolicSturctre}\label{corres}
The correspondence  $I \mapsto \para{P}_{I}$ defines a bijection between subsets of simple roots and standard parabolic subgroups of $G$.
\end{prop}

\begin{defn}
A parabolic subgroup $\para{P}$ will be called a \textbf{maximal parabolic subgroup} if the only subgroup $H$ of $G$ such that $\para{P} \subset H$ is $H=G$.
\end{defn}
\begin{remk}
According to Proposition  (\ref{corres}) a maximal parabolic subgroup $\para{P}$ corresponds to the set of simple roots $\Theta= \Delta \setminus \{\alpha_{i}\}$ for some $\alpha_{i} \in \Delta$. In that case we say that $\para{P}$ corresponds to $\alpha_{i}$ and we denote it by $\para{P}_{i}$.
\end{remk}

Let $\para{P}$ be a parabolic subgroup of $G$ corresponding to $I \subset \Delta$. The set $\Phi_{I}$, which is the set of all $\alpha \in \Phi$ that are integral linear combinations of elements of $I$, forms a root system, with the  Weyl group $\weyl{I}$. Furthermore, the set of roots of $\para{P}$ equals  to $\Phi^{+} \cup (\Phi^{-} \cap \Phi_{I})$. We denote by $N_{I}=R_{u}(\para{P})$. 
\begin{prop}\cite{ParabolicSturctre}
$N_{I}=<x_{\alpha}(r) \: | \: \alpha \in \Phi^{+} \setminus \Phi_{I}, \:\:  r \in F >$.
\end{prop}
Let $T_{I}=(\cap_{\alpha \in \Phi_{I}} ker \: \alpha)^{0}$, and $M_{I}=Z_{G}(T_{I})$. The group $M_{I}$ normalizes $N_{I}$, and it holds that $\para{P}=M_{I} N_{I}$ is a semi--direct product. This decomposition is known as \textbf{Levi decomposition} of $\para{P}$. The group $M_{I}$ is called the \textbf{Levi factor} of $\para{P}_{I}$.

Let $\para{P}=MN$ be a parabolic subgroup of $G$ with its Levi decomposition. We define $\delta_{\para{P}} \in X(T) \otimes \mathbb{Q}$ to be  $\delta_{\para{P}}=\sum_{\alpha \in \Phi^{+}_{G} \setminus \Phi_{I}} \alpha$ .

We denote by $X(M)$ (resp. $Y(M)$) the character (resp. cocharacter) group of $M$.

Let $\para{P}=MN$ be a parabolic subgroup of $G$.
Set
\begin{align*}
\mathfrak{a}_{M} =& \operatorname{Hom}(X(M),\mathbb{R}), \\
\operatorname{Re} \: \mathfrak{a}_{M}^{*} =& X(M) \otimes_{\mathbb{Z}} \mathbb{R}, \\
\mathfrak{a}_{M}^{*} =& X(M) \otimes_{\mathbb{Z}} \mathbb{C}.
\end{align*}
\begin{remk}
If $\para{P}=\para{B}$ we set $\mathfrak{a}=\mathfrak{a}_{T}, \quad Re \: \mathfrak{a}^{*}=Re \: \mathfrak{a}_{T}^{*}, \quad 
\mathfrak{a}^{*}=\mathfrak{a}_{T}^{*}$.
\end{remk}
We denote by $\para{F}=\sum_{\alpha_i \in \Delta_{G} } \mathbb{C} \alpha_i$ and by 
$Re \: \para{F}=\sum_{\alpha_i \in \Delta_{G} } \mathbb{R} \alpha_i $. Then it holds 
$\mathfrak{a}^{*}=\mathfrak{z}^{*} \oplus \para{F}$
where $\mathfrak{z}^{*}= \{ x \in \mathfrak{a}^{*} \: : \: \left \langle x,\check{\alpha} \right \rangle =0 \quad \forall \alpha \in \Delta_{G} \}$.
\begin{remk} \label{latticedim}
Let $\para{P}$ be a parabolic subgroup of $G$, corresponds to  the set of simple roots $\Theta=\Delta \setminus \{\alpha_{i_1}, \dots \alpha_{i_l}\} $. Then
$X(M) \simeq \mathbb{Z}^{l}$. Moreover, the vector space $\mathfrak{a}^{*}_{M}$ is spanned  by $\{\fundamental{i_1}, \dots \fundamental{i_l}\}$. 
\end{remk}

 \begin{defn}
 Two parabolic subgroup of $G$ are called \textbf{opposite} if their intersection
 is the Levi component of each of them.
 \end{defn}
\begin{defn}
We denote by $\bar{U}$ unipotent radical of the opposite Borel $\para{B}'$.
\end{defn}

\newpage 

\section{Representation theory of reductive groups over local fields}\label{Ch1:section2}
Let $G,T,\para{B},\Delta$ be as above. 
\subsection{Non--archimedean  fields}
Let $F$ be a non--archimedean local field.  Let $\para{O}$ be the ring of integers of $F$, and $q$ the cardinality of the residue field.
By abuse of notation we write $G$ as the group of $F$-- points of $G$.
 Let $K=G(\para{O})$, be a maximal open subgroup of $G$. We fix a Haar measure $\mu$ on $G$, such that $\mu(K)=1$.

\begin{defn}
A pair $(\pi,V)$ is called a representation of $G$ if $V$ is a $\mathbb{C}$ vector space and $\pi: \: G \rightarrow GL(V)$ is a group homomorphism.
\end{defn} 

\begin{defn}
A representation is called \textbf{irreducible} if it does not have a  non--trivial invariant subspaces. 
\end{defn}

\begin{defn}
The representation $(\pi,V)$ of $G$ is \textbf{smooth} if for any $v\in V$, the stabilizer $Stab_{G}(v)=\{g \in G \: | \: \pi(g)(v) = v\}$ of $v$ in $G$ is an open subgroup of $G$.
\end{defn}

\begin{defn}
A smooth representation $(\pi,V)$ of $G$ is called \textbf{admissible} if for any compact open subgroup $C$ of $G$ the space $V^{C} = \{v \in V \: | \: \pi(c)v=v \:\: \forall c \in C \}$ is finite dimensional. 
\end{defn}

\begin{defn}
Let $(\pi,V)$ be a smooth irreducible  representation of $G$. The representation $(\pi,V)$ is called \textbf{spherical} (unramified) if  there is a non-zero vector $v_0 \in V$ which is a $K$--fixed vector, i.e.
$\pi(k)v_0= v_0$ for all $k \in K$. A non--zero $v \in V^{K}$ is called a spherical vector.
\end{defn}
Example:  one dimensional representation $\chi : G(F) \rightarrow \mathbb{C}^{\times}$ is spherical if $\chi_{|_{G(\para{O})}}=1$. Such $\chi$ is called an \textbf{unramified character}. 
\begin{prop}[\cite{Bump},Proposition 4.6.2]
Let $(\pi,V)$ be a smooth irreducible representation of $G$. Then $\operatorname{dim} V^{K} \leq 1$.
\end{prop}
Let $\para{P}$ be a parabolic subgroup of $G$ corresponding to $I \subset \Delta$. We fix a left Haar measure $\mu$ on $\para{P}$. Since $\para{P}$ is a locally compact topological group $\mu$ is unique up to a positive  scalar multiple. Since the right translation on $\mu$ by an element $p \in \para{P}$ is also a left Haar measure $\mu_{p}$, there exists a function $\delta_{\para{P}}: \: \para{P} \rightarrow \mathbb{R}_{>0}$
called \textbf{modular character} such that $\mu_p=\delta_{\para{P}}(p)\mu$.

\begin{remk}
Let $\para{P}=MN$ be a parabolic subgroup with Levi decomposition. Hence it holds that $\mathfrak{p}= \mathfrak{m} \oplus \mathfrak{n}$, where $\mathfrak{p}, \mathfrak{m},\mathfrak{n}$ is the Lie algebra of $\para{P},M,N$. Then $$\delta_{\para{P}}= \sum_{\alpha : \mathfrak{g}_{\alpha} \subset \mathfrak{n}}\alpha.$$
\end{remk}
\begin{defn}
Let $H$ be a closed subgroup of $G$. Let $(\pi,V)$ be a smooth representation of $H$. We define the \textbf{induced representation}, denote by $Ind_{H}^{G}(\pi)$, to be the space of all functions $f : G \rightarrow V$ such that: 
\begin{enumerate}
\item
$f(hg)= \pi(h)f(g)$ for all $h\in H, g \in G$.
\item
There exists a open compact subgroup $K_1$ of $G$ such that $f(gk)=f(g)$ for all $k \in K_1$ .
\end{enumerate}
\end{defn}

\begin{prop}[\cite{Cass},Proposition 2.4.5]
\label{induction_stage_close}
Let $H_1 \subset H_2$ be  closed subgroups of $G$. Let $(\pi,V)$ be a smooth representation of $H_1$. Then $$Ind_{H_2}^{G} ( Ind_{H_1}^{H_2}(\pi)) \simeq Ind_{H_1}^{G} (\pi).$$
\end{prop}
Induction from parabolic subgroup, which are closed, is an important class of smooth representations.

\begin{defn}\label{ind_def}
Let $\para{P}$ be a parabolic subgroup of $G$ with Levi decomposition \\ $\para{P}=MN$. Let $(\pi,V)$ be a smooth representation of $M$. We define the \textbf{normalized parabolic induction}, denote by $Ind_{\para{P}}^{G}(\pi)$, to be the space of all functions \\ $f : G \rightarrow V$ such that: 
\begin{enumerate}
\item
$f(mng)=\delta_{\para{P}}^{\frac{1}{2}}(m) \pi(m)f(g)$ for all $m\in M, n \in N, g \in G$.
\item
There exists a open compact subgroup $K_1$ of $G$ such that $f(gk)=f(g)$ for all $k \in K_1$ .
\end{enumerate}
\end{defn}
\begin{remk}
In our thesis we will restrict ourself only to the case where $\pi$ is an unramified character of $M$.
\end{remk}
\begin{remk}\label{induction_stage_parabolic}
Let $\para{P}_1 \subset \para{P}_2$ are parabolic subgroups of $G$, with Levi decomposition $\para{P}_i=M_i\times N_i$. Let $\chi$ be representation of $M_1$. Then by proposition (\ref{induction_stage_close}) it holds that $Ind_{\para{P}_2}^{G}(Ind_{\para{P}_1}^{\para{P}_2}(\chi)) \simeq Ind_{\para{P}_1}^{G}(\chi)$.
\end{remk}
\begin{thm} \cite{Cass}
Let $(\pi,V)$ be a spherical representation of $G$. Then $(\pi,V)$ is a constituent of $Ind_{\para{B}}^{G} (\chi)$, where $\chi$ is an unramified character of $T$.
\end{thm}



Let $\chi\in X(T) \otimes \mathbb{C}$ be an unramified character of $T$, $w\in \weyl{G}$. We define, formally, the \textbf{local intertwining operator} 
$$M_{w}(\chi): Ind_{\para{B}}^{G}(\chi) \rightarrow Ind_{\para{B}}^{G}(\weylelement{}\chi)$$
$$(M_{w}(\chi))(f)(g)=\int_{U \cap w \bar{U}  w^{-1}\backslash U  }f(w^{-1}ug) du.$$

\begin{thm}\label{intertwiningProp}\cite{IntertwiningProp}
The local intertwining operator $M_{w}(\chi)$ converges absolutely and uniformly as $g$ varies in a compact set for $Re \: \left \langle  \chi,\check{a} \right \rangle \gg 0$ for every $\alpha \in \Phi^{+}_{G}$. Moreover, it  admits an analytic continuation as meromorphic function of $\chi \in X(T) \otimes \mathbb{C}$.
\end{thm}
\begin{prop}\cite{Ike1} \label{operatorInstage}
Let $\weylelement{}=\weylelement{1},\weylelement{2}
\in \weyl{G}$, such that $l(w)=l(w_1)+l(w_2)$ then for every $\chi \in X(T) \otimes \mathbb{C}$ it holds that
$M_{\weylelement{}}(\chi)=M_{\weylelement{1}}(\weylelement{2}\chi) \circ M_{\weylelement{2}}(\chi)$.
\end{prop}

\begin{thm} \label{gid}\cite{Ike1}
Let $f^{0} \in Ind_{\para{B}}^{G}(\chi)$ be the normalized spherical vector. Then for every $w\in \weyl{G}$ it holds that 
$M_{w}(\chi)f^{0}=C_{w}(\chi)f^{0}_{\weylelement{}\chi}$, where $$C_{w}(\chi)= \prod_{\tiny \begin{aligned}
\alpha\in \Phi^{+}_{G}\\ 
 w\alpha <0 \\
\end{aligned}
}{\frac{\zeta(\left \langle \chi,\check{\alpha}\right \rangle)}{\zeta(\left \langle \chi,\check{\alpha}\right \rangle+1)}},$$
where $\zeta$ is the local $\zeta$ function.
\end{thm}
As a corollary we get 
\begin{cor} \label{gid_cor}
Let $\weylelement{}=\weylelement{1}\weylelement{2}$ be a reduced word. Then 
for any $\chi \in X(T)$ it holds that 
$$C_{\weylelement{}}(\chi)=C_{\weylelement{1}}(\weylelement{2}\chi)C_{\weylelement{2}}(\chi).$$
\end{cor}
\newpage
\begin{prop} \label{trivial_indced}
$ $
\begin{enumerate}
\item
 The induced representation $Ind_{\para{B}}^{G}(\delta_{\para{B}}^{-\frac{1}{2}})$ has a unique irreducible sub-representation $\pi$. Moreover, $\pi$ is the trivial representation.
\item
 The induced representation $Ind_{\para{B}}^{G}(\delta_{\para{B}}^{\frac{1}{2}})$ has a unique irreducible quotient $\pi$. Moreover, $\pi$ is the trivial representation. In this case  the trivial representation is the image of $M_{w_0}$ where $w_0 \in \weyl{G}$ is the longest element.
\end{enumerate}
\end{prop}
\begin{remk}
Hence, using Proposition (\ref{operatorInstage}) we deduce that  the trivial representation is a  quotient of 
$I_{\para{P}}(\delta_{\para{P}}^{\frac{1}{2}})$.
\end{remk}
\begin{remk}
Let $\para{P}=MN$ be a parabolic subgroup of $G$. Let $\tilde{\chi}$ be an unramified character of $M$. Then,
$$Ind_{\para{P}}^{G} (Ind_{\para{B}}^{\para{P}}(\delta_{\para{B}\cap M}^{-\frac{1}{2}}) \tilde{\chi})=Ind_{\para{B}}^{G}(\tilde{\chi} \otimes \delta_{\para{P}}^{\frac{1}{2}} \otimes \delta_{\para{B}}^{-\frac{1}{2}}). $$
Let $\chi=\tilde{\chi} \otimes \delta_{\para{P}}^{\frac{1}{2}} \otimes \delta_{\para{B}}^{-\frac{1}{2}} $
 it holds that $Ind_{\para{P}}^{G}(\tilde{\chi}) \hookrightarrow Ind_{\para{B}}^{G}(\chi)$. Thus, it is possible to restrict $M_{\weylelement{}}(\chi)$ to $Ind_{\para{P}}^{G}(\tilde{\chi})$. 
\end{remk}
\begin{defn}
Let $\para{P}=MN$ be a parabolic subgroup of $G$. Let $(\pi,V)$ be an admissible  representation of $G$. Denote by $V(N)=\operatorname{Span} \{\pi(n)v -v\: : n\in N, \quad v \in V  \}$. 
Let $V_{N}=V \slash V(N)$ and $(\pi_N,V_N)$  the corresponding $M$-representation.  This is called the
\textbf{Jacquet module} of $V$ (with respect to $\para{P}$), and $V \mapsto V_N$ is called the \textbf{Jacquet functor} (denoted by $r_{M,G}(\pi)$).
\end{defn}

Let $\para{P}=MN$ be a parabolic subgroup of $G$, corresponding to $I= \Delta_{G} \setminus J $, where  $J=\{\alpha_{i_1} \dots \alpha_{i_l}\} \subset \Delta $. We set 
\begin{align*}
(\mathfrak{a}_{M})^{*}_{+}=&\{
x \in Re \:\mathfrak{a}_{M}^{*} \: : \left \langle  x ,\check{\alpha_i} \right \rangle >0 \: \: \forall \alpha_i  \in  J
 \},
 \\
 ^{+}\bar{\mathfrak{a}}_{M}^{*}=&\{
 x \in Re \:\mathfrak{a}_{M}^{*} \: : x=\sum_{\alpha \in J}{c_{\alpha}\alpha} \quad 
 c_{\alpha} \geq 0
  \}.
 \end{align*}
 
 \begin{defn}  Let $\pi$ be an irreducible admissible representation of $G$. Let 
 \begin{align*}
\para{M}_{min}&=\{ \text{$L$ standard Levi} \: : \text{$r_{L,G}(\pi)\neq 0$ but $r_{N,G}(\pi)=0$ for all $N \subset L$} \}, \\
\para{E}xp(\pi)&=\left\{ 
  \begin{aligned}
& x \in Re \:\mathfrak{a}^{*} \: : \text{ $x \otimes \rho \leq r_{L,G}(\pi)$ for some  representation $\rho$ of  $L$}\\
 &\quad \quad \quad \quad \quad \quad \quad \text{ with unitary central character , $L \in \para{M}_{min}(\pi)$} 
  \end{aligned}
\right\} 
 \end{align*}

 \end{defn}
 \newpage
\begin{defn}\cite{ban2008}
Let $(\pi,V) $ be an irreducible representation of $G$,  with a unitary central character. Then 
\begin{enumerate}
\item
$\pi$ is called \textbf{square integrable} iff $\para{E}xp(\pi) \subset \: \mathfrak{a}^{*}_{+}$.
\item
$\pi$ is called \textbf{tempered}  iff $\para{E}xp(\pi) \subset \: ^{+}\bar{\mathfrak{a}}_{}^{*}$.
\end{enumerate} 
\end{defn}

\begin{defn}
Let $\para{P}=MN$  be a parabolic subgroup of $G$ corresponds to $\Delta \setminus J$ where $J \subset \Delta$. Let $\lambda \in (\mathfrak{a}_{M})^{*}_{+}$  and $\sigma$ be an irreducible tempered representation of $M$.   
The triple $(\para{P},\sigma,\lambda)$ will be called \textbf{Langlands data}.
\end{defn}
\begin{thm}\cite{Konno1} 
 \label{langlands_data} 
Let $(\para{P},\sigma,\lambda)$ be a Langlands data then $Ind_{\para{P}}^{G}(\sigma \otimes \lambda)$ has a unique irreducible quotient denote by $\para{J}(\para{P},\sigma,\lambda)$. Conversely, if $\pi$ is an irreducible admissible representation of $G$, then there exists a unique $(\para{P},\sigma,\lambda)$ such that $\pi \simeq   \para{J}(\para{P},\sigma,\lambda)$.
\end{thm}
\newpage
\section{Automorphic forms}\label{Ch1:section3}
In this section $F$ is a number field and $\adele{A}$ its ring of adeles. Let $\places{P}$ be the set of places of $F$. For each $\nu \in \places{P}$ we denote by $F_{\nu}$ the completion of $F$ with respect to $\nu$. Let $G$ be an algebraic reductive split group over $F$. For every $\nu \in \places{P}$, we fix a maximal compact subgroup $K_{\nu} \subset G(F_{\nu})$. Denote by $K=\prod_{\nu \in \places{P}}K_{\nu}$, the maximal compact subgroup of $G(\adele{A})$. We fix a Haar measure $\mu$ on $G(\adele{A})$ such that $\mu(K)=1$. 

In our thesis we restrict ourself to very special representations called the automorphic representations. The definition can be read in \cite{Auto_def}. The automorphic representations are realized in the space of automorphic functions, in the literature they are known as automorphic forms. The definition of such functions is in \cite{EisenstienSeries} ($I.2.17$). The space of automorphic forms of $G$ is denote by $\automorphicspace{G}{F}$.
\begin{defn}
Let $\para{P}(\adele{A})$ be a parabolic subgroup of $G(\adele{A})$ with Levi decomposition \\ $\para{P}(\adele{A})=M(\adele{A})N(\adele{A})$. Let $(\pi,V)$ be an automorphic representation of $M(\adele{A})$. We define the \textbf{normalized parabolic induction}, denote by $I_{\para{P}}(\pi)=Ind_{\para{P(\adele{A})}}^{G(\adele{A})}(\pi)$, to be the space of all functions $f : G(\adele{A}) \rightarrow V$ such that: 
\begin{enumerate}
\item
$f(mng)=\delta_{\para{P}}^{\frac{1}{2}}(m) \pi(m)f(g)$ for all $m\in M(\adele{A}), n \in N(\adele{A}), g \in G(\adele{A})$.
\item
$f$ is right $K$ --finite.
\end{enumerate}
\end{defn}



\begin{remk}
Let $\para{P}$ be a parabolic subgroup of $G$ correspond to $I=\Delta_{G} \setminus \{\alpha_{i_1},\dots , \alpha_{i_n}\}$. Then by remark (\ref{latticedim}), the set $X^{un}(M)$ of unramified characters of $M$ is isomorphic $\mathbb{C}^{n}$ by
$(s_1, \dots ,s_n) \longleftrightarrow \sum_{l=1}^{n} s_{i_l}\fundamental{i_l}$.
\end{remk}
\begin{remk}
When $\para{P}$ is a maximal parabolic subgroup we define $I_{\para{P}}(s) = Ind_{\para{P}}^{G} (\delta_{\para{P}}^{s})$.
\end{remk}
\begin{defn}
Let $\para{B}(\adele{A})=T(\adele{A})U(\adele{A})$ be a Borel subgroup of $G(\adele{A})$. Let $\chi$ be an automorphic character of $T(\adele{A})$.
For every $\weylelement{}\in \weyl{G}$ we define the \textbf{global intertwining operator} to be:
$$M_{\weylelement{}}(\chi): \: I_{\para{B}}(\chi) \rightarrow I_{\para{B}}(\weylelement{}\chi)$$   
$$(M_{\weylelement{}}(\chi)f)(g)=\int_{\adelegroup{U}{A}\cap \weylelement{}\overline{\adelegroup{U}{A}}\weylelement{}^{-1}\backslash \adelegroup{U}{A} }{f(\weylelement{}^{-1}ug)du}=\otimes_{\nu}(M_{\weylelement{}}(\chi_{\nu})f_{\nu})(g_{\nu}).$$
\end{defn}

\begin{prop} \cite{EisenstienSeries}
The global intertwining operator converges absolutely and uniformly as $g$ varies in a compact set if $Re \: \left \langle  \:\chi,\check{\alpha}\right \rangle \gg 0$ for every $\alpha \in \Phi_{G}^{+}$ such that $w\alpha<0$. Moreover it admits a meromorphic continuation for all $\chi \in \mathfrak{a}^{*} $.      
\end{prop}
Let $\para{P}=MN$ be a parabolic subgroup of $G$. By Iwasawa decomposition \\ $\adelegroup{G}{A}=\adelegroup{P}{A}K$  we conclude that every $f \in I_P(\chi _{\underbar{s}})$
is determined  by its restriction to $K$.
Moreover, every smooth function $f \in \{f :K \rightarrow C : f(mk)=f(k) \quad \forall m\in \adelegroup{M}{A}\cap K \}$ can be extended uniquely to $f_{\underbar{s}} \in I_P(\chi_{\underbar{s}}) \quad \forall \underbar{s} \in \mathbb{C}^{n} $. This family  $\{f_{\underbar{s}} : \underbar{s} \in \mathbb {C}^{n}\}$ is called a \textbf{flat section}.

\begin{defn}
Let $\para{P}$ be a parabolic subgroup of $G$. Given a flat section $f_{\underbar{s}} \in I_{\para{P}}(\chi_{\underbar{s}})$
  such that $f_{{\underbar{s}}_ {|_\adelegroup{K}{A}}}=f$
 we define the \textbf{Eisenstein series} on $\adelegroup{G}{A}$ to be 
$$E_{\para{P}}(f,g,\underbar{s})=\sum_{\gamma \in \para{P}(F)\backslash G(F)}f_{\underbar{s}}(\gamma g).$$
Whenever this series converges it defines an automorphic form.
\end{defn}
Recall the properties of Eisenstein series
\begin{thm}\cite{WeeTeck}
[Eisenstein series properties]\label{EisensteinProperties}
$ $
\begin{enumerate}
\item
The Eisenstein series is left $G(F)$--invariant when it is defined.
\item
There exists an open cone of $ \mathfrak{a}_M^{*}$ such that the series  converges absolutely and uniformly when $g$ varies in a compact set.
\item
For any $f$ and $g$ the function
$\underbar{s} \mapsto E_{\para{P}}(f,g,\underbar{s})$
admits a meromorphic continuation to all of $\mathbb{C}^{n}$.
\item
At a point $s_0$ where $E_{\para{P}}(f,g,s_0)$ is holomorphic for all $f$ and $g$, the function
$E_{\para{P}}(f,s_0,\cdot )$ is an automorphic form on $G$. Also, the map
$f \mapsto E_{\para{P}}(f,s_0,-)$
is a $\adelegroup{G}{A}$--equivariant map of $I_{\para{P}}(\chi_{\underbar{s}_0})$ to $\automorphicspace{G}{F}$.
\item 
Let $\weylelement{}\in \weyl{G}$. Then
$$E_{\para{B}}(f,g,\underbar{s})=E_{\para{B}}(M_{\weylelement{}}(\underbar{s})f,g,\weylelement{}\underbar{s}).$$
\end{enumerate}
\end{thm}

\begin{defn}
Let $\para{P}$ be a maximal parabolic subgroup of $G$. We say that $s_0 \in \mathbb{C}$ is a \textbf{pole} of order $d$ of the Eisenstein series if $-d=inf_{f,g} \: ord_{s=s_0}\{E_{\para{P}}(f,g,s)\}$.
\end{defn}
At the point $s_0$ where $E_{\para{P}}$ has a pole we consider the Laurent expansion of $E_{\para{P}}(f,g,s)$ at $s=s_0$
$$E_{\para{P}}(f,g,s)=\frac{\leadingterm{-d}^{\para{P}}(f,g,s_0)}{(s-s_0)^d}+\frac{\leadingterm{-d+1}^{\para{P}}(f,g,s_0)}{(s-s_0)^{d-1}}+\dots$$
where $d$ is the order of the pole attained at $s_0$ and  $\leadingterm{-d}^{\para{P}}(f,g,s_0)$ is the first non-zero term for at least one holomorphic section $f_{s_0}\in I_{\para{P}}(s_0)$.
\begin{prop} \cite{WeeTeck}
The function $f \mapsto \leadingterm{-d}^{\para{P}}(f,g,s_0)$ is a $\adelegroup{G}{A}$--equivariant map of $I_{\para{P}}(s_0)$ to $\automorphicspace{G}{F}$.
\end{prop}
\begin{thm} \cite{WeeTeck}
The poles of the Eisenstein series are the same as the poles of its constant term along the Borel subgroup.
\end{thm}

\chapter{Algorithm: Compute pole order} 

\label{ch4:algo} 

\lhead{Chapter \ref{ch4:algo}. \emph{Algorithm}} 


In this chapter we will describe the algorithm to determine  the poles of the Eisenstein series $E_{\para{P}}(f^{0},g,s)$ and compute their orders.\\ The chapter outline is as follows:\\
The chapter starts with few definitions that introduce the language for describing the algorithm. The chapter continues with the description of the algorithm itself.
Finally we will prove the correctness of the algorithm.

We fix a maximal parabolic subgroup $P$. 
We denote by  $W(\para{P},G)$ the set of all shortest representative in $\weyl{G} \slash \weyl{\para{P}}$. Set $\chi_{\para{P},s}=\delta_{\para{P}}^{s+\frac{1}{2}}\otimes\delta_{\para{B}}^{-\frac{1}{2}}$ since there  is no risk of confusion, we  shall omit the subscript  $\para{P}$.
\begin{defn}
Let $s_0 \in \mathbb{C}$, $s_0$ is called  \textbf{potential pole} of $E_{\para{P}}(f^{0},s,g)$ if $\operatorname{Re} s_0>0$ and there exists $\weylelement{} \in W(\para{P},G)$ such that $C_{\weylelement{}}(s)$ (as defined at \ref{gid})  admits a pole at $s=s_0$.
\end{defn}
\begin{defn}
For a potential pole $s_0$ we said that  $\weylelement{},u \in W(\para{P},G)$ are \textbf{equivalent} if $u.\chi_{s_0}=\weylelement{}.\chi_{s_0}$. We denote it by $u
 \sim_{s_0} {\weylelement{}}$. The equivalence class will be denoted by $[u]_{s_0}$.
\end{defn}
\begin{defn}
For a potential pole $s_0$ and equivalence class $[u]_{s_0}$ we define $$M_{u}^{\#}(s)=\sum_{\weylelement{}\in [u]_{s_0}}C_{\weylelement{}}(s)$$ 
and by 
$N_{u}^{\#}(s)$ its numerator.
\end{defn}
\begin{defn}
For a potential pole $s_0$, an equivalence class $[u]_{s_0}$ is called \textbf{square integrable} if $\operatorname{Re}(u.\chi_{s_0})$ can be written as sum of simple roots with  negative coefficients.
\end{defn}

After we put the language we are ready to describe our algorithm.

\subsection*{Pseudo code}
\begin{enumerate}
\item\label{algo::first}
Find out the set of all shortest representative in $\weyl{G} \slash \weyl{\para{P}}$ (denoted by $W(\para{P},G)$).
\item\label{algo::sec}
Find out the set of potential poles.
(denoted by $\eta(\para{P})$).
\item
For each potential pole $s_0 \in \eta(\para{P})$ do:
\begin{enumerate}
\item\label{algo::thirda}
Divide $W(\para{P},G)$ into equivalence classes.
\item
For each equivalence class do:
\begin{enumerate}
\item
\label{algo::thirdb_1}
 Define $M_{\weylelement{}}^{\#}$
\item
\label{algo::thirdb_2}
 Write Laurent expansion of $N_{\weylelement{}}^{\#}$ around $s=s_0$.
we denote by $d_{0}$ the formal order of 
$N_{\weylelement{}}^{\#}$ at $s_0$.
\item
\label{algo::thirdb_3}
Find out if the equivalence class is square integrable or not
(denote by $square$).
\end{enumerate}
\item\label{algo::thirdc}
$d= \operatorname{Max}{d_0}$
\item\label{algo::third_d}
$L2= \bigwedge_{(d_0,square)| d=d_0} square$.
\item\label{algo::thirde}
 Deduce that the Eisenstein series admits at  most a pole of order $d$ at $s=s_0$ and its leading term is square integrable (or not ) according to the value of $L2$.
\end{enumerate}
\end{enumerate}
\newpage
\section*{Correctness}
Recall that the poles of the degenerate Eisenstein series coincide with the poles of  its constant term along any parabolic. We shall compute the constant term of the Eisenstein series along the Borel subgroup. 
By its nature, the problem requires a lot  of computation (step \ref{algo::thirdb_2}) thus we use the computer.

We denote by $W(G,\para{P})$ the set of the shortest representatives in $\weyl{G} \slash \weyl{\para{P}}$. 

\begin{prop} \label{Constant_term_along_borel}
 Let $\para{B}=TU$ be the Borel subgroup of $G$. Then : 
$$\int_{U(F) \backslash U(\adele{A})} E_{\para{P}}(f^{0},s,ug)du= \sum_{\weylelement{} \in W(G,\para{P}) }M_{\weylelement{}}(s)f^{0}_{s}$$
\end{prop}
\begin{proof}
The Eisenstein series admits a meromorphic continuation to all $\mathbb{C}$, hence it is enough to prove this proposition only when the Eisenstein series converges absolutely and uniformly.   We denote by $U^{\gamma}=\gamma^{-1}\para{P}\gamma \cap U$ 
\begin{align*}
\int_{U(F) \backslash U(\adele{A})} E_{\para{P}}(f^{0},s,ug)du&= \int_{U(F) \backslash U(\adele{A})} \sum_{\gamma \in \para{P}(F) \backslash G(F)}f^{0}_{s}(\gamma ug)du  \\
&= \int_{U(F) \backslash U(\adele{A})}  
\sum_{\gamma \in \para{P}(F)  \backslash G(F) \slash \para{B}(F)} \: \: \: \:
\sum_{\delta \in \gamma^{-1}\para{P}(F)\gamma \cap \para{B}(F) \backslash \para{B}(F)}
f^{0}_{s}(\gamma \delta ug)du
\\
\\
&= \int_{U(F) \backslash U(\adele{A})}  
\sum_{\gamma \in \para{P}(F)  \backslash G(F) \slash \para{B}(F)} \: \: \: \:
\sum_{\delta \in 
U^{\gamma}(F) \backslash U(F)}
f^{0}_{s}(\gamma \delta ug)du
\\
&=\sum_{\gamma \in \para{P}(F) \backslash G(F)} \int_{U(F) \backslash U(\adele{A})} \sum_{\delta \in U^{\gamma}(F) \backslash U(F)}
f^{0}_{s}(\gamma \delta ug)du \\
\\
&=\sum_{\gamma \in \para{P}(F) \backslash G(F)} \int_{U^{\gamma}(F) \backslash U(\adele{A})}
f^{0}_{s}(\gamma ug)du 
\\
&=\sum_{\gamma \in \para{P}(F) \backslash G(F)} \int_{U^{\gamma}(\adele{A}) \backslash U(\adele{A})} \int_{U^{\gamma}(F) \backslash U^{\gamma}(\adele{A})}
f^{0}_{s}(\gamma u_1u_2g)du_1du_2
\\
&\overset{\#1}{=}\sum_{\gamma \in \para{P}(F) \backslash G(F)} \int_{U^{\gamma}(\adele{A}) \backslash U(\adele{A})}f^{0}_{s}(\gamma ug)du\\
&\overset{\#2}{=}
\sum_{\weylelement{} \in W(G,\para{P})}M_{\weylelement{}}(s)f^{0}_{s}(g)\\
&\overset{(\ref{gid})}{=}
\sum_{\weylelement{} \in W(G,\para{P})} C_{\weylelement{}}(s)f^{0}_{\weylelement{}s}(g).
\end{align*}
The equality ($\#1$) is due to left invariant properties of $f^{0}$ and taking the measure 
of $F\backslash \adele{A}$ to be one.
The equality ($\#2$) is the definition of $M_{\weylelement{}} $.
\end{proof}
From now on, we assume that $w\in W(G,\para{P})$.
Therefore the potential poles are the points where $C_{\weylelement{}}(s)$ has a pole for some $w\in W(G,\para{P})$.
Recall that $C_{\weylelement{}}(s)$ is described by a product of quotients of zeta functions (Theorem \ref{gid}). 
Let $f^{0}_{s} \in I_{\para{P}}(s)$. Since $I_{\para{P}}(s) \subset I_\para{B}(\chi_s)$
where $\chi_s= \delta_{\para{P}}^{s+\frac{1}{2}} \otimes \delta_{\para{B}}^{-\frac{1}{2}}$.
Then  $$C_{\weylelement{}}(s)= \prod_{\tiny \begin{aligned}
\alpha\in \Phi^{+}_{G}\\ 
 w\alpha <0 \\
\end{aligned}}
 \frac{ {\zeta(\left \langle  \chi_s,\check{\alpha} \right \rangle) }} {{\zeta(\left \langle  \chi_s,\check{\alpha} \right \rangle +1) }} \quad . 
$$
After performing cancellation we obtained a reduced form 
$$C_{\weylelement{}}(s)= \prod
 \frac{ {\zeta(p_i(s))}} {{\zeta(q_i(s)) }} 
$$
where $p_{i}(s)=\left \langle \chi_s,\check{\alpha}_{i} \right \rangle$  and 
$q_{i}(s)=\left \langle \chi_s,\check{\beta}_{i} \right \rangle+1$
for some positive roots $\alpha_i,\beta_i$.
 Hence the pole of each term  occurs either when the numerator has a pole (when the term $p_{i}(s)=1,0$) or when denominator is zero. 
\begin{asmp}\label{ASUM}
Let $s\in \mathbb{C}$ such that $\operatorname{Re} s>0$ then:\\
 For every  $\weylelement{} \in W(G,\para{P})$ 
the denominator of the reduced term of $C_{\weylelement{}}(s_0)$ is holomorphic and non zero (every $\operatorname{Re}q_i(s) >1$). In other words the potential  poles of $C_{\weylelement{}}(\chi_s)$ are coming  only from its numerator.
\end{asmp} 
\begin{remk}
This assumption implies that 
$$\eta(\para{P})\subset
\{s \in \mathbb{C} \: : \: \operatorname{Re} s>0 \text { and }  \exists \alpha \in N_{\para{P}} \: \inner{\chi_s,\check{\alpha}}=0,1 \}.$$
Moreover, all potential poles are reals.
\end{remk}
 \begin{obv}\label{ZETADOM}
 Let $G$ be a  group of rank $\leq 8$. Let  $\para{P}$ be a maximal parabolic subgroup of $G$. For every  $\weylelement{} \in W(G,\para{P})$ the denominator of the reduced term of $C_{\weylelement{}}(s)$ for $\operatorname{Re} s >0$ is holomorphic and non zero (every $q_i(s) >1$). In other words,  Assumption \ref{ASUM} holds.
 \end{obv}

In order to make the algorithm more efficient,  we take 
$$\eta(\para{P})=\{s \in \mathbb{C} \: : \: \operatorname{Re} s>0 \text { and }  \exists \alpha \in N_{\para{P}} \: \inner{\chi_s,\check{\alpha}}=0,1 \}$$
to be the set in step 2.
Now we elaborate on step 3. Fix $s_0 \in  \eta(\para{P})$.
\begin{enumerate}
\item
Reorganize the $\sum_{\weylelement{} \in
W(G,\para{P}) }M_{\weylelement{}}(s)f^{0}$ as follows:
\begin{align*}
\sum_{\weylelement{} \in W(G,\para{P}) }M_{\weylelement{}}(s)f^{0}=& 
\sum_{\weylelement{} \in W(G,\para{P})} \frac{1}{|[\weylelement{}]_{s_0}|} \sum_{u \in [\weylelement{}]_{s_0}} M_{u}(s)f^{0} \\
=& \sum_{\weylelement{} \in W(G,\para{P})} \frac{1}{|[\weylelement{}]_{s_0}|}  M_{w}^{\#}(s)f^{0}
\end{align*}
where $M^{\#}_{\weylelement{}}(s)f^{0}=\left(\sum_{u \in [\weylelement{}]_{s_0}} M_{u}(s)\right)f^{0}$.
\end{enumerate}
\begin{prop}
Let $s_0 \in \eta(\para{P})$.
 The order of the pole of $E_{\para{P}}(f^{0},s,g)$ at $s=s_0$ is bounded by 
$\operatorname{max}_{\weylelement{} \in W(G,\para{P}) \slash \sim_{s_0}}\{ d \: :\: d=order_{s=s_0}M^{\#}_{\weylelement{}}\}$
\end{prop}
\begin{proof}
If $M_{\weylelement{}}^{\#}$ admits a pole of order $d$ it can not be canceled. The reason is 
$$
(M_{w}^{\#}(s)f^{0})(t)=(w\chi_s)(t)(M_{w}^{\#}(s)f^{0})(1) \neq (u\chi_s)(t)(M_{w}^{\#}(s)f^{0})(1)=(M_{u}^{\#}(s)f^{0})(t)
$$
 here $t$ is arbitrary element of the torus  and $(u\chi_{s_0})\neq (w\chi_{s_0})$ for $u \not \sim_{s_0} w$ .
\end{proof}
\begin{remk}
This is only an upper bound since it may happened that the leading coefficient is sum of zeta values and we do not know if it is zero or not.
\end{remk}
\begin{remk}
Iwasawa decomposition $g=tuk$, implies that  the order of the pole of $(M_{u}^{\#}(s)f^{0})(g)$ at $s=s_0$ is the same as the order of the pole of $(M_{u}^{\#}(s)f^{0})(1)$ at $s=s_0$. 
\end{remk}
After we reorganize the sum, for every equivalence class  $u$, we determine the order of the pole  of 
$M^{\#}_u$ at $s=s_0$. This is done as follows:\\
For each $w$ we put 
$$m= \: \operatorname{max}_{u \in W(\chi_s,w)} \{n : M_u \text{ admits a pole of order } n \text{ at $s=s_0$} \}.$$
Notice that $M_{\weylelement{}}^{\#}(s)$ admits a pole of order at most $m$.
Recall that :
 $$(M_{\weylelement{}}^{\#}(s)f^{0})(1)=\sum_{
 u \in W(\chi_{s_0},\weylelement{})
 } C_{u}(s)=\sum_{u \in W(\chi_{s_0},\weylelement{})}\prod \frac{\zeta(p_{i,u}(s))}{\zeta(q_{i,u}(s))}.$$ 
sum all together and do common denominator. By our observation the denominator is holomorphic and non-zero for every $s$ such that $\operatorname{Re} s>0$ in particular for $s=s_0$. Hence we may ignore it.
\\
Therefore the numerator is of the shape
$\sum_{i} \prod_{l} \zeta(p_{il}(s))$.
For each term in the sum, $\prod_{l} \zeta(p_{il}(s))$, we expand every factor in the product by the following rule.
\\
 If  $p_{il}(s_0)>\frac{1}{2}$, we write the Laurent expansion of $\zeta(p_{il}(s_0))$ around $s_0$.
 If $p_{il}(s_0)\leq \frac{1}{2}$ we use the function equation $\zeta(s)=\zeta(1-s)$ and write the Laurent expansion of $\zeta(1-p_{il}(s))$ around $s_0$. 
\begin{remk}\label{expan}
To check possible cancellation it is enough to write the Laurent expansion up to order $m+1$.
\end{remk}
\begin{remk}
Since $\zeta(s)=\zeta(1-s)$ for $s=\frac{1}{2}$ it holds that  the $(2n+1)^{th} $ derivative of $\zeta(s)$ is zero for every $n \in \mathbb{N}$. Hence,
around $\frac{1}{2}$ 
$$\zeta(s)= \sum_{j=0}^{\infty} a_{2j}(s-\frac{1}{2})^{2j}. \quad 
a_{2j}= \frac{\zeta^{(2j)}(\frac{1}{2}) }{2j!}
$$
In other words all the odd derivatives of $\zeta(s)$ at $s=\frac{1}{2}$ are vanish.
\end{remk}

Summing all the terms, we find  the order of the pole at $s_0$ for this $N _{u}^{\#}(s)$ and also for $M _{u}^{\#}(s)$.
The order of the pole at $s_0$ is bounded by the maximum of the orders of the pole of the all $M_{u}^{\#}(s)$ at $s=s_0$. 

This algorithm also  gives an answer to the question whenever the residual representation (i.e the leading term $\leadingterm{-m}^{\para{P}}(f^{0},s_0,g)$)  is square integrable or not. This is by applying 
Langlands criterion for $L^{2}$ . We recall this criterion
\begin{thm}\cite{EisenstienSeries}[Lemma  I.4.11]
\label{square_init}
Let $\phi$ be an automorphic form. Let $\para{P}=M_{\para{P}}N_{\para{P}}$ standard parabolic subgroup of $G$. Denote be $\Pi_{0}(M_{\para{P}},\phi)$ the cuspidal support of $\phi$ along $\para{P}$. Then for $\phi$ to be square integrable, it is necessary  and sufficient that for all $\para{P}$  and all $\pi \in\Pi_{0}(M_{\para{P}},\phi)$
the character $\operatorname{Re} \pi$ can be written in the form
$$\operatorname{Re}\pi =\sum_{\alpha \in \Delta_{\para{P}}}x_{\alpha}\alpha$$ with coefficients $x_\alpha \in \mathbb{R} , x_{\alpha} <0$.
\end{thm}
\begin{remk}
Since we are in the degenerate case, the cuspidal support of the leading term is only the Borel subgroup. Hence, this criterion can be written as follows: 
\\
Assume that $E_{\para{P}}(f^{0},g,s)$ admits a pole of order $d$ at $s_0$. Let $$Y=\{ \weylelement{}
\in W(G,\para{P}) \: : \: M_{\weylelement{}}^{\#} \text{ contributes a pole of order d at $s_0$}.
\}$$
Then  $\leadingterm{-d}^{\para{P}}(f^{0},s_0,g) \in L^{2}(G(F) \backslash G(\adele{A}))$  if  for every $\weylelement{} \in Y$  it  holds 
$$ \weylelement{}\chi_{s_0}= \sum_{\alpha \in \Delta_{G}}x_{\alpha}\alpha $$ with negative coefficients.
\end{remk}
\begin{remk}
In order to show that the leading term is square integrable, we have to show that all the equivalence classes that may contribute a pole of order $d$ and are non square integrable contribute a pole of order at most $d-1$. 
\end{remk}

\begin{obv}
If $F=\mathbb{Q}$ (this assumption can be lifted if we assume that several values of zeta function are non zero), and we are in the case that the algorithm found 
that $\operatorname{order}_{s=s_0}E_{\para{P}}(f^{0},s,g)\leq d$ where $d$ is as in step \ref{algo::thirdc} then there exists an equivalence class $[u]_{s_0}$ such that $\operatorname{order}_{s=s_0} M_{u}^{\#}(s)=d$.
 Moreover, if the leading term is  not square integrable then there exists an equivalence class such that contributes a pole of order exactly $d$ and that is not square integrable.
\end{obv}

\chapter{Poles of degenerate Eisenstein series} 

\label{ch2:Poles} 

\lhead{Chapter \ref{ch2:Poles}. \emph{Results}} 


Let $G$ be a split simply connected group of exceptional type $G_2,F_4,E_6$ or $E_7$.  
 Using the algorithm described in the last chapter we determine the poles of degenerate Eisenstein series 
$E_{\para{P}}(f^{0},g,s)$ for $Re \: s>0$  associated to various maximal parabolic subgroups,  
compute their orders, and determine  the square integrable ones. For $G=G_2$ our results agree 
with the results obtained in \cite{G2_Ginz}. 

In \cite{op} it was shown 
 that there exists a bijection between distinguished unipotent orbits of  the dual group
$^{L}G$ 
and the spherical residual representations that are square integrable with the cuspidal support $[T,1]$.

In the case where the residual representation at $s_0$ is square integrable we will determine the distinguished
 orbit of  $^{L}G$  related to it. Whenever several residual representations correspond to 
one orbit there will be identities between the leading terms of the Eisenstein series 
as will be shown in Chapter \ref{ch3:ids}.
If $t(s_0)$ is the representative in the dominant chamber of the Satake conjuagacy class
of $I_{\para{P}}(s_0)$ then $t(s_0)^2$ is the weighted Dynkin diagram of the 
distinguished unipotent orbit. For example
the principal unipotent orbits corresponds to the trivial representations. 
The label of the orbits is the same as in \cite{Orbits}.

\section*{Chapter structure}
In sections \ref{GLn_header},\ref{G2_header} we 
consider $G=GL_{n}$ and $G=G_2$. We 
recall the known results regarding the poles
and their orders of degenerate Eisesntein series.

In sections \ref{F4_header},\ref{E6_header} ,and \ref{E7_header} 
we compute the order of the poles using the algorithm described in Chapter \ref{ch4:algo} for the group of type $F_4,E_6,E_7$.
The algorithm has been implemented using the sage packet. The code may be viewed at my homepage \footnote{
\url{https://www.math.bgu.ac.il/~halawi/}}. Below we only state the results. 
 A typical output of the program can be seen at Appendix \ref{Ap1:Code_example}.

We also determine which leading terms of the spherical Eisenstein series above 
are square integrable.  

For the rest of the chapter let us fix some notations: \\
Let $\para{P}_m$ be maximal parabolic subgroup of $G$ corresponded to $\Delta_{G} \setminus \{\alpha_m\}$. We
denote by $I_{\para{P}_m}(s)=Ind_{\para{P}_{m}(\adele{A})}^{G(\adele{A})}(\delta_{\para{P}_m}^{s})=\otimes_{\nu} Ind_{\para{P}_{m}(F_{\nu})}^{G(F_{\nu})}(\delta_{\para{P}_m}^{s})$.
For the exceptional groups we list our results in the following form:\\
For each $s_0 \in \mathbb{C}$  such that $\operatorname{Re}s_0>0$ and $E_{\para{P}}(f^{0},s,g)$ admits a pole at $s=s_0$ we write its order and whenever is square integrable or not. If it is square integrable we also write the orbit that corresponds for it in the dual group. 
\newpage
\section{The group $G=GL_n$} \label{GLn_header}
Mui\'c and Hanzer in  \cite{GL_n} have studied the poles of degenerate Eisenstein series for \\ $G=GL_n$. Let us recall their results. 

The Dynkin diagram of $G$ is of type $A_{n-1}$ and we labeled the roots as follows:
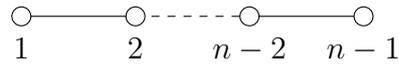
\begin{figure}[H]
\begin{center}
\begin{tikzpicture}[scale=0.5]
\draw (-1,0) node[anchor=east] {};
\draw (0 cm,0) -- (3 cm,0);
\draw[dashed,->] (3 cm,0) -- (6 cm,0);
\draw (6 cm,0) -- (9 cm,0 cm);
\draw[fill=white] (0 cm, 0 cm) circle (.25cm) node[below=4pt]{$1$};
\draw[fill=white] (3 cm, 0 cm) circle (.25cm) node[below=4pt]{$2$};
\draw[fill=white] (6 cm, 0 cm) circle (.25cm) node[below=4pt]{$n-2$};
\draw[fill=white] (9 cm, 0 cm) circle (.25cm) node[below=4pt]{$n-1$};
\end{tikzpicture}
\end{center}
\end{figure}
Let $\para{P}_{m}=MN$ be a maximal parabolic subgroup of $G$, with $M\simeq GL_{m} \times GL_{n-m}$. 
\begin{remk}
Note that in our notations, the representation 
 $I_{\para{P}_m}(s)$ in  \cite{GL_n} is denoted  $I_{\para{P}_m}(\frac{s}{n})$.
\end{remk}
\subsection{Poles of spherical Eisenstein series}
For every maximal parabolic subgroup $\para{P}_m$ we associate its Eisenstein series $E_{\para{P}_m}(f,g,s)$.
\begin{thm}[\cite{GL_n},Theorem 5.1,5.2]
Let $m \leq \frac{n}{2}$, and $Re \: s \geq 0$.
Then 
\begin{enumerate}
\item
 For $s \not \in \{\frac{1}{2}-\frac{a}{n} \: :  \: a \in \mathbb{Z}, \:\:\: 0\leq a \leq m-1   \}$ 
 the Eisenstein series $E_{\para{P}_m}(f,g,s)$ is holomorphic and non-zero.
\item
For $s_0  \in \{\frac{1}{2}-\frac{a}{n} \: : \: a \in \mathbb{Z}, \: \:\: 0\leq a \leq m-1   \}$ 
 the Eisenstein series $E_{\para{P}_m}(f,g,s_0)$ admits at most  a simple pole, that is attained by the normalized spherical section.
\end{enumerate} 
\end{thm}
\begin{remk}\label{GLn::remk}
The representation $\Pi_{-1,s_0,\para{P}_m}^{0}$ is not in $L^{2}_{d}(G(F) \backslash G(\adele{A}))$ except for the case $s_0=\frac{1}{2}$.
\end{remk}
Remark (\ref{GLn::remk}) agrees with the results of \cite{op} since the only distinguished  orbit of $GL_n$ is the principal one.
\newpage
\section{The group $G=G_2$}\label{G2_header}
We recall the results of Ginzburg and Jiang  in  \cite{G2_Ginz}.
The group $G=G_2$ is an exceptional split group. Its simple roots are labeled as follows:
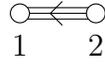
\begin{figure}[H]
\begin{center}
\begin{tikzpicture}[scale=0.5]
\draw (-1,0) node[anchor=east] {};
\draw (0,0) -- (2 cm,0);
\draw (0, 0.15 cm) -- +(2 cm,0);
\draw (0, -0.15 cm) -- +(2 cm,0);
\draw[shift={(0.8, 0)}, rotate=180] (135 : 0.45cm) -- (0,0) -- (-135 : 0.45cm);
\draw[fill=white] (0 cm, 0 cm) circle (.25cm) node[below=4pt]{$1$};
\draw[fill=white] (2 cm, 0 cm) circle (.25cm) node[below=4pt]{$2$};
\end{tikzpicture}
\end{center}
\end{figure}

\begin{remk}
Note that in our notations, the representation 
 $I_{\para{P}_m}(s)$ in  \cite{G2_Ginz} is denoted  $I_{\para{P}_m}(s-\frac{1}{2})$.
\end{remk}
\subsection{Poles of spherical Eisenstein series}
\begin{prop}\cite{G2_Ginz}
 For $Re \: s >0$ it holds that:
 \begin{enumerate}
 \item
 The Eisenstein series $E_{\para{P}_1}(f^{0},g,s)$ is holomorphic except for $s\in \{\frac{1}{10} 
 ,\frac{1}{2}\}$. At these points it admits a simple pole.
 \item
 The Eisenstein series $E_{\para{P}_2}(f^{0},g,s)$ is holomorphic except for $s\in \{\frac{1}{6} 
  ,\frac{1}{2}\}$. At $s=\frac{1}{6}$ (resp. $s=\frac{1}{2}$) it admits a pole of order 2 (resp. 1).
  \item
  For the data above, the leading term $\leadingterm{-d}(f,g,s)$  of $E_{\para{P}_i}(f,g,s)$   at $s=s_0$ is square integrable.
 \end{enumerate}
 \end{prop}
 
\begin{remk}\label{G2Orbirts}
In $G_2$ we have 2 distinguished orbits: 
\begin{enumerate}
\item
The principal orbit that corresponds to the trivial representation ($s=\frac{1}{2}$).
\item
The subregular orbit, labeled as $G_2(\alpha_1)$, that  corresponds to 
$(\para{P}_1,\frac{1}{10}),(\para{P}_2,\frac{1}{6})$.
\end{enumerate}

\end{remk}
\newpage
\section{The group $G=F_4$}\label{F4_header}
The group $G=F_4$ is an exceptional split group. Its simple roots are labeled as follows: 
\begin{figure}[H]
\begin{center}
\begin{tikzpicture}[scale=0.5]
\draw (-1,0) node[anchor=east] {};
\draw (0 cm,0) -- (2 cm,0);
\draw (2 cm, 0.1 cm) -- +(2 cm,0);
\draw (2 cm, -0.1 cm) -- +(2 cm,0);
\draw (4.0 cm,0) -- +(2 cm,0);
\draw[shift={(3.2, 0)}, rotate=0] (135 : 0.45cm) -- (0,0) -- (-135 : 0.45cm);
\draw[fill=white] (0 cm, 0 cm) circle (.25cm) node[below=4pt]{$1$};
\draw[fill=white] (2 cm, 0 cm) circle (.25cm) node[below=4pt]{$2$};
\draw[fill=white] (4 cm, 0 cm) circle (.25cm) node[below=4pt]{$3$};
\draw[fill=white] (6 cm, 0 cm) circle (.25cm) node[below=4pt]{$4$};
\end{tikzpicture}
\end{center}
\end{figure}

\subsection{Poles of spherical Eisenstein series}
 \begin{tabular}{|c|Hc|Hc|Hc|p{1cm}|c|HHc|c|c|c|} 
 \cline{0-6}\cline{9-15}
$\para{P}_{ 1 }$ & $\frac{1}{16}$ & $\frac{1}{8}$ & $\frac{3}{16}$ & $\frac{1}{4}$ & $\frac{3}{8}$ & $\frac{1}{2}$ 
&
&
$\para{P}_{ 2 }$ & $\frac{1}{30}$ & $\frac{1}{20}$ & $\frac{1}{10}$ & $\frac{1}{5}$ & $\frac{3}{10}$ & $\frac{1}{2}$
\\ 
 \cline{0-6}\cline{9-15}
Pole order & $0$ & $1$ & $0$ & $1$ & $0$ & $1$
&
& 
Pole order & $0$ & $0$ & $3$ & $1$ & $2$ & $1$ 
 \\  \cline{0-6}\cline{9-15}
$L_2$ & \xmark & \cmark & \xmark & \cmark & \xmark & \cmark 
&
& 
$L_2$ & \xmark & \xmark & \cmark & \xmark & \cmark & \cmark
\\  \cline{0-6}\cline{9-15}
Orbit &  & $F_4(a_2)$ &  & $F_4(a_1)$ &  & $F_4$ 
&
&
Orbit &  &  & $F_4(a_3)$ &  & $F_4(a_1)$ & $F_4$
\\  \cline{0-6}\cline{9-15}
\end{tabular} 
\vspace*{0.5cm}
\\
 \begin{tabular}{|c|Hc|Hc|c|c|p{5pt}|c|c|Hc|HHc|} 
 \cline{0-6}\cline{9-15}
$\para{P}_{ 3 }$ & $\frac{1}{42}$ & $\frac{1}{14}$ & $\frac{1}{7}$ & $\frac{3}{14}$ & $\frac{5}{14}$ & $\frac{1}{2}$ 
&
&
$\para{P}_{ 4 }$ & $\frac{1}{22}$ & $\frac{3}{22}$ & $\frac{5}{22}$ & $\frac{7}{22}$ & $\frac{9}{22}$ & $\frac{1}{2}$
\\  \cline{0-6}\cline{9-15} 

Pole order & $0$ & $2$ & $0$ & $2$ & $1$ & $1$
&
&
Pole order & $1$ & $0$ & $1$ & $0$ & $0$ & $1$ 
\\  \cline{0-6}\cline{9-15} 
$L_2$ & \xmark & \cmark & \xmark & \cmark & \xmark & \cmark
&
&
$L_2$ & \cmark & \xmark & \cmark & \xmark & \xmark & \cmark 
\\  \cline{0-6}\cline{9-15} 
Orbit &  & $F_4(a_3)$ &  & $F_4(a_2)$ &  & $F_4$ 
&
&
Orbit & $F_4(a_2)$ &  & $F_4(a_1)$ &  &  & $F_4$ 
\\  \cline{0-6}\cline{9-15} 
\end{tabular} 
\newpage
\section{The group $G=E_6$}\label{E6_header}
The group $G=E_6$ is an exceptional split group. Its simple roots are labeled as follows:
\begin{figure}[H]
\begin{center}
\begin{tikzpicture}[scale=0.5]
\draw (-1,0) node[anchor=east] {};
\draw (0 cm,0) -- (8 cm,0);
\draw (4 cm, 0 cm) -- +(0,2 cm);
\draw[fill=white] (0 cm, 0 cm) circle (.25cm) node[below=4pt]{$1$};
\draw[fill=white] (2 cm, 0 cm) circle (.25cm) node[below=4pt]{$3$};
\draw[fill=white] (4 cm, 0 cm) circle (.25cm) node[below=4pt]{$4$};
\draw[fill=white] (6 cm, 0 cm) circle (.25cm) node[below=4pt]{$5$};
\draw[fill=white] (8 cm, 0 cm) circle (.25cm) node[below=4pt]{$6$};
\draw[fill=white] (4 cm, 2 cm) circle (.25cm) node[right=3pt]{$2$};
\end{tikzpicture}
\end{center}
\end{figure}
\subsection{Poles of spherical Eisenstein series}
 \begin{tabular}{|c|HHc|HHc|p{1.7cm}|c|c|Hc|c|Hc|} 
 \cline{0-6}\cline{9-15}
$\para{P}_{ 1 }, \para{P}_{ 6 } $ & $\frac{1}{12}$ & $\frac{1}{6}$ & $\frac{1}{4}$ & $\frac{1}{3}$ & $\frac{5}{12}$ & $\frac{1}{2}$ 
&
&
$\para{P}_{ 2 }$ & $\frac{1}{22}$ & $\frac{3}{22}$ & $\frac{5}{22}$ & $\frac{7}{22}$ & $\frac{9}{22}$ & $\frac{1}{2}$ \\  \cline{0-6}\cline{9-15}
Pole order & $0$ & $0$ & $1$ & $0$ & $0$ & $1$
&
&
Pole order & $1$ & $0$ & $1$ & $1$ & $0$ & $1$ \\ 
 \cline{0-6}\cline{9-15}
$L_2$ & \xmark & \xmark & \cmark & \xmark & \xmark & \cmark
&
&
$L_2$ & \cmark & \xmark & \xmark & \cmark & \xmark & \cmark 
 \\ 
  \cline{0-6}\cline{9-15}
Orbit &  &  & $E_6(a_1)$ &  &  & $E_6$ 
&
&
Orbit & $E_6(a_3)$ &  &  & $E_6(a_1)$ &  & $E_6$ \\
  \cline{0-6}\cline{9-15}
\end{tabular} 
 \vspace{0.5cm}
 \\
 \begin{tabular}{|c|HHc|c|c|c|p{3pt}|c|Hc|c|c|c|c|} 
 \cline{0-6}\cline{9-15}
$\para{P}_{ 3 }$,$\para{P}_{ 5 }$ & $\frac{1}{18}$ & $\frac{1}{9}$ & $\frac{1}{6}$ & $\frac{5}{18}$ & $\frac{7}{18}$ & $\frac{1}{2}$
&
&
$\para{P}_{ 4 }$ & $\frac{1}{42}$ & $\frac{1}{14}$ & $\frac{1}{7}$ & $\frac{3}{14}$ & $\frac{5}{14}$ & $\frac{1}{2}$ 
 \\  \cline{0-6}\cline{9-15}
Pole order & $0$ & $0$ & $2$ & $1$ & $1$ & $1$
&
&
Pole order & $0$ & $2$ & $1$ & $3$ & $2$ & $1$ \\  \cline{0-6}\cline{9-15}
$L_2$ & \xmark & \xmark & \cmark & \xmark & \xmark & \cmark
&
&
$L_2$ & \xmark & \xmark & \xmark & \cmark & \cmark & \cmark 
 \\  \cline{0-6}\cline{9-15}
Orbit &  &  & $E_6(a_3)$ &  &  & $E_6$ 
&
&
Orbit &  &  &  & $E_6(a_3)$ & $E_6(a_1)$ & $E_6$ \\ \cline{0-6}\cline{9-15}
\end{tabular} 
\\
\begin{remk}
For the case $\para{P}=\para{P}_4$ and $s_0=\frac{1}{7}$ the result is valid under the assumption that $\zeta(\frac{1}{2})\neq 0$.
\end{remk}
\newpage
\section{The group $G=E_7$}\label{E7_header}
The group $G=E_7$ is an exceptional split group. Its simple roots are labeled as follows:
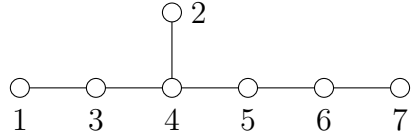
\begin{figure}[H]
\begin{center}
\begin{tikzpicture}[scale=0.5]
\draw (-1,0) node[anchor=east] {};
\draw (0 cm,0) -- (10 cm,0);
\draw (4 cm, 0 cm) -- +(0,2 cm);
\draw[fill=white] (0 cm, 0 cm) circle (.25cm) node[below=4pt]{$1$};
\draw[fill=white] (2 cm, 0 cm) circle (.25cm) node[below=4pt]{$3$};
\draw[fill=white] (4 cm, 0 cm) circle (.25cm) node[below=4pt]{$4$};
\draw[fill=white] (6 cm, 0 cm) circle (.25cm) node[below=4pt]{$5$};
\draw[fill=white] (8 cm, 0 cm) circle (.25cm) node[below=4pt]{$6$};
\draw[fill=white] (10 cm, 0 cm) circle (.25cm) node[below=4pt]{$7$};
\draw[fill=white] (4 cm, 2 cm) circle (.25cm) node[right=3pt]{$2$};
\end{tikzpicture}
\end{center}
\end{figure}
\subsection{Poles of spherical Eisenstein series}
The following results are valid under the following assumptions:
\begin{enumerate}
\item
 In the case where 
$\para{P}=\para{P}_4$ and $s \in \{\frac{1}{16},\frac{3}{16} \}$ we assume that $\zeta(\frac{1}{2}) \neq 0$.
\item
In the case where 
$\para{P}=\para{P}_4$ and $s =\frac{1}{12}$
we assume that  $\zeta(\frac{2}{3}) \neq 0$. 
\end{enumerate}
\begin{landscape}

 \begin{tabular}{|c|cHH|c|Hc|HHc|p{64pt}|c|Hc|Hc|c|c|c|Hc|}
\cline{0-9}\cline{12-21}
$\para{P}_{ 1 }$ & $\frac{1}{34}$ & $\frac{3}{34}$ & $\frac{5}{34}$ & $\frac{7}{34}$ & $\frac{9}{34}$ & $\frac{11}{34}$ & $\frac{13}{34}$ & $\frac{15}{34}$ & $\frac{1}{2}$
&
&
$\para{P}_{ 2 }$ & $\frac{1}{28}$ & $\frac{1}{14}$ & $\frac{3}{28}$ & $\frac{1}{7}$ & $\frac{3}{14}$ & $\frac{2}{7}$ & $\frac{5}{14}$ & $\frac{3}{7}$ & $\frac{1}{2}$
 \\ \cline{0-9}\cline{12-21}
Pole order & $1$ & $0$ & $0$ & $1$ & $0$ & $1$ & $0$ & $0$ & $1$ 
&
&
Pole order & $0$ & $1$ & $0$ & $1$ & $1$ & $1$ & $1$ & $0$ & $1$ \\ \cline{0-9}\cline{12-21}
$L_2$ & \cmark & \xmark & \xmark & \cmark & \xmark & \cmark & \xmark & \xmark & \cmark
&
&
$L_2$ & \xmark & \cmark & \xmark & \xmark & \xmark & \xmark & \cmark & \xmark & \cmark \\ \cline{0-9}\cline{12-21}
Orbit & $E_7(a_3)$ &  &  & $E_7(a_2)$ &  & $E_7(a_1)$ &  &  & $E_7$ 
&
&
Orbit &  & $E_7(a_4)$ &  &  &  &  & $E_7(a_1)$ &  & $E_7$ \\ \cline{0-9}\cline{12-21}
\end{tabular} 
\vspace{20pt}
\\
 \begin{tabular}{|cH|cH|cH|c|c|c|c|p{4pt}|cHH|c|c|c|c|c|c|c|} 
\cline{0-9}\cline{12-21}
$\para{P}_{ 3 }$ & $\frac{1}{66}$ & $\frac{1}{22}$ & $\frac{1}{11}$ & $\frac{3}{22}$ & $\frac{2}{11}$ & $\frac{5}{22}$ & $\frac{7}{22}$ & $\frac{9}{22}$ & $\frac{1}{2}$
&
&
$\para{P}_{ 4 }$ & $\frac{1}{32}$ & $\frac{1}{24}$ & $\frac{1}{16}$ & $\frac{1}{12}$ & $\frac{1}{8}$ & $\frac{3}{16}$ & $\frac{1}{4}$ & $\frac{3}{8}$ & $\frac{1}{2}$ \\ \cline{0-9}\cline{12-21}
Pole order & $0$ & $2$ & $0$ & $2$ & $0$ & $2$ & $1$ & $1$ & $1$
&
&
Pole order & $0$ & $0$ & $1$ & $1$ & $4$ & $1$ & $3$ & $2$ & $1$ \\ \cline{0-9}\cline{12-21}
$L_2$ & \xmark & \cmark & \xmark & \cmark & \xmark & \cmark & \xmark & \xmark & \cmark 
&
&
$L_2$ & \xmark & \xmark & \xmark & \xmark & \cmark & \xmark & \cmark & \cmark & \cmark \\ \cline{0-9}\cline{12-21}
Orbit &  & $E_7(a_5)$ &  & $E_7(a_4)$ &  & $E_7(a_3)$ &  &  & $E_7$ 
&
&
Orbit &  &  &  &  & $E_7(a_5)$ &  & $E_7(a_3)$ & $E_7(a_1)$ & $E_7$ \\ \cline{0-9}\cline{12-21}
\end{tabular} 
\vspace{20pt}
\\
 \begin{tabular}{|c|HHHc|c|c|c|c|c|p{51pt}|c|c|HHHc|c|Hc|c|}
 \cline{0-9}\cline{12-21}
$\para{P}_{ 5 }$ & $\frac{1}{30}$ & $\frac{1}{20}$ & $\frac{1}{15}$ & $\frac{1}{10}$ & $\frac{3}{20}$ & $\frac{1}{5}$ & $\frac{3}{10}$ & $\frac{2}{5}$ & $\frac{1}{2}$ 
&
&
$\para{P}_{ 6 }$ & $\frac{1}{26}$ & $\frac{1}{13}$ & $\frac{3}{26}$ & $\frac{2}{13}$ & $\frac{5}{26}$ & $\frac{7}{26}$ & $\frac{9}{26}$ & $\frac{11}{26}$ & $\frac{1}{2}$ \\ \cline{0-9}\cline{12-21}
Pole order & $0$ & $0$ & $0$ & $3$ & $1$ & $2$ & $2$ & $1$ & $1$ 
&
&
Pole order & $1$ & $0$ & $0$ & $0$ & $2$ & $1$ & $0$ & $1$ & $1$ 
\\ \cline{0-9}\cline{12-21}
$L_2$ & \xmark & \xmark & \xmark & \cmark & \xmark & \xmark & \cmark & \xmark & \cmark 
&
&
$L_2$ & \cmark & \xmark & \xmark & \xmark & \cmark & \cmark & \xmark & \xmark & \cmark \\ \cline{0-9}\cline{12-21}
Orbit &  &  &  & $E_7(a_5)$ &  &  & $E_7(a_2)$ &  & $E_7$ 
&
&
Orbit & $E_7(a_4)$ &  &  &  & $E_7(a_3)$ & $E_7(a_2)$ &  &  & $E_7$ \\ \cline{0-9}\cline{12-21}
\end{tabular} 
\vspace{20pt}
\\
 \begin{tabular}{|c|c|HHHc|HHHc|} \hline
$\para{P}_{ 7 }$ & $\frac{1}{18}$ & $\frac{1}{9}$ & $\frac{1}{6}$ & $\frac{2}{9}$ & $\frac{5}{18}$ & $\frac{1}{3}$ & $\frac{7}{18}$ & $\frac{4}{9}$ & $\frac{1}{2}$ \\ \hline
Pole order & $1$ & $0$ & $0$ & $0$ & $1$ & $0$ & $0$ & $0$ & $1$ \\ \hline
$L_2$ & \cmark & \xmark & \xmark & \xmark & \cmark & \xmark & \xmark & \xmark & \cmark \\ \hline
Orbit & $E_7(a_2)$ &  &  &  & $E_7(a_1)$ &  &  &  & $E_7$ \\ \hline
\end{tabular} 
\end{landscape}
\chapter{Normalized Eisenstein series} 

\label{ch2:NormalizedEise} 

\lhead{Chapter \ref{ch2:NormalizedEise}. \emph{Normalized Eisenstein series}} 
In this chapter we will see the relations between degenerate Eisenstein series associated to maximal parabolic subgroup and the Eisenstein series associated to the Borel subgroup. This relation stated in \cite{Ike1}. We also introduce the normalized Eisenstein series, that is entire and $\weyl{G}$ invariant. 
Its properties, stated  in \cite{Ike1}, play a crucial role in the proof of our main theorem. For the convenience of the reader we include the proof of its properties. \\

We denote by $W(G,\para{P})$ the set of the shortest representatives in $\weyl{G} \slash \weyl{\para{P}}$.
Let fix some notations:
let $\chi \in X(T)$ be an unramified character of $T$. For every $\alpha \in \Phi$ we denote by 
\begin{align*}
l_{\alpha}^{\pm}(\chi)=& \left \langle \chi,\check{\alpha} \right \rangle \pm 1 &
\mathcal{F}_{\alpha}^{-}=&\{\chi \in \mathfrak{a}^{\ast}_{\mathbb{C}} \: : l_{\alpha}^{-}(\chi)=0 \}.
\end{align*}
\begin{prop}\label{ReducetoParabolic}
Let $\para{P}=\para{P}_r$ be maximal parabolic of $G$.
	Let $\chi_{s}= \delta_{\para{P}}^{s-\frac{1}{2}}\otimes 
	\delta_{\para{B}}^{\frac{1}{2}}$  Then the  iterated  residue of $E_{\para{B}}(\chi,f^{0},g)$ along 
	$\mathcal{F}_{\hat{r}}=\mathcal{F}^{-}_{\alpha_1}, \dots, \mathcal{F}^{-}_{\alpha_{r-1}},\mathcal{F}_{\alpha_{r+1}}^{-},\dots ,\mathcal{F}^{-}_{\alpha_{n}}$
	is equal to $Res_{\mathcal{F}_{\hat{r}}} C_{\weylelement{_{\para{P}}}}(\chi) \times E_{\para{P}}(f^{0},s,g)
	$ where $\weylelement{_{\para{P}}}$ is the longest element in $\weyl{\para{P}}$. Moreover 
	$Res_{\mathcal{F}_{\hat{r}}} C_{\weylelement{\para{P}}}(\chi_s) \neq 0$ and is a constant function.
\end{prop}

\begin{proof}
Since $E_{\para{B}}(f^{0},\chi_ s,g)$ and $E_{\para{P}}(f^{0},s,g)$ have the same cuspidal support, it is enough to show that 
$\operatorname{Res}_{\mathcal{F}_{\hat{r}}} E_{\para{B}}(f^{0},\chi_s,g)$
 and $\operatorname{Res}_{\mathcal{F}_{\hat{r}}} C_{\weylelement{_{\para{P}}}}(\chi) \times E_{\para{P}}(f^{0},s,g)$ have the same constant term along the unipotent radical $U$ of the Borel see (\cite{EisenstienSeries}, Proposition I.3.4).
Indeed, 
\begin{align*}
 E_{\para{B}}^{0}(f^{0},\chi_s,g)&= 
\sum_{\weylelement{} \in \weyl{G}}M_{\weylelement{}}(\chi_s)f^{0}_{ \chi_s} \\
&=
\sum_{\weylelement{} \in \weyl{G} \slash \weyl{\para{P}}} \sum_{u \in \weyl{\para{P}}}
M_{\weylelement{}u}(\chi_s)f^{0}_{ \chi_s} \\
&\overset{\ref{gid}}{=}
\sum_{\weylelement{} \in \weyl{G} \slash \weyl{\para{P}}} \sum_{u \in \weyl{\para{P}}}
C_{u}(\chi_s)M_{\weylelement{}}(u\chi_s)f^{0}_{u \chi_s} 
\end{align*}
If $u \neq \weylelement{\para{P}}$, i.e. $u$ is not the longest element in $\weyl{\para{P}}$, then there exists a simple root $\alpha_i \in \Delta_{\para{P}}$ such that 
$u\alpha_i>0$. Hence $C_{u}(\chi)$ does not contain the factor $\frac{\zeta(\inner{\chi_s,\check{\alpha}_i})}{
\zeta(\inner{\chi_s,\check{\alpha}_i}+1)}$.
Therefore, 
$$\operatorname{Res}_{\mathcal{F}_{\hat{r}}} C_{u}(\chi) =\operatorname{lim}_{\chi \rightarrow \chi_s}\prod_{\alpha \in \Delta_{\para{P}}} (\inner{\chi,\check{\alpha}}-1) C_{u}(\chi)=\begin{cases}
0 & \mbox{if } u \neq \weylelement{\para{P}} \\
 \operatorname{Res}_{\mathcal{F}_{\hat{r}}} C_{\weylelement{\para{P}}}(\chi_s) & \mbox{if } u = \weylelement{\para{P}} \\
\end{cases}
$$
Therefore,
\begin{align*}
 \operatorname{Res}_{\mathcal{F}_{\hat{r}}}E_{\para{B}}^{0}(f^{0},\chi_s,g)&= \operatorname{Res}_{\mathcal{F}_{\hat{r}}} C_{\weylelement{\para{P}}}(\chi_s) \times 
 \sum_{\weylelement{} \in \weyl{G} \slash  \weyl{\para{P}}}M_{\weylelement{}}(\weylelement{\para{P}} \chi_s)f^{0}_{\weylelement{}\weylelement{\para{P}\chi_s}}\\
 &=
 \operatorname{Res}_{\mathcal{F}_{\hat{r}}} C_{\weylelement{\para{P}}}(\chi_s) \times 
  E_{\para{P}}^{0}(f^{0},s,g).
\end{align*}
\\
Note that 
$$C_{\weylelement{\para{P}}}(\chi_s)	= \prod_{\alpha \in \Phi_{M}^{+}}\frac{\zeta(\langle \chi_{s},\check{\alpha} \rangle)}{\zeta(\langle \chi_{s},\check{\alpha} \rangle+1)}.$$ 
\\
For given $\alpha>0$ it holds that   $\check{\alpha} =\sum_{\alpha \in \Delta_{G}}n_{\alpha}^{(\check{\alpha}_i)}\check{\alpha}_i$ where 
$n_{\alpha}^{(\check{\alpha}_i)} \in \mathbb{N}\cup \{0\}$.
Moreover if $\alpha \in \Phi^{+}_{M}$ it holds that
 $n_{\alpha}^{(\check{\alpha}_r)}=0$.
 Therefore for every $\alpha \in \Phi_{M}^{+} \setminus{\Delta_{G}}$  it holds 
 $1<\langle \chi_{s},\check{\alpha} \rangle \in \mathbb{N}$.
 Since the zeros of $\zeta(s)$ lie in $0 < Re \: s <1 $, we conclude that the every term corresponds to a non-simple root in that product is a non zero  constant.
 For $\alpha \in \Phi_{M}^{+} \cap \Delta_{G}$ it holds that 
 $\langle \chi_{s},\check{\alpha} \rangle=1$.  Therefore by taking the iterated residue along the $\mathcal{F}_{\hat{r}}$ we get a non zero constant number.
\end{proof}

\begin{defn}
	Let $f^{0} \in Ind_{\para{B}}^{G}(\chi)$ be the normalized spherical section. We define the \textbf{normalized Eisenstein series} to be 
	$$E_{\para{B}}^{\#}(\chi,g) = \prod_{\alpha \in \Phi^{+}}\zeta(l_{\alpha}^{+}(\chi))\cdot l_{\alpha}^{+} (\chi) \cdot
	l_{\alpha}^{-}(\chi)E_{\para{B}}(f^{0},\chi,g).$$
\end{defn}
\newpage
\begin{thm}\label{Ikeda_Series}
The normalized Eisenstein series $E_{\para{B}}^{\#}$ is entire and $\weyl{G}$ invariant, i.e. $E_\para{B}^{\#}(\chi,g)=E_\para{B}^{\#}(\weylelement{}\chi,g)$ for every $\weylelement{}\in \weyl{G}$.
\end{thm}
\begin{proof}
We shall prove first the  $\weyl{G}$ invariance property. 
Observe that for given $\weylelement{} \in \weyl{G}$ it holds that 
\begin{align*}
 l_{\alpha}^{\pm}(\weylelement{}\chi)=& l_{\weylelement{}^{-1}\alpha}^{\pm}(\chi).
\end{align*}
Furthermore, since $\weyl{G}$ is generated by simple reflections it is enough to show the $\weyl{G}$ invariance for simple reflection.   
Let $\weylelement{}= \weylelement{i}$ be a simple reflection corresponding  to the simple root $\alpha_i$. Note that the following holds:
	\begin{enumerate}
	\item
	$w^{-1}=w$.
	\item \label{main_t1}
	$w(\Phi^{+} \setminus \{\alpha_i\})=\Phi^{+} \setminus \{\alpha_i\}$.
	\end{enumerate}
It implies that:
\begin{align}
\label{main_l_invariant}\prod_{\alpha \in \Phi^{+} \setminus \{\alpha_i\}} l_{\alpha}^{+}(\chi)l_{\alpha}^{-}(\chi)=&
\prod_{\alpha \in \Phi^{+} \setminus \{\alpha_i\}}
l_{\alpha}^{+}(\weylelement{}\chi)l_{\alpha}^{-}(\weylelement{}\chi)\\
\label{main_zeta_invariant}\prod_{\alpha \in \Phi^{+} \setminus \{\alpha_i\}} \zeta(l_{\alpha}^{+}(\chi))=&
\prod_{\alpha \in \Phi^{+} \setminus \{\alpha_i\}}
\zeta(l_{\alpha}^{+}(\weylelement{}\chi)).
\end{align}
Therefore, we get:
\begin{align*}
\prod_{\alpha \in \Phi^{+} \setminus\{\alpha_i\}} \zeta(l_{\alpha}^{+}(\weylelement{}\chi)) l_{\alpha}^{+}(\weylelement{}\chi)l_{\alpha}^{-}(\weylelement{} \chi)&= \prod_{\alpha \in \Phi^{+} \setminus\{\alpha_i\}} \zeta(l_{\alpha}^{+}(\chi)) l_{\alpha}^{+}(\chi)l_{\alpha}^{-}( \chi).
\end{align*}
So, it remains to show that 
$$\zeta( l_{\alpha_{i}}^{+}(\chi)) l_{\alpha_{i}}^{+}(\chi) l_{\alpha_{i}}^{-}(\chi)
E_{\para{B}}(f^{0},\chi,g)= \zeta( l_{\alpha_{i}}^{+}(\weylelement{}\chi)) l_{\alpha_{i}}^{+}(\weylelement{}\chi) l_{\alpha_{i}}^{-}(\weylelement{}\chi)
E_{\para{B}}(f^{0},\weylelement{}\chi,g)
$$

Recall that both  Eisenstein series and the zeta function admit  functional equations:
\begin{align*}
E_{\para{B}}(f^{0},\chi,g) =& E_{\para{B}}(M_{\weylelement{}}(\chi)f^{0},\weylelement{}\chi,g) \\
\zeta(s) =& \zeta(1-s). \\
\end{align*}
Using (\ref{gid}) it holds that:
$$ E_{\para{B}}(f^{0},\chi,g) = \frac{\zeta(\inner{\chi,\check{\alpha_i}})}{
\zeta(\inner{\chi,\check{\alpha_i}}+1)
} E_{\para{B}}(f^{0},\weylelement{}\chi,g).$$
Observe that $l_{\alpha_i}^{\pm}(\weylelement{i}\chi)= -l_{\alpha_i}^{\mp}(\chi)$.  Thus, 
\begin{align*}
\zeta(l_{\alpha_i}^{+}(\weylelement{}\chi)) 
(l_{\alpha_i}^{+}(\weylelement{}\chi)) 
(l_{\alpha_i}^{-}(\weylelement{}\chi)) 
 E_{\para{B}}(f^{0},\weylelement{} \chi,g) &= 
 \zeta(-l_{\alpha_i}^{-}(\chi)) 
 (l_{\alpha_i}^{+}(\chi)) 
 (l_{\alpha_i}^{-}(\chi)) E_{\para{B}}(f^{0},\weylelement{}\chi,g) \\
 &=
 \zeta(l_{\alpha_i}^{+}(\chi)) 
  (l_{\alpha_i}^{+}(\chi)) 
  (l_{\alpha_i}^{-}(\chi)) \frac{\zeta(\inner{\chi,\check{\alpha_i}})}{\zeta(l_{\alpha_i}^{+}(\chi))} E_{\para{B}}(f^{0},\weylelement{}\chi,g) \\
  &=
  \zeta(l_{\alpha_i}^{+}(\chi)) 
  (l_{\alpha_i}^{+}(\chi)) 
  (l_{\alpha_i}^{-}(\chi)) 
   E_{\para{B}}(f^{0},\chi,g).
\end{align*}
Since $E_{\para{B}}^{\#}(\chi,g)$ is a product of two $\weylelement{}$ invariant functions we conclude that $E_{\para{B}}^{\#}$ is $\weylelement{}$ invariant. 
Thus $E_{\para{B}}^{\#}(\chi,g)$  is $\weyl{G}$ invariant since it is invariant for every generator of $\weyl{G}$.
\\

Let's prove that $E_{\para{B}}^{\#}(\chi,g)$ is entire.
Recall that the constant term of the degenerate Eisenstein series along the unipotent radical of the Borel subgroup $E_{\para{B}}^{0}(f^{0},\chi,g)$ and the Eisenstein series itself share the same analytical behavior. Thus, in order to show that $E_{\para{B}}^{\#}$ is entire, it will be enough to show that the following function is entire
$$\Psi(\chi,g) =\prod_{\alpha \in \Phi^{+}}l_{\alpha}^{+}(\chi) l_{\alpha}^{-}(\chi) \zeta (l_{\alpha}^{+}(\chi)) E^{0}_{\para{B}}(f^{0},\chi,g). $$

Furthermore, by Proposition (\ref{Constant_term_along_borel}) the constant term can be written as :
 $$E^{0}_{\para{B}}(f^{0},\chi,g)  =
 \sum_{\weylelement{} \in \weyl{G}} C_{\weylelement{}}(\chi)f^{0}_{\weylelement{}\chi} =
 \sum_{\weylelement{} \in \weyl{G}}\prod_{\tiny \begin{aligned}
	\alpha\in \Phi^{+}_{G}\\ 
	w\alpha <0 \\
	\end{aligned}
}{\frac{\zeta(\left \langle \chi,\check{\alpha}\right \rangle)}{\zeta(\left \langle \chi,\check{\alpha}\right \rangle+1)}}
f_{\weylelement{}\chi}^{0} .
$$
Note that $ \Psi(\chi,g) $ can be written as: \\
\begin{align*}
\Psi(\chi,g)=&
\prod_{\alpha \in \Phi^{+}}l_{\alpha}^{+}(\chi) l_{\alpha}^{-}(\chi) \zeta (l_{\alpha}^{+} (\chi)) E^{0}_{\para{B}}(f^{0},\chi,g) \\
 =&
\prod_{\alpha \in \Phi^{+}}l_{\alpha}^{+}(\chi) l_{\alpha}^{-}(\chi) \zeta (l_{\alpha}^{+}(\chi)) 
\times 
\sum_{\weylelement{} \in \weyl{G}}
\prod_{\tiny \begin{aligned}
	\alpha\in \Phi^{+}_{G}\\ 
	w\alpha <0 \\
	\end{aligned}
}{\frac{\zeta(\left \langle \chi,\check{\alpha}\right \rangle)}{\zeta(\left \langle \chi,\check{\alpha}\right \rangle+1)}}
f^{0}_{\weylelement{}\chi} \\
=&
\underbrace{
\prod_{\alpha \in \Phi^{+}}l_{\alpha}^{+}(\chi) l_{\alpha}^{-}(\chi)
}_{L(\chi)} \times
\sum_{\weylelement{} \in \weyl{G}}
\underbrace{
\prod_{
\tiny \begin{aligned}
\alpha\in \Phi^{+}_{G}\\ 
w\alpha >0 \\
\end{aligned}
}		
 \zeta (l_{\alpha}^{+}(\chi)) 
\times 
\prod_{\tiny \begin{aligned}
	\alpha\in \Phi^{+}_{G}\\ 
	w\alpha <0 \\
	\end{aligned}
}\zeta(\left \langle \chi,\check{\alpha}\right \rangle)
f^{0}_{\weylelement{}\chi} }_{F_{\weylelement{}}(\chi)}
.
\end{align*}
While  $L(\chi)$ is entire, $F_{\weylelement{}}(\chi)$ can have possible poles along the hyperplanes \\
$$H^{\epsilon}_{\alpha} (\chi) = \{ \chi \in \mathfrak{a}^{\ast}_{\mathbb{C}} \: : \: \left\langle
\chi,\check{\alpha}\right \rangle= \epsilon \} \: \text{  for  } \: \epsilon \in \{-1,0,1\}.$$ \\
We will show that these poles are canceled either by each other or by the zeros of $L(\chi)$.

Observe that 
\begin{align} \label{main_t6}
F_{\weylelement{}}(\chi) = \left(
\prod_{\alpha \in \Phi^{+}} \zeta(l_{\alpha}^{+}(\chi)) \times C_{\weylelement{}}(\chi) \right) \times f_{\weylelement{}\chi}^{0}
\end{align}
 where $C_{\weylelement{}}(\chi)$ is the factor as in Theorem (\ref{gid}). 
  
We recall Hartog theorem.
\begin{thm}\cite[Vol~1, part~D, Theorem~4]{Gunning}
Suppose that $E$ is an analytic subset of $\mathbb{C}^{n}$ where $(n \geq 2)$ of complex dimension at most $n - 2$, then every function holomorphic on $\mathbb{C}^{n} \setminus  E$ can be
extended holomorphically to $\mathbb{C}^{n}$.
\end{thm}
Let
\begin{align*}
X =& \bigcup_{
\tiny \begin{aligned}
0<\alpha \neq \alpha'\\
\epsilon,\epsilon' \in \{-1,0,1\}
\end{aligned}}
( H_{\alpha}^{\epsilon} \cap
H_{\alpha'}^{\epsilon'}) &
Y =& 
\bigcup_{
	\tiny \begin{aligned}
	0<&\alpha \\
	\epsilon \in& \{-1,0,1\}
	\end{aligned}}
 H_{\alpha}^{\epsilon} \quad.
\end{align*}
Note that $X$ is of codimension 2. Thus, by Hartog theorem, it is enough to show that for every $\chi \in \mathfrak{\alpha}^{\ast}_{\mathbb{C}} \setminus X$ the function $\Psi(\chi,g)$ is holomorphic.
Moreover,
for given $\chi \not \in Y$  $\Psi(\chi,g)$ is holomoprhic.
 Hence, it remains to show that the for given  $\chi \in Y \setminus X$  the function $\Psi(\chi,g)$ is holomorphic.
\\
For given $\chi \in Y \bigcap \underset{\alpha>0}{\cup}(H^{\pm1}_{\alpha})$
it easy to see that the poles of $F_{\weylelement{}}(\chi)$ are canceled by the zeros of $L(\chi)$.\\
Hence, it remains to prove the holomorphic property for  $\chi_0 \in H^{0}_{\alpha} \setminus X $. 
Recall that  function of several complex variables is holomorphic at $s_0 \in \mathbb{C}^{n}$ if and only if it holomorphic with respect to each variable separately, i.e. it enough to show that $f(s_0 + se_i)$  at $s=0$, where $e_i$ is the standard basis element of $\mathbb{C}^{n}$.

Let $\{v_1, \dots, v_n\}$ be a basis for $\mathbb{C}^{n}$ such that for every $i$ it holds that  $\inner{v_i,\check{\alpha}}=2$.
Therefore it is enough to show 
 that $\Psi(\chi,g)$ is holomoprhic at $\chi =\chi_0$ with respect to this coordinate system.
\item{\textbf{Along $v_i$}} \\
Our main goal is to show that $\Psi(\chi_0 +sv_i,g)$  is holomoprhic at $s=0$ for every $v_{i}$. It is enough to show that 
for every $\weylelement{} \in \weyl{G} \slash \inner{\weylelement{\alpha}}$ the function $$F_{\weylelement{}\weylelement{\alpha}}(\chi_0 +sv_i) + 
F_{\weylelement{}}(\chi_0 +sv_i)$$ is holomoprhic at $s=0$.
Without loss of generality $l(\weylelement{}\weylelement{\alpha}) = l(\weylelement{})+1$.
\begin{align*}
& F_{\weylelement{}\weylelement{\alpha}}(\chi) + 
F_{\weylelement{}}(\chi) = \\&
\underbrace{
\left( \prod_{
\tiny \begin{aligned}
\alpha\in \Phi^{+}_{G} \setminus \{\alpha\}\\ 
w\alpha >0 \\
\end{aligned}
}		
 \zeta (l_{\alpha}^{+}(\chi)) 
\times 
\prod_{\tiny \begin{aligned}
	\alpha\in \Phi^{+}_{G} \setminus \{\alpha\} \\ 
	w\alpha <0 \\
	\end{aligned}
}\zeta(\left \langle \chi,\check{\alpha}\right \rangle)
\right)}_{K_1(\chi,\weylelement{} )} \times 
\underbrace{
\left(
\zeta(\inner{\chi,\check{\alpha}})f^{0}_{\weylelement{}\weylelement{\alpha}\chi} 
+ \zeta(\inner{\chi,\check{\alpha}}+1)
f^{0}_{\weylelement{}\chi}
\right)}_{K_2(\chi,\weylelement{})}.
\end{align*}
It is obvious that $K_1(\chi,\weylelement{})$ is holomorphic at small neighborhood of $\chi_0$ 
 since \\
$\chi_0 \in H_{\alpha}^{0} \setminus X$. Let us  show that $K_2(\chi_0 +sv_i,\weylelement{})$ is holomoprhic at $s=0$.
Denote by $\chi_s =\chi_0 +sv_i$ then for  $s=0$  it holds that $f^{0}_{\weylelement{}\weylelement{\alpha}\chi_s}=
f^{0}_{\weylelement{}\chi_s}$ and 
\begin{align*}
\operatorname{Res}_{s=0}
K_2(\chi_0 +sv_i,\weylelement{}) &= 
\operatorname{Res}_{s=0}
\zeta(2s)f^{0}_{\weylelement{}\weylelement{\alpha}\chi_s} 
+ \zeta(2s+1)
f^{0}_{\weylelement{}\chi_s}=0.
\end{align*} 
Hence it holomoprhic at $s=0$.
To sum up we showed that for every $\chi_0 \in H_{\alpha}^{0} \setminus X$ it holds that $\Psi(\chi,g)$ is holomoprhic for $\chi=\chi_0$  with respect to each variable separately. Hence  it holomoprhic at $\chi=\chi_0$. So we are done.
\end{proof}
Let us fix some notations:
Let $\para{P}$ be a maximal parabolic subgroup of $G$. 
Consider $\chi_s= \delta_{\para{P}}^{s- \frac{1}{2}} \otimes \delta_{\para{B}}^{\frac{1}{2}}$.
\begin{cor}\label{Lemma1}
With the notations as above, $E_{\para{B}}^{\#}(\chi_s,g)$ can be written as 
 $$E_{\para{B}}^{\#}(\chi_s,g) = A_{\weylelement{\para{P}}} \times
 G_{\para{P}}(\chi_s)   \times \prod_{\alpha \in \Delta_{\para{P}}} \zeta(l_{\alpha}^{+}(\chi_s))l_{\alpha}^{+}(\chi_s) \times E_{\para{P}}(f^{0},s,g)
 $$
 where 
 $G_{\para{P}}(\chi_s)= \prod_{\alpha \in \Phi^{+} \setminus \Delta_{\para{P}}} \zeta(l_{\alpha}^{+}(\chi_s))
 l_{\alpha}^{+}(\chi_s)
 l_{\alpha}^{-}(\chi_s)$
 and $A_{\weylelement{\para{P}}}=Res_{\mathcal{F}_{\hat{i}}} C_{\weylelement{\para{P}}}(\chi_s)$.
\end{cor}
\begin{remk}
The product $\prod_{\alpha \in \Delta_{\para{P}}} \zeta(l_{\alpha}^{+}(\chi_s))l_{\alpha}^{+}(\chi_s)$
is equal to $(2\zeta(2))^{n-1}$.
since for every $\alpha \in \Delta_{M}$ is holds 
$\inner{\chi_s,\check{\alpha}} =1$.
\end{remk}
Let $d_{\para{P}}(\chi_{s_0})$ denote the order of the zero of 
$G_{\para{P}}(\chi_s)$ at $s=s_0$.
\begin{thm}
The Eisenstein series $E_{\para{P}}(f^{0},s,g)$ admits a pole of order at most $d_{\para{P}}(\chi_{s_0})$ at $s=s_0$.
\end{thm}
\begin{proof}
Since $E_{\para{B}}^{\#}(\chi_s,g)$ is entire, the order of the zero of $G_{\para{P}}(s)$ at $s=s_0$ is an upper bound for the order of the pole of $E_{\para{P}}(f^{0},s,g)$ at $s=s_0$.
\end{proof}
We are mostly interested in the case where $\operatorname{Re} s>0$. Our observation shows that the poles of the degenerate Eisenstein series in the right half plane are real. In that case, we can express the  number $d_{\para{P}}(\chi_{s_0})$ in geometric terms for $F =\mathbb{Q}$.
Let,
\begin{align*}
N_{\epsilon}(\chi) =& \{  \alpha \in \Phi^{+} \: : \: \inner{\chi,\check{\alpha}}=\epsilon   \}  & 
N_{\epsilon_1,\dots ,\epsilon_k}(\chi) =& \bigcup_{i=1}^{k} N_{\epsilon_i}(\chi).
\end{align*}
\begin{thm}\label{Upper}
For $F= \mathbb{Q}$ and for $s \in \mathbb{R}$ it holds that 
$$d_{\para{P}}(\chi_s) = | N_{1}(\chi_{s_0}) | - | N_{0}(\chi_{s_0}) |-(n-1).$$
\end{thm}
\begin{proof}
Note that $\Delta_{\para{P}} \subset N_{1}(\chi_s)$ for every $s$.
Therefore,
we can rewrite as
\begin{align*}
G_{\para{P}}(\chi_s) =&\left( \prod_{\alpha \in N_{-1}(\chi_s)}
\zeta(l_{\alpha}^{+}(\chi_s)) l_{\alpha}^{+}(\chi_s)
\right) \times \prod_{\alpha \in N_{-1}(\chi_s)} l_{\alpha}^{-}(\chi_s) \\
\times & 
\left( \prod_{\alpha \in N_{0}(\chi_s)}
l_{\alpha}^{+}(\chi_s)) l_{\alpha}^{-}(\chi_s)
\right) \times \prod_{\alpha \in N_{0}(\chi_s)} \zeta(l_{\alpha}^{+}(\chi_s)) \\
\times & 
\left( \prod_{\alpha \in N_{1}(\chi_s)\setminus \Delta_{\para{P}}}
\zeta(l_{\alpha}^{+}(\chi_s)) l_{\alpha}^{+}(\chi_s)
\right) \times \prod_{\alpha \in N_{1}(\chi_s) \setminus \Delta_{\para{P}}} l_{\alpha}^{-}(\chi_s) \\
\times  &
\left( \prod_{\alpha \in\Phi^{+} \setminus N_{\pm1,0}(\chi_s)}
\zeta(l_{\alpha}^{+}(\chi_s)) l_{\alpha}^{+}(\chi_s)l_{\alpha}^{-}(\chi_s) \right)
\end{align*}
Note that:
\begin{align*}
 \left( \prod_{\alpha \in N_{0}(\chi_s)}
l_{\alpha}^{+}(\chi_s)) l_{\alpha}^{-}(\chi_s)
\right)=&(-1)^{|N_{0}(\chi_s)|}
& \prod_{\alpha \in N_{-1}(\chi_s)} l_{\alpha}^{-}(\chi_s)=&(-2)^{|N_{-1}(\chi_s)|}
\\
\left( \prod_{\alpha \in N_{1}(\chi_s)\setminus \Delta_{\para{P}}}
\zeta(l_{\alpha}^{+}(\chi_s)) l_{\alpha}^{+}(\chi_s)
\right)=& (2\zeta(2))^{|N_{1}(\chi_s)|-(n-1)}
\end{align*}
Hence,
\newpage
\begin{align}
G_{\para{P}}(\chi_s) &= \left( \prod_{\alpha \in N_{-1}(\chi_s)}
\zeta(l_{\alpha}^{+}(\chi_s)) l_{\alpha}^{+}(\chi_s)
\right) \label{h1}\\
&\times \prod_{\alpha \in N_{0}(\chi_s)} \zeta(l_{\alpha}^{+}(\chi_s)) \times  \prod_{\alpha \in N_{1}(\chi_s) \setminus \Delta_{\para{P}}} l_{\alpha}^{-}(\chi_s) \label{h2}\\
& \times 
\left( \prod_{\alpha \in\Phi^{+} \setminus N_{\pm1,0}(\chi_s)}
\zeta(l_{\alpha}^{+}(\chi_s)) l_{\alpha}^{+}(\chi_s)l_{\alpha}^{-}(\chi_s) \right) \label{h3}\\
& \times (-1)^{|N_{0,-1}(\chi_s)|} \times 2^{|N_{\pm1}(\chi_s)| -(n-1)|} \times 
\zeta(2)^{|N_{1}(\chi_s)| -(n-1)|}
\end{align}

In the case where $s\in \mathbb{R}$ it holds that 
$l_{\alpha}^{+}(\chi_{s}) \in \mathbb{R}$, and
hence  $\zeta(l_{\alpha}^{+}(\chi_s))\neq 0$
for $F=\mathbb{Q}$. As a result (\ref{h3}) is a non zero
and holomorphic.
The term (\ref{h1}) is holomorphic and non zero.
Hence the zero order of $G_{\para{P}}(\chi_s)$ is exactly the zero order of (\ref{h2}) which is  $|N_{1}(\chi_s)|- |N_{0}(\chi_s)|-(n-1)$. To sum up the order of the zero of $G_{\para{P}}(\chi_s)$ is exactly 
$|N_{1}(\chi_s)|- |N_{0}(\chi_s)|-(n-1)$.
\end{proof}
\begin{cor} \label{EiseConst}
For $F= \mathbb{Q}$, and for $s_0 \in \mathbb{R}$ there exists a constant $C \in \mathbb{C} ^{\ast}$ such that at $s=s_0$
$$E_{\para{B}}^{\#}(\chi_{s_0},g)= C \times \leadingterm{-d_{\para{P}}(\chi_{s_0})}^{\para{P}}
(f^{0},s_0,g).$$
\end{cor}
\begin{remk}
Under the assumption that the term (\ref{h3}) does not vanish, the assumption that $F=\mathbb{Q}$ and $s\in \mathbb{R}$ can be lifted. 
\end{remk}
Our  goal is to find out identities between various leading term. In order to do that, we introduce the notions of (positive) admissible data. As we can see it will help us to find the desire identities. 
 \begin{defn}
 The quintuple $(\para{P},s_0,\para{Q},t_0,\weylelement{})$ where $\para{P},\para{Q}$ are maximal parabolic subgroups,
 $s_0,t_0 \in \mathbb{R}$ and $\weylelement{}\in \weyl{G}$ is called \textbf{admissible data} if 
$$\weylelement{} ( \delta_{\para{P}}^{s_0-\frac{1}{2}} \otimes 
 \delta_{\para{B}}^{\frac{1}{2}}) = \delta_{\para{Q}}^{t_0-\frac{1}{2}} \otimes 
 \delta_{\para{B}}^{\frac{1}{2}}.$$
 If also $s_0,t_0$ are non negative numbers it is called  \textbf{positive admissible data}.
 \end{defn}
 \newpage
\begin{thm}\label{Weak_Thm}
Let $F= \mathbb{Q}$, and let $(\para{P},s_0,\para{Q},t_0,\weylelement{})$ be an admissible data. 
 Then there exists a constant $C \in \mathbb{C}^{\ast}$ such that 
 $$\leadingterm{-d_{\para{P}}
 (\chi_{s_0})}^{\para{P}}(f^{0},s_0,g) = C \times 
 \leadingterm{-d_{\para{Q}}
  (\chi_{t_0})}^{\para{Q}}(f^{0},t_0,g)$$
where $\chi_{s_0}=\delta_{\para{P}}^{s_0-\frac{1}{2}} \otimes   \delta_{\para{B}}^{\frac{1}{2}}$ and 
$\chi_{t_0}=\delta_{\para{Q}}^{t_0-\frac{1}{2}} \otimes   \delta_{\para{B}}^{\frac{1}{2}}$.
\end{thm}
\begin{proof}
$E_{\para{B}}^{\#}(\chi,g)$ is a $\weyl{G}$ invariant function. In particular,
$$E_{\para{B}}^{\#}(\chi_{t_0},g) = E_{\para{B}}^{\#}(\weylelement{}\chi_{s_0},g)= E_{\para{B}}^{\#}(\chi_{s_0},g).$$

From Corollary (\ref{EiseConst}) we deduce that there exists 
$c_1 \in \mathbb{C}^{\ast}$ (resp. $c_2 \in \mathbb{C}^{\ast}$) such that 
\begin{align*}
E_{\para{B}}^{\#}(\chi_{s_0},g)&= c_1 \times \leadingterm{-d_{\para{P}}(\chi_{s_0})}^{\para{P}}
(f^{0},s_0,g) \\
E_{\para{B}}^{\#}(\chi_{t_0},g)&= c_2 \times \leadingterm{-d_{\para{Q}}(\chi_{t_0})}^{\para{Q}}
(f^{0},t_0,g).
\end{align*}
Hence 
$$
\leadingterm{-d_{\para{P}}(\chi_{s_0})}^{\para{P}}
(f^{0},s_0,g) = \frac{c_2}{c_1} \times \leadingterm{-d_{\para{Q}}(\chi_{t_0})}^{\para{Q}}
(f^{0},t_0,g) .$$
\end{proof}
In Chapter \ref{ch:Technical lemmas} we will give an explicit formula for the  $C$ in Theorem
(\ref{Weak_Thm}).
In Chapter \ref{ch3:ids} we will find all the  positive admissible data.

\chapter{Refining The Basic Identity} 

\label{ch:Technical lemmas} 

\lhead{Chapter \ref{ch:Technical lemmas}. \emph{
Refining The Basic Identity
}}

Our main goal for this chapter is to write explicitly the constant appearing in Theorem (\ref{Weak_Thm}).
Let $\para{P}=\para{P}_i$ be a maximal parabolic subgroup of $G$. Let fix some notations:
\begin{enumerate}
\item
$\delta_{\para{P}_i}=b_{ii}\fundamental{i}$
\item
$R= \operatorname{Res}_{s=1}\zeta(s)$
\item
$\chi_{\para{P},s}=\delta_{\para{P}}^{s-\frac{1}{2}} \otimes \delta_{\para{B}}^{\frac{1}{2}}$
\item $\check{\alpha}=\sum_{i=1}^{n}n_{i}^{(\check{\alpha})}\check{\alpha}_i$
\item
$N_{\epsilon}(\chi) = \{ \alpha \in \Phi^{+} \: : \: \inner{\chi,\check{\alpha}}=\epsilon \}$ 
\item
$N_{\epsilon_1,\dots \epsilon_n }(\chi) =\bigcup_{i=1}^{n} N_{\epsilon_i} (\chi)$
\item
$B_{\epsilon}(\chi) = \{ \alpha \in \Phi \: : \: \inner{\chi,\check{\alpha}}=\epsilon \} $
\item
$B_{\epsilon_1,\dots \epsilon_n }(\chi) =\bigcup_{i=1}^{n} B_{\epsilon_i} (\chi)$
\item
$h_1(\chi_{\para{P},s},s_0)= \prod_{\alpha \in N_{-1}(\chi_{\para{P},s_0})}
\zeta(l_{\alpha}^{+}(\chi_{\para{P},s})) l_{\alpha}^{+}(\chi_{\para{P},s}) $
\item
$h_2(\chi_{\para{P},s},s_0)=
\prod_{\alpha \in N_{0}(\chi_{\para{P},s_0})} \zeta(l_{\alpha}^{+}(\chi_{\para{P},s})) \times  \prod_{\alpha \in N_{1}(\chi_{\para{P},s_0}) \setminus \Delta_{\para{P}}} l_{\alpha}^{-}(\chi_{\para{P},s})$
\item
$h_{3,\para{P}}(\chi_{\para{P},s},s_0)=
\prod_{\alpha \in\Phi^{+} \setminus N_{\pm1,0}(\chi_{\para{P},s_0})}
\zeta(l_{\alpha}^{+}(\chi_{\para{P},s})) l_{\alpha}^{+}(\chi_{\para{P},s})l_{\alpha}^{-}(\chi_{\para{P},s})$
\item
$\weylelement{\para{P}_i}$ is the longest element in $\weyl{\para{P}_i}$.
\item
$\mathcal{F}_{\alpha}^{-}=\{\chi \in \mathfrak{a}^{\ast}_{\mathbb{C}} \: : l_{\alpha}^{-}(\chi)=0 \}$
\item
$\mathcal{F}_{\hat{r}}=\overset{n}{\underset{
\tiny{\begin{aligned}
	i&=1\\
	i& \neq r
	\end{aligned}}}
{\bigcap}}\mathcal{F}_{\alpha_i}$
\end{enumerate}
 \begin{remk}
 When there is no risk of confusion, we  shall omit the subscript  $\para{P}$.
 \end{remk}
 Let state the main result of this chapter.
\begin{thm*} \label{main_thm1}
Let $(\para{P}_i,s_0,\para{P}_j,t_0,\weylelement{})$ be an admissible data. For $F=\mathbb{Q}$ 
the constant $c$ appears in Theorem (\ref{Weak_Thm}) equals to
 $$\frac{\epsilon_{t_0}}{\epsilon_{s_0}}\times 
 \prod_{\tiny {\begin{aligned}
   		\alpha \in& \Phi^{+} \setminus N_{\pm1,0}(\chi_{\para{P}_i,s_0})\\
   		\weylelement{}\alpha & \in \Phi^{-}
   		\end{aligned}}}
   \frac{
   	\zeta (\left \langle \chi_{\para{P}_i,s_0} ,\check{\alpha} \right \rangle)
   }{\zeta(
   \left \langle \chi_{\para{P}_i,s_0} ,\check{\alpha} \right \rangle+1)
 }
  \times 
 \left(
 \frac{R}{\zeta(2)} \right)^{d}\times \frac{A_{\para{P}_j}}{A_{\para{P}_i}}
 $$
 where 
 \begin{align*}
A_{\para{P}_i}&=
Res_{\mathcal
{F}_{\hat{i}}} C_{\weylelement{\para{P}_i}} (\chi_{\para{P}_i,s_0}) & A_{\para{P}_j}&=
Res_{\mathcal
{F}_{\hat{j}}} C_{\weylelement{\para{P}_j}} (\chi_{\para{P}_j,t_0})\\
d=&|N_{-1}(\chi_{\para{P}_j,t_0})|-|N_{-1}(\chi_{\para{P}_i,s_0})| \\
\epsilon_{s_0} =&
	b_{ii}^{|N_1(\chi_{s_0})| - |N_0(\chi_{s_0})|-(n-1)}
	\times
	\prod_{
\alpha \in N_{0}(\chi_{s_0})
	}\frac{1}{n_{i}^{(\check{\alpha})}
	}
	\times
	\prod_{
			\alpha \in N_{1}(\chi_{s_0}) \setminus \Delta_{\para{P}} 	
		}{n_{i}^{(\check{\alpha})}}
 \end{align*}
\end{thm*}
In order to prove it we will prove a sequence of Lemmas. \\
 With the notation of 
 Corollary (\ref{Lemma1}) it holds that 
  $$E_{\para{B}}^{\#}(\chi_{\para{P},s},g) = A_{\weylelement{\para{P}}} \times
  G_{\para{P}}(\chi_{\para{P},s})   \times \prod_{\alpha \in \Delta_{\para{P}}} \zeta(l_{\alpha}^{+}(\chi_{\para{P},s}))l_{\alpha}^{+}(\chi_{\para{P},s}) \times E_{\para{P}}(f^{0},s,g)
  $$
  where 
  $G_{\para{P}}(\chi_{\para{P},s})= \prod_{\alpha \in \Phi^{+} \setminus \Delta_{\para{P}}} \zeta(l_{\alpha}^{+}(\chi_{\para{P},s}))
  l_{\alpha}^{+}(\chi_{\para{P},s})
  l_{\alpha}^{-}(\chi_{\para{P},s})$
  and $A_{\weylelement{\para{P}}}=Res_{\mathcal{F}_{\hat{i}}} C_{\weylelement{\para{P}}}(\chi_{\para{P},s})$.
 
 Moreover, Theorem (\ref{Upper}) shows that  $G_{\para{P}}(\chi_{s})$ can be written as
 \begin{align*}
 G_{\para{P}}(\chi_s) =& h_1(\chi_s,s_0) \times h_2(\chi_s,s
_0) \times h_3 (\chi_s,s_0)
\\
\times& (\zeta(2))^{|N_1(\chi_{s})|-(n-1)} \times (-1)^{|N_{0,-1}(\chi_s)|}
\times 
2^{|N_{\pm1}(\chi_s)|  -(n-1)}.
\end{align*}
\begin{lemm}\label{Corollary-1}
Let $\para{P}=\para{P}_i$ be a maximal parabolic subgroup of $G$. 
 Then 
\begin{enumerate}
	\item
	$\operatorname{lim}_{s \rightarrow s_0} h_1(\chi_s ,s_0)
 =(-R)^{|N_{-1}(\chi_{s_0})|}$.
	\item
	$h_{2}(\chi_s,s_0)$ has a zero of order 
	$|N_{1}(\chi_{s_0})| - |N_{0}(\chi_{s_0})|-(n-1)$ at $s=s_0$. Moreover, its leading coefficient is equal to  
$\epsilon_{s_0} \times R^{|N_{0}(\chi_{s_0})|}	$
where 
\begin{align*}
\epsilon_{s_0} =&
	b_{ii}^{|N_1(\chi_{s_0})| - |N_0(\chi_{s_0})|-(n-1)}
	\times
	\prod_{
\alpha \in N_{0}(\chi_{s_0})
	}\frac{1}{n_{i}^{(\check{\alpha})}
	}
	\times
	\prod_{
			\alpha \in N_{1}(\chi_{s_0}) \setminus \Delta_{\para{P}} 	
		}{n_{i}^{(\check{\alpha})}}
	\end{align*}
\end{enumerate}
\end{lemm}
\begin{proof}
\begin{enumerate}
\item
It will be enough to show that if $\alpha \in N_{-1}(\chi_{s_0})$ then $$\zeta(l_{\alpha}^{+}(\chi_{s_0}))
l_{\alpha}^{+}(\chi_{s_0}) = -R.$$
But it is just the residue of $\zeta(s)$ at $s=0$ so we are done.
\item
Note that $h_2(\chi_s,s_0)$ is a meromorphic function (it is a product of such functions) therefore it admits a Laurent expansion around $s=s_0$. 
For every $\alpha \in \Phi^{+}$ it holds that 
$l_{\alpha}^{+}(\chi_s)$ is a linear function in $s$. and  $$l_{\alpha}^{+}(\chi_s)=b_{ii}n_{i}^{(\check{\alpha})}s+ m(\alpha)$$
where  $m(\alpha)$ is a constant depends in $\alpha$.

In particular, for every $\alpha \in N_{0}(\chi_{s_0})$ the element 
$\zeta(l_{\alpha}^{+}(\chi_s))$  can be written as $\zeta(F(s,\alpha))$ where 
$$F(s,\alpha)=b_{ii}n_{i}^{(\check{\alpha})}s+ m(\alpha).$$
By definition  $F(s_0,\alpha)=1$. Therefore $\zeta(F(s,\alpha))$ admits a simple pole at $s=s_0$. The  Laurent expansion of $\zeta(F(s,\alpha))$ around $s_0$ is 
$$\zeta(F(s,\alpha)) = \sum_{j=-1}^{\infty} a_j(F(s,\alpha)-1)^j$$
where $a_{-1}=R$. 
Moreover,
\begin{align*}
F(s,\alpha)-1 &= F(s,\alpha)- F(s_0,\alpha) \\
&= b_{ii}n_{i}^{(\check{\alpha})}s+ m(\alpha) - ( b_{ii}n_{i}^{(\check{\alpha})}s_0+ m(\alpha)) \\
&= b_{ii}n_{i}^{(\check{\alpha})}(s-s_0).
\end{align*}
Therefore 
$$\zeta(F(s,\alpha))= \sum_{j=-1}^{\infty} a_j(b_{ii}n_{i}^{(\check{\alpha})})^{j} (s-s_0)^{j}.$$
The leading term  around $s=s_0$ of
$$\prod_{\alpha \in N_{0}(\chi_{s_{0}})}
\zeta(l_{\alpha}^{+}(\chi_{s}))$$ is the product of the leading term of every $\zeta(l_{\alpha}^{+}(\chi_{s}))$ around $s=s_0$.
Therefore, the product 
contributes a pole of order $|N_{0}(\chi_{s_0})|$ and its leading coefficient equals to 
$$\prod_{\alpha \in N_{0}(\chi_{s_0})
	}\frac{1}{n_{i}^{(\check{\alpha})}
	}
	\times \left(\frac{R}{b_{ii}}\right)^{|N_{0}(\chi_{s_0})|}.
 $$
The product 
$$\prod_{\tiny{
			\begin{aligned}
			\alpha \in N_{1}(\chi_{s_0}) \\
			\alpha \not \in \Delta_{\para{P}}\\
			\end{aligned}}}l_{\alpha}^{-}(\chi_s)
$$
has a zero of order $|N_{1}(\chi_{s_0}) \setminus \Delta_{\para{P}}| = |N_{1}(\chi_{s_0})| -(n-1)$.
Since every $\alpha \in \Delta_{\para{P}}$ satisfies  
$\inner{\chi_{s},\check{\alpha}}=1$.
By a similar argument,
its leading coefficient equals to 
$$b_{ii}^{|N_{1}(\chi_{s_0})|-(n-1)}
\times 
	\prod_{\tiny{
				\begin{aligned}
				\alpha \in N_{1}(\chi_{s_0}) \\
				\alpha \not \in \Delta_{\para{P}}\\
				\end{aligned}}				
				}
				n_{i}^{(\check{\alpha})}
 .$$
Hence, the leading coefficient of $h_2(\chi_s,s_0)$ around $s=s_0$  equals to
$$\epsilon_{s_0} \times R^{|N_{0}(\chi_{s_0})|}$$
and it admits a zero of order 
$$N_{1}(\chi_s) -N_{0}(\chi_s) -(n-1).$$
\end{enumerate}
\end{proof}
\begin{cor}\label{redEis}
For the notations as above and for $F=\mathbb{Q}$ it holds that 
\begin{align*}
E_{\para{B}}^{\#}(\chi_{s_0},g) =&
h_{3}(\chi_{s_0},s_0)\\
\times &
\epsilon_{s_0} \times (-1)^{|N_{0}(\chi_{s_0})|} \times 
2^{|N_{\pm1}(\chi_{s_0})|}
\times \zeta(2) ^{|N_{1}(\chi_{s_0})|}\\
\times &
R^{|N_{0,-1}(\chi_{s_0})|}
\times 
A_{\weylelement{_\para{P}}}\times 
\leadingterm{-d_{\para{P}}(\chi_{s_0})  }^{\para{P}}(f^{0},s_0,g).
\end{align*}
\end{cor}
Let $(\para{P},s_0,\para{Q},t_0,\weylelement{})$ be an admissible data. Our goal is to determine the constant for Theorem (\ref{Weak_Thm}). For this purpose, we shall evaluate the quotient $\frac{h_{3,\para{Q}}(\chi_{\para{Q},t_0})}
{h_{3,\para{P}}(\chi_{\para{P},s_0})}$. We shall need the following lemmas.
\begin{lemm}\label{Lemma-T1}
Let $\chi$ be an unramified character of $T$. For every $\epsilon\in \mathbb{C}$ and $\weylelement{} \in \weyl{G} $ it holds that:
\begin{enumerate}
\item\label{Lemma-T1-1}
$B_{\epsilon}(\chi)=-B_{-\epsilon}(\chi)$.
\item\label{Lemma-T1-2}
	$B_{\epsilon}(\weylelement{}\chi)=\weylelement{}B_{\epsilon}(\chi)$.
\item \label{Lemma-T1-3}
$|N_{\pm\epsilon}(\chi)|=|N_{\pm\epsilon}(\weylelement{}\chi)|$.
\item \label{Lemma-T1-4}
$|N_{1}(\chi)| - | N_{-1} (\weylelement{}\chi) | = |N_{1}(\weylelement{}\chi)| -
|N_{-1}(\chi)|$.
\item \label{Lemma-T1-5}
$|N_{0}(\chi)| = |N_{0}(\weylelement{}\chi)|$.
\end{enumerate}
\end{lemm}
\begin{proof}
(\ref{Lemma-T1-1}) and (\ref{Lemma-T1-2}) are immediate from the definition and the $\weyl{G}$ invariance of the paring. (\ref{Lemma-T1-4}) and (\ref{Lemma-T1-5})  follow from (\ref{Lemma-T1-3}). Thus we need to prove only (\ref{Lemma-T1-3}).

Note that
\begin{align*}
 B_{\epsilon}(\chi) &= \{ \alpha \in \Phi \: : \: \inner{\chi,\check{\alpha}} = \epsilon \} \\
 &= \{ \alpha \in \Phi^{+} \: : \: \inner{\chi,\check{\alpha}} = \epsilon \} \cup 
 \{ \alpha \in \Phi^{-} \: : \: \inner{\chi,\check{\alpha}} = \epsilon \} \\
 &= N_{\epsilon}(\chi) \cup \{ -\alpha \in \Phi^{+} \: : \: \inner{\chi,-\check{\alpha}} = -\epsilon \} \\
 &= N_{\epsilon}(\chi) \cup  - N_{-\epsilon}(\chi).
 \end{align*}
 Hence, 
 \begin{align*}
 |B_{\epsilon}(\chi)| &=  |N_{\epsilon}(\chi)| + |-N_{-\epsilon}(\chi)| =  |N_{\epsilon}(\chi)| + |N_{-\epsilon}(\chi)|.
\end{align*}
By the same argument 
$$|B_{\epsilon}(\weylelement{}\chi)| =    |N_{\epsilon}(\weylelement{}\chi)| + |N_{-\epsilon}(\weylelement{}\chi)|.$$
From (\ref{Lemma-T1-2}) and the injectivity of 
 $\weylelement{}$  it holds that $|B_{\epsilon}(\weylelement{}\chi)|=|B_{\epsilon}(\chi)|$.
Thus, we are done.
\end{proof}
\begin{lemm}\label{Lem}
For every 
$\alpha \in \Phi$  let $F_{\alpha} : \mathfrak{a}^{\ast}_{\mathbb{C}} \rightarrow \mathbb{C}$ that satisfies $F_{\alpha}(\weylelement{} \chi) = 
F_{\weylelement{}^{-1}\alpha}(\chi)$ for every $\chi \in \mathfrak{a}^{\ast}_{\mathbb{C}},\weylelement{}\in \weyl{G}$
Assume that $F_{\alpha}(\chi) \neq 0 $ for every $\alpha \in \Phi$ then
\begin{align*}
\frac{ \prod_{\alpha \in \Phi^{+} \setminus N_{\pm1,0}(\weylelement{} \chi)} F_{\alpha}(\weylelement{} \chi)}
{ \prod_{\alpha \in \Phi^{+} \setminus N_{\pm1,0}(\chi)} F_{\alpha}(\chi)}= \prod_{
\tiny {\begin{aligned}
    \alpha \in&  \Phi^{+} \setminus N_{\pm1,0}(\chi) \\
     \weylelement{}\alpha &\in \Phi^{-}
    \end{aligned}}}
\frac{F_{-\alpha}(\chi)}{F_{\alpha}(\chi)}.
\end{align*}
\end{lemm}
\begin{proof}
\begin{align*}
\prod_{\alpha \in \Phi^{+} \setminus N_{\pm1,0}(\weylelement{}\chi)}F_{\alpha}(\weylelement{}\chi)=&\prod_{\tiny {\begin{aligned}
 \alpha &\in  \Phi^{+} \\
 \alpha \not \in& B_{\pm1,0}(\weylelement{}\chi)
 \end{aligned}}} F_{\alpha}(\weylelement{}\chi)  
 \overset{\#1}{=}
 \prod_{\tiny {\begin{aligned}
  \alpha &\in  \Phi^{+} \\
  \alpha \not \in& \weylelement{}B_{\pm1,0}(\chi)
  \end{aligned}}} F_{\weylelement{}^{-1}\alpha}(\chi)  \\
 =
 &\prod_{\tiny {\begin{aligned}
  \alpha &\in  \Phi^{+} \\
  \weylelement{}^{-1}\alpha \not \in& B_{\pm1,0}(\chi)
  \end{aligned}}} F_{\weylelement{}^{-1}\alpha}(\chi)  
 =
 \prod_{\tiny {\begin{aligned}
  \alpha &\in  \Phi^{+} \\
  \beta =&\weylelement{}^{-1}\alpha \\
   \beta \not \in& B_{\pm1,0}(\chi)
  \end{aligned}}} F_{\beta}(\chi)  \\
 =&\prod_{\tiny {\begin{aligned}
  \alpha &\in  \Phi^{+} \\
  \beta =&\weylelement{}^{-1}\alpha \\
   \beta \not \in& B_{\pm1,0}(\chi) \\
   \beta &\in  \Phi^{+} \\
  \end{aligned}}} F_{\beta}(\chi)  \times 
   \prod_{\tiny {\begin{aligned}
    \alpha &\in  \Phi^{+} \\
    \beta =&\weylelement{}^{-1}\alpha \\
     \beta \not \in& B_{\pm1,0}(\chi) \\
     \beta &\in  \Phi^{-} \\
    \end{aligned}}} F_{\beta}(\chi) \\
 =&\prod_{\tiny {\begin{aligned}
  \weylelement{}\beta &\in  \Phi^{+} \\
   \beta \not \in& B_{\pm1,0}(\chi) \\
   \beta &\in  \Phi^{+} \\
  \end{aligned}}} F_{\beta}(\chi)  \times 
   \prod_{\tiny {\begin{aligned}
    \weylelement{}\beta &\in  \Phi^{+} \\
     \beta \not \in& B_{\pm1,0}(\chi) \\
     \beta &\in  \Phi^{-} \\
    \end{aligned}}} F_{\beta}(\chi) \\
 =&\prod_{\tiny {\begin{aligned}
  \weylelement{}\beta &\in  \Phi^{+} \\
   \beta \not \in& B_{\pm1,0}(\chi) \\
   \beta &\in  \Phi^{+} \\
  \end{aligned}}} F_{\beta}(\chi)  \times 
   \prod_{\tiny {\begin{aligned}
    -\weylelement{}\beta &\in  \Phi^{-} \\
     -\beta \not \in& -B_{\pm1,0}(\chi) \\
     -\beta &\in  \Phi^{+} \\
    \end{aligned}}} F_{\beta}(\chi) \\    
 \overset{\#2}{=}&\prod_{\tiny {\begin{aligned}
  \weylelement{}\beta &\in  \Phi^{+} \\
   \beta \not \in& B_{\pm1,0}(\chi) \\
   \beta &\in  \Phi^{+} \\
  \end{aligned}}} F_{\beta}(\chi)  \times 
   \prod_{\tiny {\begin{aligned}
    \weylelement{}\gamma &\in  \Phi^{-} \\
     \gamma \not \in& B_{\pm1,0}(\chi) \\
     \gamma &\in  \Phi^{+} \\
    \end{aligned}}} F_{-\gamma}(\chi) 
\end{align*}
where 
$\#1$ is due to (\ref{Lemma-T1}) item (\ref{Lemma-T1-2}) and 
$\#2$ is due to
Lemma (\ref{Lemma-T1}) item (\ref{Lemma-T1-1}) since
\begin{align*}
-B_{\pm1,0}(\chi) &= -B_{1}(\chi)\cup -B_{0}(\chi) \cup -B_{-1}(\chi) \\
&=B_{-1}(\chi)\cup B_{0}(\chi) \cup B_{1}(\chi) \\
&=B_{\pm1,0}(\chi).
\end{align*}
Hence,
\begin{align*}
\frac{ \prod_{\alpha \in \Phi^{+} \setminus N_{\pm1,0}(\weylelement{} \chi)} F_{\alpha}(\weylelement{} \chi)}
{ \prod_{\alpha \in \Phi^{+} \setminus N_{\pm1,0}(\chi)} F_{\alpha}(\chi)}= \prod_{
\tiny {\begin{aligned}
    \alpha \in&  \Phi^{+} \setminus N_{\pm1,0}(\chi) \\
     \weylelement{}\alpha &\in \Phi^{-}
    \end{aligned}}}
\frac{F_{-\alpha}(\chi)}{F_{\alpha}(\chi)}.
\end{align*}
\end{proof}
\begin{cor}\label{corr}
Let $(\para{P},s_0,\para{Q},t_0,\weylelement{})$ be an admissible data. For $F=\mathbb{Q}$ it holds that
\begin{align}
\frac{h_{3,\para{Q}}(\chi_{\para{Q},t},t_0)}{h_{3,\para{P}}(\chi_{\para{P},s},s_0)}&=\prod_{
\tiny {\begin{aligned}
    \alpha \in&  \Phi^{+} \setminus N_{\pm1,0}(\chi_{\para{P},s_0}) \\
     \weylelement{}\alpha &\in \Phi^{-}
    \end{aligned}}}
\frac{\zeta(\inner{\chi_{\para{P},s_0},\check{\alpha}})}{
\zeta(\inner{\chi_{\para{P},s_0},\check{\alpha}}+1)
}.
\end{align}
\end{cor}
\begin{proof}
Notice that $l_{\alpha}^{\pm}(\weylelement{}\chi)=l_{\weylelement{}^{-1}\alpha}^{\pm}(\chi)$. Hence, we can apply Lemma (\ref{Lem}) with the function 
$$F_{\alpha}(\chi) = l_{\alpha}^{+}(\chi) l_{\alpha}^{-}(\chi) \zeta(l_{\alpha}^{+}(\chi)).$$
Since $(\para{P},s_0,\para{Q},t_0,\weylelement{})$ is an admissible data, it holds that $\weylelement{} (\chi_{\para{P},s_0}) = \chi_{\para{Q},t_0}$.
Hence,
$$ \frac{h_{3,\para{Q}}(\chi_{\para{Q},t},t_0)}{h_{3,\para{P}}(\chi_{\para{P},s},s_0)} =
\frac{ \prod_{\alpha \in \Phi^{+} \setminus N_{\pm1,0}(\chi_{\para{Q},t_0})} F_{\alpha}(\chi_{\para{Q},t_0})}
{ \prod_{\alpha \in \Phi^{+} \setminus N_{\pm1,0}(\chi_{\para{P},s_0})} F_{\alpha}(\chi_{\para{P},s_0})}= \prod_{
\tiny {\begin{aligned}
    \alpha \in&  \Phi^{+} \setminus N_{\pm1,0}(\chi_{\para{P},s_0}) \\
     \weylelement{}\alpha &\in \Phi^{-}
    \end{aligned}}}
\frac{F_{-\alpha}(\chi_{\para{P},s_0})}{F_{\alpha}(\chi_{\para{P},s_0})}.
$$
Notice that
\begin{align*}
l_{-\alpha}^{+}(\chi) l_{-\alpha}^{-}(\chi)  =
l_{\alpha}^{+}(\chi) l_{\alpha}^{-}(\chi).
\end{align*}
Moreover, by functional equation of the $\zeta$ function it holds that 
$\zeta(l_{-\alpha }(\chi)) =\zeta(\inner{\chi,\check{\alpha}})$.
Hence 
$$\frac{F_{-\alpha}(\chi)}{F_{\alpha}(\chi)} = \frac{\zeta(\inner{\chi,\check{\alpha}})}{\zeta(\inner{\chi,\check{\alpha}}+1)}
.$$
So the Corollary follows.
\end{proof}
\newpage
\begin{thm} \label{main_thm}
Let $(\para{P}_i,s_0,\para{P}_j,t_0,\weylelement{})$ be an admissible data. For $F=\mathbb{Q}$ 
there exists $A \in \mathbb{C}$ such that
$$\leadingterm{-d_{\para{P}_i}(\chi_{\para{P}_i,s_0})}^{\para{P}_i}(f^{0}_{\para{P}_i},s_0,g)=A \times \leadingterm{
-d_{\para{P}_j}(\chi_{\para{P}_j,t_0})
} 
 ^{\para{P}_j}(f^{0}_{\para{P}_j},t_0,g)
 \quad \forall g \in G.
 $$
 Moreover,
 $$A= \frac{\epsilon_{t_0}}{\epsilon_{s_0}}\times 
 \prod_{\tiny {\begin{aligned}
   		\alpha \in& \Phi^{+} \setminus N_{\pm1,0}(\chi_{\para{P}_i,s_0})\\
   		\weylelement{}\alpha & \in \Phi^{-}
   		\end{aligned}}}
   \frac{
   	\zeta (\left \langle \chi_{\para{P}_i,s_0} ,\check{\alpha} \right \rangle)
   }{\zeta(
   \left \langle \chi_{\para{P}_i,s_0} ,\check{\alpha} \right \rangle+1)
 }
  \times 
 \left(
 \frac{R}{\zeta(2)} \right)^{d}\times \frac{A_{\para{P}_j}}{A_{\para{P}_i}}
 $$
 where 
 \begin{align*}
A_{\para{P}_i}&=
Res_{\mathcal
{F}_{\hat{i}}} C_{\weylelement{\para{P}_i}} (\chi_{\para{P}_i,s_0}) & A_{\para{P}_j}&=
Res_{\mathcal
{F}_{\hat{j}}} C_{\weylelement{\para{P}_j}} (\chi_{\para{P}_j,t_0})\\
d=&|N_{-1}(\chi_{\para{P}_j,t_0})|-|N_{-1}(\chi_{\para{P}_i,s_0})|
 \end{align*}
\end{thm}
\begin{proof}
Using Corollary (\ref{redEis}) we obtained  that:
\begin{align*}
E_{\para{B}}^{\#}(\chi_{\para{P},s_0},g) =&
h_{3}(\chi_{\para{P}_i,s_0},s_0) \\
\times &
\epsilon_{s_0} \times (-1)^{|N_{0}(\chi_{\para{P}_i,s_0})|} \times 
2^{|N_{\pm1}(\chi_{\para{P}_i,s_0})|}
\times \zeta(2) ^{|N_{1}(\chi_{\para{P}_i,s_0})|}\\
\times &
R^{|N_{0,-1}(\chi_{\para{P}_i,s_0})|}
\times 
A_{\para{P}_i}\times 
\leadingterm{-d_{\para{P}}(\chi_{\para{P}_i,s_0})  }^{\para{P}_i}(f^{0}_{\para{P}_i},s_0,g).
\end{align*}
By a similar argument we get :
\begin{align*}
E_{\para{B}}^{\#}(\chi_{\para{P}_j,t_0},g) =&
h_{3}(\chi_{\para{P}_j,t_0},t_0)\\
\times &
\epsilon_{t_0} \times (-1)^{|N_{0}(\chi_{\para{P}_j,t_0})|} \times 
2^{|N_{\pm1}(\chi_{\para{P}_j,t_0})|}
\times \zeta(2) ^{|N_{1}(\chi_{\para{P}_j,t_0})|}\\
\times &
R^{|N_{0,-1}(\chi_{\para{P}_j,t_0})|}
\times 
A_{\para{P}_j}\times 
\leadingterm{-d_{\para{P}_j}(\chi_{\para{P}_j,t_0})  }^{\para{P}_j}(f^{0}_{\para{P}_j},t_0,g).
\end{align*}
Since $E_{\para{B}}^{\#}$ is $\weyl{G}$ invariant, both terms are equal. So, we get:
\begin{align*}
\leadingterm{-d_{\para{P}_i}(\chi_{\para{P}_i,s_0})}^{\para{P}_i}(f^{0}_{\para{P}_i},s_0,g)&=
\frac{A_{\para{P}_j}}{A_{\para{P}_i}} \times 
\frac{\epsilon_{t_0}}{\epsilon_{s_0}} \times \frac{h_{3}(\chi_{\para{P}_j,t_0},t_0)}
{h_{3}(\chi_{\para{P}_i,s_0},s_0)} \\
&\times 
2^{|N_{\pm1} (\chi_{\para{P}_j,t_0})|- |N_{\pm1} (\chi_{\para{P}_i,s_0})|} \times (-1)^{|N_{0} (\chi_{\para{P}_j,t_0})|- |N_{0} (\chi_{\para{P}_i,s_0})|}\\
&\times 
\zeta(2)^{|N_{1} (\chi_{\para{P}_j,t_0})|- |N_{1} (\chi_{\para{P}_i,s_0})|} \times R^{|N_{0,-1} (\chi_{\para{P}_j,t_0})|- |N_{0,-1} (\chi_{\para{P}_i,s_0})|} \\
&\times
\leadingterm{-d_{\para{P}_j}(\chi_{\para{P}_j,t_0})  }^{\para{P}_j}(f^{0}_{\para{P}_j},t_0,g).
\end{align*}
 Using Lemma (\ref{Lemma-T1}) it holds that:
 \begin{align*}
 (-1)^{|N_{\pm1}(\chi_{\para{P}_j,t_0})| - N_{\pm1}(\chi_{\para{P}_i,s_0})|} &= 1, &
 2^{|N_{\pm1}(\chi_{\para{P}_j,t_0})| - N_{\pm1}(\chi_{\para{P}_i,s_0})|} &= 1, &
R^{|N_{0}(\chi_{\para{P}_j,t_0})| - |N_{0}(\chi_{\para{P}_i,s_0})|} &=1,
   \end{align*}
Hence it holds that:
\begin{align*}
\leadingterm{-d_{\para{P}_i}(\chi_{\para{P}_i,s_0})}^{\para{P}_i}(f^{0}_{\para{P}_i},s_0,g)&=
\frac{A_{\para{P}_j}} {A_{\para{P}_i}} \times 
\frac{\epsilon_{t_0}}{\epsilon_{s_0}} \times \frac{h_{3}(\chi_{\para{P}_j,t_0},t_0)}
{h_{3}(\chi_{\para{P}_i,s_0},s_0)} \\
&\times 
\zeta(2)^{|N_{1} (\chi_{\para{P}_j,t_0})|- |N_{1} (\chi_{\para{P}_i,s_0})|} \times R^{|N_{-1} (\chi_{\para{P}_j,t_0})|- |N_{-1} (\chi_{\para{P}_i,s_0})|} \\
&\times
\leadingterm{-d_{\para{P}_j}(\chi_{\para{P}_j,t_0})  }^{\para{P}_j}(f^{0}_{\para{P}_j},t_0,g).
\end{align*}
Using Lemma (\ref{Lemma-T1}) again we obtain
\begin{align*}
\zeta(2)^{|N_{1} (\chi_{\para{P}_j,t_0})|- |N_{1} (\chi_{\para{P}_i,s_0})|} \times R^{|N_{-1} (\chi_{\para{P}_j,t_0})|- |N_{-1} (\chi_{\para{P}_i,s_0})|}
=& \left(\frac{R}{\zeta(2)}\right)^{d}.
\end{align*}
Corollary (\ref{corr}) implies that 
\begin{align*}
\frac{h_{3}(\chi_{\para{P}_j,t_0},t_0)}
{h_{3}(\chi_{\para{P}_i,s_0},s_0)} & =
\prod_{
\tiny {\begin{aligned}
    \alpha \in&  \Phi^{+} \setminus N_{\pm1,0}(\chi_{\para{P}_i,s_0}) \\
     \weylelement{}\alpha &\in \Phi^{-}
    \end{aligned}}}
\frac{\zeta(\inner{\chi_{\para{P}_i,s_0},\check{\alpha}})}{
\zeta(\inner{\chi_{\para{P}_i,s_0},\check{\alpha}}+1)
}.
\end{align*}
Hence,
\begin{align*}
 \leadingterm{-d_{\para{P}_i}(\chi_{\para{P}_i,s_0})}^{\para{P}_i}(f^{0}_{\para{P}_i},s_0,g)&=
\frac{A_{\para{P}_j}}{A_{\para{P}_i}} \times 
\frac{\epsilon_{t_0}}{\epsilon_{s_0}}
\\
&\times 
\left(\frac{R}{\zeta(2)}\right)^{d} \times 
\prod_{
\tiny {\begin{aligned}
    \alpha \in&  \Phi^{+} \setminus N_{\pm1,0}(\chi_{\para{P}_i,s_0}) \\
     \weylelement{}\alpha &\in \Phi^{-}
    \end{aligned}}}
\frac{\zeta(\inner{\chi_{\para{P}_i,s_0},\check{\alpha}})}{
\zeta(\inner{\chi_{\para{P}_i,s_0},\check{\alpha}}+1)
}\\
&\times 
\leadingterm{-d_{\para{P}_j}(\chi_{\para{P}_j,t_0})}^{\para{P}_j}(f^{0}_{\para{P}_j},t_0,g).
\end{align*}
 Observe that 
 \begin{align*}
 d_{\para{P}_i}(\chi_{\para{P}_i,s_0})-d=&
 |N_{1}(\chi_{\para{P}_i,s_0})| -
 |N_{0}(\chi_{\para{P}_i,s_0})|-(n-1)
 - |N_{-1}(\chi_{\para{P}_j,t_0})| + |N_{-1}(\chi_{\para{P}_i,s_0})|\\
 =&
 |N_{1}(\chi_{\para{P}_i,s_0})| -
 |N_{0}(\chi_{\para{P}_i,s_0})|-(n-1)
 + |N_{1}(\chi_{\para{P}_j,t_0})| - |N_{1}(\chi_{\para{P}_i,s_0})| \\
 =& |N_{1}(\chi_{\para{P}_j,t_0})| -|N_{0}(\chi_{\para{P}_i,s_0})| -(n-1) =d_{\para{P}_j}(\chi_{\para{P}_j,t_0}).
 \end{align*}
\end{proof}

\chapter{Identities} 

\label{ch3:ids} 

\lhead{Chapter \ref{ch3:ids}. \emph{Identities}} 

In this chapter we explore the identities between leading terms of various Eisenstein series associated to various maximal parabolic subgroups. The motivation comes from the famous Siegel Weil Formula.

Let $W$ be a symplectic space of dimension $2n$ and $V$ be an orthogonal space of dimension $m$. A. Weil in his seminal paper \cite{We}
has evaluated an average of the  theta function on the group $\widetilde Sp(W\otimes V)$ over $O(V)$ as a special value
of the degenerate Eisenstein on $Sp(W)$ series associated to the Siegel parabolic subgroup under the assumption that 
the integral converges. The convergence condition has been lifted in the paper of Kudla and Rallis \cite{KR} 
who considered the regularized average of the theta function above. The regularized integral is naturally related 
to a residue of a non-Siegel Eisenstein series. However it was essential to find the relation between the average integral
 the residue of the Siegel  Eisenstein series, whose analytic properties can be further related to 
the standard L-functions by \cite{GPSR}.  So Kudla and Rallis have found the identity between the leading terms 
of the Siegel and non-Siegel Eisenstein series.  Later their work has been generalized for all classical groups. 

Jiang in \cite{Dih1} has considered various identities for the leading terms of degenerate spherical Eisenstein series on
symplectic groups, for various maximal parabolic subgroups. 

Finally Ginzburg and Jiang  in \cite{G2_Ginz} have considered similar identities for the exceptional group $G_2. $

\paragraph*{The kind of identities we expect to see} 
One type of the identities follows from the functional equation satisfied  by 
Eisenstein series that relate the Eisenstein series associated to
maximal parabolic subgroups whose Levi parts are conjugated by elements of the Weyl group.

However, there exists another type of identities that relate
 between the Eisenstein series associated to maximal parabolic subgroups whose Levi parts are not conjugated by elements of the Weyl group.
The simplest  example is the trivial representation that can be realized as residue of Eisenstein series associated to any maximal parabolic subgroup 
$\para{P}$ at $s=\frac{1}{2}$. This kind of identities is sometimes called  Siegel-Weil identities. 

Let us briefly describe the context of the chapter.
The chapter begins with the proof of the following combinat results 
\begin{thm}\label{we}
For every pair of maximal parabolic subgroup
$\para{P}_i,\para{P}_j$ there exist $\weylelement{} \in \weyl{G}$ and $s_0,t_0 \in \mathbb{R}$ such that 
$(\para{P}_i,s_0,\para{P}_j,t_0,
\weylelement{})$ is an admissible data.
\end{thm}
During the proof, we give an uniformly formula for those data.
As a corollary in the case where $-1 \in \weyl{G}$ it holds that :
\begin{cor*}
For every pair of maximal parabolic subgroup
$\para{P}_i,\para{P}_j$ there exist $\weylelement{} \in \weyl{G}$ and $s_0,t_0 \in \mathbb{R}$ non negative numbers such that 
$(\para{P}_i,s_0,\para{P}_j,t_0,
\weylelement{})$ is a positive admissible data.
\end{cor*}
The positive admissible data constructed by the above method is called \textbf{special admissible} data if $s_0,t_0$ are non negative. However, Theorem
(\ref{we}) gives only a partial list of admissible data.
Surprisingly, the list of admissible data that  can be derived from Theorem (\ref{we}) is almost the complete list of admissible data that exists.

In section \ref{section:anoterway} we describe an algorithm that gives a complete list of admissible data.  
\\
In sections \ref{section:G2},\ref{section:F4},\ref{section:E6},\ref{section:E7}  we write explicitly the admissible data and apply  Theorem (\ref{main_thm}) for the groups of types $G_2,F_4,E_6,E_7$. 

\section{Special admissible data} \label{section_suit_data}
Our goal is to prove Theorem (\ref{we}). We need few notations.

Let $\para{P}_i,\para{P}_j$ be  maximal parabolic subgroups of $G$. Denote by $\para{R}=\para{P}_i \cap \para{P}_j$. Let 
$\weyl{\para{P}_i},\weyl{\para{P}_j},\weyl{\para{R}}$ be the Weyl group of $\para{P}_i,\para{P}_j,\para{R}$. Let $\tilde{\weylelement{}}_{0,i}$ (resp. $\tilde{\weylelement{}}_{0,j}$) be the shortest representative of the longest Weyl element of $\weyl{\para{P}_i} \slash \weyl{\para{R}}$ (resp. $\weyl{\para{P}_j} \slash \weyl{\para{R}}$).

Note that since $\para{R}$ is the intersection of two maximal parabolic subgroups the set of the unramified characters of $\para{R}$ is isomorphic to $\mathbb{C}^2$. 
Let us  define a matrix $B =(b_{ij})$ whose  entries satisfy
\begin{align*}
\delta_{\para{P}_i}&=b_{ii}\fundamental{i} &
\delta_{\para{P}_j}&=b_{jj}\fundamental{j} & 
\delta_{\para{R}}&=b_{ij}\fundamental{i}+ b_{ji}\fundamental{j}  \quad \mbox{ if } i<j.
\end{align*}
\begin{lemm}
Let 
 $c_{ij}=\frac{b_{ij}}{b_{ii}} -\frac{1}{2}$ and  $c_{ji}=\frac{b_{ji}}{b_{jj}} -\frac{1}{2}$.  
Then:
$$\delta_{\para{P}_{i}}^{c_{ij}} \otimes ( \delta_{\para{R}}^{\para{P}_{i}})^{-\frac{1}{2}} =
\delta_{\para{P}_{j}}^{-c_{ji}} \otimes ( \delta_{\para{R}}^{\para{P}_{j}})^{\frac{1}{2}}.$$
\end{lemm}
\begin{proof}
\begin{align*}
\delta_{\para{P}_i}^{c_{ij}}\otimes ( \delta_{\para{R}}^{\para{P}_{i}})^{-\frac{1}{2}}=&
c_{ij}b_{ii}\fundamental{i} - \frac{1}{2}(b_{ij}\fundamental{i} +b_{ji}\fundamental{j} -b_{ii}\fundamental{i}) \\
=&
b_{ij}\fundamental{i}-\frac{b_{ii}}{2}\fundamental{i} -\frac{1}{2}b_{ij}\fundamental{i}-\frac{1}{2}b_{ji}\fundamental{j}+\frac{1}{2}b_{ii}\fundamental{i}\\
=& \frac{1}{2}b_{ij}\fundamental{i}-\frac{1}{2}b_{ji}\fundamental{j}.
\end{align*}
On the other hand 
\begin{align*}
\delta_{\para{P}_j}^{-c_{ji}}\otimes ( \delta_{\para{R}}^{\para{P}_{j}})^{\frac{1}{2}}=&
-c_{ji}b_{jj}\fundamental{j} + \frac{1}{2}(b_{ij}\fundamental{i} +b_{ji}\fundamental{j} -b_{jj}\fundamental{j}) \\
=&
-b_{ji}\fundamental{j}+\frac{b_{jj}}{2}\fundamental{j} +\frac{1}{2}b_{ij}\fundamental{i}+\frac{1}{2}b_{ji}\fundamental{j}-\frac{1}{2}b_{jj}\fundamental{j}\\
=& \frac{1}{2}b_{ij}\fundamental{i}-\frac{1}{2}b_{ji}\fundamental{j}.
\end{align*} 
So we are done.
\end{proof}
 \begin{lemm} \label{Longest_weyl}
 Let $\chi_s= \delta_{\para{P}}^{s -\frac{1}{2}} \otimes \delta_{\para{B}}^{\frac{1}{2}}$ and let $\weylelement{_\para{P}} \in \weyl{\para{P}}$ be the longest element. Then $\weylelement{_\para{P}}(\chi_s)= \delta_{\para{P}}^{s+ \frac{1}{2}} \otimes \delta_{\para{B}}^{-\frac{1}{2}}$. 
 \end{lemm}
 \begin{proof}
For every $\alpha \in \Phi_{G}^{+} \setminus \Phi_{M}^{+}$ it holds that $\weylelement{_\para{P}}\alpha \in 
 \Phi_{G}^{+} \setminus \Phi_{M}^{+}$
 and for every $\beta \in \Phi_{M}^{+}$ it holds that $\weylelement{_\para{P}}\beta \in  \Phi_{M}^{-} $.
Now rewrite $\delta_{\para{B}}=\delta_{\para{P}} \otimes  (\delta_{\para{B}} \otimes \delta_{\para{P}}^{-1})$.
Observe that 
$$(\delta_{\para{B}} \otimes \delta_{\para{P}}^{-1}) = \prod_{\alpha \in \Phi_{M}^{+}} \alpha.$$ 
Hence, $\weylelement{_\para{P}}(\delta_{\para{B}} \otimes \delta_{\para{P}}^{-1}) =
(\delta_{\para{B}}^{-1} \otimes \delta_{\para{P}})$. Hence, 
\begin{align*}
\weylelement{_\para{P}}(\delta_{\para{P}}^{s -\frac{1}{2}} \otimes \delta_{\para{B}}^{\frac{1}{2}})=& 
\weylelement{_\para{P}} (\delta_{\para{P}}^{s}) \otimes 
\weylelement{_\para{P}} (\delta_{\para{P}}^{-\frac{1}{2}} \otimes 
\delta_{\para{B}}^{\frac{1}{2}})\\
=& \delta_{\para{P}}^{s} \otimes 
\delta_{\para{P}}^{\frac{1}{2}} \otimes \delta_{\para{B}}^{-\frac{1}{2}} \\
=& \delta_{\para{P}}^{s + \frac{1}{2}} \otimes \delta_{\para{B}}^{-\frac{1}{2}}.
\end{align*}
 \end{proof}
\begin{thm*}\label{Weyl}
With the notation as above it holds that there exists $\weylelement{} \in \weyl{G}$ such that 
$$\weylelement{}(\delta_{\para{P}_i}^{c_{ij}-\frac{1}{2} } \otimes\delta_{\para{B}}^{\frac{1}{2}}) =
\delta_{\para{P}_j}^{-c_{ji}-\frac{1}{2} } \otimes\delta_{\para{B}}^{\frac{1}{2}}$$
i.e $(\para{P}_i,c_{ij},\para{P}_j,-c_{ji},\weylelement{})$ is an admissible data.
\end{thm*}
\begin{proof}
By Lemma (\ref{Longest_weyl}) we know that 
$$\weylelement{_{\para{P}_i}} (\delta_{\para{P}_i}^{c_{ij}-\frac{1}{2} } \otimes\delta_{\para{B}}^{\frac{1}{2}})=
\delta_{\para{P}_{i}}^{c_{ij} + \frac{1}{2}} \otimes \delta_{\para{B}}^{-\frac{1}{2}}.$$ 
Note that 
\begin{align*}
\delta_{\para{P}_{i}}^{c_{ij} + \frac{1}{2}} \otimes \delta_{\para{B}}^{-\frac{1}{2}}
&=\delta_{\para{P}_{i}}^{c_{ij} + \frac{1}{2}} \otimes \delta_{R}^{-\frac{1}{2}} \otimes 
\delta_{R}^{\frac{1}{2}} \otimes 
\delta_{\para{B}}^{-\frac{1}{2}}\\
&\overset{}{=}
\delta_{\para{P}_{j}}^{-c_{ji} - \frac{1}{2}} \otimes \delta_{R}^{\frac{1}{2}} \otimes 
\delta_{R}^{\frac{1}{2}} \otimes 
\delta_{\para{B}}^{-\frac{1}{2}}\\
&=\delta_{\para{P}_j}^{-c_{ji}-\frac{1}{2}} \otimes
\delta_{\para{R}}^{\frac{1}{2}}\otimes 
(\delta_{\para{B}}^{\para{R}})^{-\frac{1}{2}}.
\end{align*}
Let $\weylelement{\para{R}}$ be the longest element in $\weyl{\para{R}}$. Then 
\begin{align*}
\weylelement{\para{R}}\delta_{\para{P}_j}&=
\delta_{\para{P}_j},
&
\weylelement{\para{R}}\delta_{\para{R}}&=
\delta_{\para{R}},
&
\weylelement{\para{R}}\delta_{\para{B}}^{\para{R}}&=
(\delta_{\para{B}}^{\para{R}})^{-1}.
\end{align*}
Hence
\begin{align*}
\weylelement{\para{R}}(\delta_{\para{P}_j}^{-c_{ji}-\frac{1}{2}} \otimes
\delta_{\para{R}}^{\frac{1}{2}}\otimes 
(\delta_{\para{B}}^{\para{R}})^{-\frac{1}{2}})&=
\delta_{\para{P}_j}^{-c_{ji}-\frac{1}{2}} \otimes
\delta_{\para{R}}^{\frac{1}{2}}\otimes 
(\delta_{\para{B}}^{\para{R}})^{\frac{1}{2}}\\
&=\delta_{\para{P}_j}^{-c_{ji}-\frac{1}{2} } \otimes\delta_{\para{B}}^{\frac{1}{2}}.
\end{align*}
Therefore  for
$\weylelement{}=\weylelement{\para{R}}\weylelement{\para{P}_i}$  it holds that:
$$\weylelement{}(\delta_{\para{P}_i}^{c_{ij}-\frac{1}{2} } \otimes\delta_{\para{B}}^{\frac{1}{2}}) =
\delta_{\para{P}_j}^{-c_{ji}-\frac{1}{2} } \otimes\delta_{\para{B}}^{\frac{1}{2}}$$
Notice that 
\begin{align*}
\weylelement{}&=\weylelement{\para{R}}\weylelement{\para{P}_i} \\
&=\weylelement{\para{P}_j}\weylelement{\para{P}_j}\weylelement{\para{R}}\weylelement{\para{P}_i} \\
&= \weylelement{\para{P}_j}
\tilde{\weylelement{}}_{0,j}
\weylelement{\para{P}_i}
\end{align*}
where $\tilde{\weylelement{}}_{0,j}$ is the shortest representative of the longest element in $\weyl{\para{P}_j}\slash \weyl{\para{R}}$.
\end{proof}
\begin{cor}
Let $G$ be an algebraic group of type different from 
$A_n, D_{2n+1},E_6$. Then for $s_{0} = |c_{ij}|,t_0=|c_{ji}|$  there exists 
$\weylelement{} \in \weyl{G} $ such that 
$$\weylelement{}(\delta_{\para{P}_i}^{s_0-\frac{1}{2} } \otimes\delta_{\para{B}}^{\frac{1}{2}}) =
\delta_{\para{P}_j}^{t_0-\frac{1}{2} } \otimes\delta_{\para{B}}^{\frac{1}{2}}.$$
\end{cor}
\begin{proof}
Let 
$w_{l} \in \weyl{G} \slash \weyl{\para{P}_j}$ and 
$w_{r} \in \weyl{G} \slash \weyl{\para{P}_i}$ be the shortest representative of the longest element.
Set
$$r=\begin{cases}
e & \mbox{if } s_0=c_{ij} \\
w_{r} & \mbox{if } s_0= -c_{ij}
\end{cases}
\quad \quad 
l=\begin{cases}
e & \mbox{if } t_0=-c_{ji} \\
w_{l} & \mbox{if } t_0= c_{ji}
\end{cases}
$$
Let $\weylelement{}$ be as in Proposition (\ref{Weyl}). Then 
$$l\weylelement{}r(\delta_{\para{P}_i}^{s_0-\frac{1}{2} } \otimes\delta_{\para{B}}^{\frac{1}{2}}) =
\delta_{\para{P}_j}^{t_0-\frac{1}{2} } \otimes\delta_{\para{B}}^{\frac{1}{2}}$$
\end{proof}
\begin{defn}
The data $(\para{P}_i,s_0,\para{P}_j,t_0,\weylelement{})$ that is constructed by the above method  is called  \textbf{special admissible data} if $s_0,t_0$ are non negative.
\end{defn}

\newpage
\section{Algorithm: Finding Positive Admissible Data}\label{section:anoterway}
In this section, we  introduce a way to compute all
the (positive) admissible data. 
Let us fix some notations.
\\
Let $\para{P},\para{Q}$ be  maximal parabolic subgroups. 
Let 
\begin{align*}
\chi_{\para{P}}(s) =&\delta_{\para{P}}^{s-\frac{1}{2}}\otimes \delta_{\para{B}}^{\frac{1}{2}} &
\chi_{\para{Q}}(t) =&\delta_{\para{Q}}^{t-\frac{1}{2}}\otimes \delta_{\para{B}}^{\frac{1}{2}}
& (\mathfrak{a}_T^{\ast})^{+}=& \{ 
\lambda \in \mathfrak{a}_T^{\ast} \: : \:
\inner{\lambda,\check{\alpha}} \geq 0 \quad \forall \alpha \in \Delta_{G}\}
\end{align*}
It is well known that 
$|(\mathfrak{a}_T^{\ast})^{+} \cap 
\{ \weylelement{}\chi_{\para{P}}(s_0) \: : \: \weylelement{}\in \weyl{G} \}|= 1$. Let
\begin{align*}
dom_{\para{P}},dom_{\para{Q}} \: &: \mathbb{R} \rightarrow 
(\mathfrak{a}_T^{\ast})^{+} 
\end{align*}
such that 
$$dom_{\para{P}}(s_0) = (\mathfrak{a}_T^{\ast})^{+} \cap 
\{ \weylelement{}\chi_{\para{P}}(s_0) \: : \: \weylelement{}\in \weyl{G} \} \quad 
dom_{\para{Q}}(s_0) = (\mathfrak{a}_T^{\ast})^{+} \cap 
\{ \weylelement{}\chi_{\para{Q}}(s_0) \: : \: \weylelement{}\in \weyl{G} \}
$$
\begin{remk}
The functions $dom_{\para{P}},dom_{\para{Q}}$ are piecewise linear.
\end{remk} 
For every pair of intervals on which $dom_{\para{P}},
dom_{\para{Q}}$ are linear, we solve the system of linear equation $dom_{\para{P}}(s) =dom_{\para{Q}}(t)$. The solution will give rise to admissible data by the following procedure. 
Let 
\begin{align*}
W_{\para{P}} \: &: \mathbb{R} \rightarrow 
\weyl{G}
   & W_{\para{Q}} \: &: \mathbb{R} \rightarrow 
\weyl{G}
\end{align*}
such that $W_{\para{P}}(s_0)$ and $W_{\para{Q}}(s_0)$ are the shortest Weyl elements such that
\begin{align*}
W_{\para{P}}(s_0)\chi_{\para{P}}(s_0) =& dom_{\para{P}}(s_0)
& W_{\para{Q}}(s_0)\chi_{\para{Q}}(s_0) =& dom_{\para{Q}}(s_0)
\end{align*}
\begin{remk}
The functions $W_{\para{P}},W_{\para{Q}}$ are piecewise constants.
\end{remk}
Suppose $dom_{\para{P}}(s_0)=dom_{\para{Q}}(t_0)$ for some numbers $s_0,t_0$. Then 
$$
w(\chi_{\para{P}}(s_0))=\chi_{\para{Q}}(t_0)
\quad  \text{ for }
w=W_{\para{Q}}(t_0)^{-1}W_{\para{P}}(s_0)
$$
Hence 
$(\para{P},s_0,\para{Q},t_0,w)$ is an  admissible   data.

\begin{remk}
Since we are mostly interested in the positive admissible data, we can restrict ourself to the domain $[0,\infty)$.
Moreover, for $s>\frac{1}{2},t>\frac{1}{2}$ we do not have any positive admissible data.
\end{remk}
In the following sections we list all the positive admissible data.
\begin{remk}
As we have mentioned earlier, for every pair of maximal parabolic subgroups $\para{P}_i,\para{P}_j$ we have the admissible  data $(\para{P}_{i},\frac{1}{2},\para{P}_{j},\frac{1}{2} ,e)$ that corresponds to the identity between the trivial representation.
Hence we may omit this.
\end{remk}
For each section we start by finding all the special data and list the identities. After this we give the complete list of positive admissible data and write the remaining identities. 

\lhead{Chapter \ref{ch3:ids}. \emph{Identities- Group of type $G_{2}$}}
\section{Group of type $G_{2}$}\label{section:G2}
\begin{prop}\label{G2_mod_char} 
Let $G=G_2$ then it holds that
\begin{align*}
 \delta_{\para{P}_1}&=5\fundamental{1} &
  \delta_{\para{P}_2}&=3\fundamental{2}
  & \delta_{\para{B}}=2\fundamental{1}+2\fundamental{2}.
  \end{align*}
 \end{prop}
 \begin{prop} 
In this case we have  only one (special) admissible data
$(\para{P}_2, \frac{1}{6}, \para{P}_1 ,\frac{1}{10})$
 \begin{remk}
In this case we have \textbf{only} one identity.
 \end{remk}
 \end{prop} 
 \begin{prop} It holds that : 
 \begin{align*} 
 A_{w_{\para{P}_{1} }} = &\frac{R}{\zeta ( 2 )}& w_{\para{P}_{1} } =&  w_{2} \\ 
 A_{w_{\para{P}_{2} }} = &\frac{R}{\zeta ( 2 )}& w_{\para{P}_{2} } =&  w_{1} \\ 
  \end{align*} 
 \end{prop} 
 \begin{tabular}{|c|c|c|c|c|c|c|c|c|c|c|} \hline
$ \para{P}_i $ & $ s $ & $ \para{P}_j $ & $ t $ & $ w $ & $\frac{h_3(\chi_t)}{h_3(\chi_s)}$ & $d_{\para{P}_i}(\chi_s)$ & $d_{\para{P}_j}(\chi_t)$ & $ d $ & $ \epsilon_s $ & $ \epsilon_t $ \\ \hline
$\para{P}_{2}$ & $ \frac{1}{6} $ & $\para{P}_{1}$ & $ \frac{1}{10} $ & $ w_{1} $ & $ 1 $ & $ 2 $ & $ 1 $ & $ 1 $ & $ 54 $ & $ 10 $ \\ \hline
\end{tabular}
 \begin{thm} 
Let $f^{0} \in I_{\para{P}_i}(s)$ be the normalized spherical section then:
\begin{enumerate} 
\item 
 $\leadingterm{-2}^{\para{P}_{2} } (f^{0},{\frac{1}{6}},g)=\frac{5}{27} \times \frac{R}{\zeta ( 2 )} \times  \leadingterm{-1}^{\para{P}_{1} } (f^{0},{\frac{1}{10}},g).$ 
\end{enumerate} 
\end{thm} 
\newpage
\lhead{Chapter \ref{ch3:ids}. \emph{Identities- Group of type $F_{4}$}}
\section{Group of type $F_{4}$}\label{section:F4}
 \begin{prop}\label{F4_mod_char} 
The matrix $B$ for $G=F_4$ is
 $$B= \left(\begin{array}{rrrr}
8 & 2 & 3 & 5 \\
4 & 5 & 3 & 4 \\
5 & 3 & 7 & 6 \\
6 & 3 & 2 & 11
\end{array}\right) $$ 
 \end{prop} 
\begin{prop} 
 The special admissible  data are as follows: 
 \begin{figure}[H] \begin{center}
 
\begin{tabular}{|c|c|c|c|c|} \hline
$[s_0,t_0]$ & $ \para{P}_{1}$ & $ \para{P}_{2}$ & $ \para{P}_{3}$ & $ \para{P}_{4}$ \\ \hline
$\para{P}_{1}$ &  & $\left[\frac{1}{4}, \frac{3}{10}\right]$ & $\left[\frac{1}{8}, \frac{3}{14}\right]$ & $\left[\frac{1}{8}, \frac{1}{22}\right]$ \\ \hline
$\para{P}_{2}$ & $\left[\frac{3}{10}, \frac{1}{4}\right]$ &  & $\left[\frac{1}{10}, \frac{1}{14}\right]$ & $\left[\frac{3}{10}, \frac{5}{22}\right]$ \\ \hline
$\para{P}_{3}$ & $\left[\frac{3}{14}, \frac{1}{8}\right]$ & $\left[\frac{1}{14}, \frac{1}{10}\right]$ &  & $\left[\frac{5}{14}, \frac{7}{22}\right]$ \\ \hline
$\para{P}_{4}$ & $\left[\frac{1}{22}, \frac{1}{8}\right]$ & $\left[\frac{5}{22}, \frac{3}{10}\right]$ & $\left[\frac{7}{22}, \frac{5}{14}\right]$ &  \\ \hline
\end{tabular}

  \end{center} 
 \end{figure} 
 \end{prop}

 \begin{prop} It holds that : 
 \begin{align*} 
 A_{w_{\para{P}_{1} }} = &\frac{R^{3}}{\zeta ( 4 ) \zeta ( 2 ) \zeta ( 6 )}& w_{\para{P}_{1} } =&  w_{4}w_{3}w_{2}w_{3}w_{4}w_{3}w_{2}w_{3}w_{2} \\ 
 A_{w_{\para{P}_{2} }} = &\frac{R^{3}}{\zeta ( 2 )^{2} \zeta ( 3 )}& w_{\para{P}_{2} } =&  w_{3}w_{4}w_{3}w_{1} \\ 
 A_{w_{\para{P}_{3} }} = &\frac{R^{3}}{\zeta ( 2 )^{2} \zeta ( 3 )}& w_{\para{P}_{3} } =&  w_{4}w_{1}w_{2}w_{1} \\ 
 A_{w_{\para{P}_{4} }} = &\frac{R^{3}}{\zeta ( 6 ) \zeta ( 4 ) \zeta ( 2 )}& w_{\para{P}_{4} } =&  w_{3}w_{2}w_{3}w_{1}w_{2}w_{3}w_{1}w_{2}w_{1} \\ 
  \end{align*} 
 \end{prop} 
 \begin{longtable}{|c|c|c|c|c|c|c|c|c|c|c|} 
\hline$ \para{P}_{i} $ & $ s $ & $ \para{P}_{j} $ &  $ t $ & $w$ & $ \frac{h_3(\chi_t)}{h_3(\chi_s)} $ & $d_{\para{P}_i}(\chi_s)$ & $d_{\para{P}_j}(\chi_t)$ & $d$ & $\epsilon_p$ & $\epsilon_q$ \\ \hline 
$ \para{P}_{1} $ & $ \frac{1}{8} $ & $ \para{P}_{4} $ &  $ \frac{1}{22} $ & $w_{4}w_{3}w_{1}w_{2}$ & $ \frac{ \zeta ( 3 ) } { \zeta ( 5 ) } $ & $1$ & $1$ & $0$ & $8$ & $22$ \\ \hline 
$ \para{P}_{2} $ & $ \frac{3}{10} $ & $ \para{P}_{1} $ &  $ \frac{1}{4} $ & $w_{1}$ & $ 1 $ & $2$ & $1$ & $1$ & $50$ & $16$ \\ \hline 
$ \para{P}_{2} $ & $ \frac{1}{10} $ & $ \para{P}_{3} $ &  $ \frac{1}{14} $ & $w_{3}w_{4}$ & $ \frac{ \zeta ( 2 ) } { \zeta ( 3 ) } $ & $3$ & $2$ & $1$ & $3000$ & $294$ \\ \hline 
$ \para{P}_{2} $ & $ \frac{3}{10} $ & $ \para{P}_{4} $ &  $ \frac{5}{22} $ & $w_{4}w_{3}$ & $ \frac{ \zeta ( 2 ) } { \zeta ( 3 ) } $ & $2$ & $1$ & $1$ & $50$ & $11$ \\ \hline 
$ \para{P}_{3} $ & $ \frac{3}{14} $ & $ \para{P}_{1} $ &  $ \frac{1}{8} $ & $w_{1}w_{2}$ & $ \frac{ \zeta ( 2 ) } { \zeta ( 3 ) } $ & $2$ & $1$ & $1$ & $98$ & $8$ \\ \hline 
$ \para{P}_{3} $ & $ \frac{5}{14} $ & $ \para{P}_{4} $ &  $ \frac{7}{22} $ & $w_{4}$ & $ 1 $ & $1$ & $0$ & $1$ & $7$ & $1$ \\ \hline 

 \end{longtable}

\newpage
 \begin{thm}  
 Let $f^{0} \in I_{\para{P}_i}(s)$ be the normalized spherical section then:\begin{enumerate} 
\item 
$\leadingterm{ -1 }^{\para{P}_{1}} (f^0,\frac{1}{8},g)=\frac{11}{4} \times \frac{ \zeta ( 3 ) } { \zeta ( 5 ) } \times \leadingterm{ -1 }^{\para{P}_{4}} (f^0,\frac{1}{22},g). 
 $\item 
$\leadingterm{ -2 }^{\para{P}_{2}} (f^0,\frac{3}{10},g)=\frac{8}{25} \times \frac{ R \zeta ( 3 ) } { \zeta ( 4 ) \zeta ( 6 ) } \times \leadingterm{ -1 }^{\para{P}_{1}} (f^0,\frac{1}{4},g). 
 $\item 
$\leadingterm{ -3 }^{\para{P}_{2}} (f^0,\frac{1}{10},g)=\frac{49}{500} \times \frac{R } { \zeta ( 3 ) } \times \leadingterm{ -2 }^{\para{P}_{3}} (f^0,\frac{1}{14},g). 
 $\item 
$\leadingterm{ -2 }^{\para{P}_{2}} (f^0,\frac{3}{10},g)=\frac{11}{50} \times \frac{R \zeta ( 2 ) } { \zeta ( 6 ) \zeta ( 4 ) } \times \leadingterm{ -1 }^{\para{P}_{4}} (f^0,\frac{5}{22},g) .
 $\item 
$\leadingterm{ -2 }^{\para{P}_{3}} (f^0,\frac{3}{14},g)=\frac{4}{49} \times \frac{R \zeta ( 2 ) } { \zeta ( 4 ) \zeta ( 6 ) } \times \leadingterm{ -1 }^{\para{P}_{1}} (f^0,\frac{1}{8},g) .
 $\item 
$\leadingterm{ -1 }^{\para{P}_{3}} (f^0,\frac{5}{14},g)=\frac{1}{7} \times \frac{R \zeta ( 3 ) } { \zeta ( 6 ) \zeta ( 4 ) } \times \leadingterm{ 0 }^{\para{P}_{4}} (f^0,\frac{7}{22},g) .
 $\end{enumerate} 
 \end{thm}
There exists a positive admissible data which is  non special admissible data and cannot be derived from the list above   
$$(\para{P}_2,\frac{1}{5},\para{P}_1,\frac{1}{16},w_{1}w_{2}w_{3}w_{4}). $$
This positive admissible data gives rise to the following identity:
 \begin{thm} 
Let $f^{0} \in I_{\para{P}_i}(s)$ be the normalized spherical section then:
\begin{enumerate} 
\item 
$\leadingterm{ -1 }^{\para{P}_{2}} (f^0,\frac{1}{5},g)=\frac{1}{10} \times \frac{R \zeta ( 2 ) \zeta ( \frac{3}{2} ) } { \zeta ( 4 ) \zeta ( 6 ) \zeta ( \frac{7}{2} ) } \times \leadingterm{ 0 }^{\para{P}_{1}} (f^0,\frac{1}{16},g). 
 $
\end{enumerate}
\end{thm}

Therefore we get the follows: 
\begin{figure}[H]
\begin{center}
\begin{tikzpicture}[thin,scale=2]
\tikzset{
  text style/.style={
    sloped, 
    text=black,
    font=\tiny,
    above
  }
}
\matrix (m) [matrix of math nodes, row sep=1em, column sep=3em]
    { 	
    & |[name=p114]|	 \text{I}_{\para{P}_1}(\frac{1}{4})
    & & |[name=p118]|	 \text{I}_{\para{P}_1}(\frac{1}{8}) 
    \\
    |[name=p2310]|	 \text{I}_{\para{P}_2}(\frac{3}{10})
    & & |[name=p3314]|	 \text{I}_{\para{P}_3}(\frac{3}{14})
    \\
    & |[name=p4522]|	 \text{I}_{\para{P}_4}(\frac{5}{22})
    & & |[name=p4122]|	 \text{I}_{\para{P}_4}(\frac{1}{22}) 
    \\
    |[name=p2110]|	 \text{I}_{\para{P}_2}(\frac{1}{10})
        & 
    |[name=p3114]|	 \text{I}_{\para{P}_3}(\frac{1}{14})        
    & |[name=p3514]|	 \text{I}_{\para{P}_3}(\frac{5}{14})
    & |[name=p4722]|	 \text{I}_{\para{P}_4}(\frac{7}{22})
\\
   |[name=p215]|	 \text{I}_{\para{P}_2}(\frac{1}{5})
  &
  |[name=p1116]|	 \text{I}_{\para{P}_1}(\frac{1}{16})
  \\
};
\draw[->]      
		(p2310)  edge [-]
		node[text style,above]{}  (p114)
        (p2310) edge [-] 
        		node[text style,above]
        		 {} (p4522) 
        (p3314)  edge [-]
        node[text style,above]{}(p118)
        
        (p2110)  edge [-]
        		node[text style,above]{}  (p3114)
	    (p3514) edge [-] 
                		node[text style,above]
                		 {
                		 } (p4722) 
	    (p215) edge [-] 
                		node[text style,above]
                		 {
                		 } (p1116) 
(p3314)
 edge [-] 
                		node[text style,above]
                		 {
                		 } (p4122) 

            ;
         
\end{tikzpicture}
\label{digram:F4s}
\end{center}
\end{figure}
\newpage
\lhead{Chapter \ref{ch3:ids}. \emph{Identities- Group of type $E_{6}$}}
\section{Group of type $E_{6}$} \label{section:E6}
 \begin{prop}\label{E6_mod_char} 
The matrix $B$ for $G=E_6$ is
 $$B= \left(\begin{array}{rrrrrr}
12 & 6 & 2 & 3 & 5 & 8 \\
8 & 11 & 5 & 2 & 5 & 8 \\
8 & 6 & 9 & 3 & 5 & 7 \\
6 & 6 & 5 & 7 & 5 & 6 \\
7 & 6 & 5 & 3 & 9 & 8 \\
8 & 6 & 5 & 3 & 2 & 12
\end{array}\right) $$ 
 \end{prop} \ 
 \begin{prop} 
 The special admissible  data are as follows: 
 \begin{figure}[H] \begin{center} 
\begin{tabular}{|c|c|c|c|c|c|c|} \hline
$[s_0,t_0]$ & $ \para{P}_{1}$ & $ \para{P}_{2}$ & $ \para{P}_{3}$ & $ \para{P}_{4}$ & $ \para{P}_{5}$ & $ \para{P}_{6}$ \\ \hline
$\para{P}_{1}$ &  &  &  &  &  &  \\ \hline
$\para{P}_{2}$ & $\left[\frac{5}{22}, 0\right]$ &  &  &  &  & $\left[\frac{5}{22}, 0\right]$ \\ \hline
$\para{P}_{3}$ & $\left[\frac{7}{18}, \frac{1}{3}\right]$ & $\left[\frac{1}{6}, \frac{1}{22}\right]$ &  &  &  & $\left[\frac{5}{18}, \frac{1}{12}\right]$ \\ \hline
$\para{P}_{4}$ & $\left[\frac{5}{14}, \frac{1}{4}\right]$ & $\left[\frac{5}{14}, \frac{7}{22}\right]$ & $\left[\frac{3}{14}, \frac{1}{6}\right]$ &  & $\left[\frac{3}{14}, \frac{1}{6}\right]$ & $\left[\frac{5}{14}, \frac{1}{4}\right]$ \\ \hline
$\para{P}_{5}$ & $\left[\frac{5}{18}, \frac{1}{12}\right]$ & $\left[\frac{1}{6}, \frac{1}{22}\right]$ &  &  &  & $\left[\frac{7}{18}, \frac{1}{3}\right]$ \\ \hline
$\para{P}_{6}$ &  &  &  &  &  &  \\ \hline
\end{tabular}
  \end{center} 
 \end{figure} 
 \end{prop} 
 \begin{prop} It holds that : 
 \begin{align*} 
 A_{w_{\para{P}_{1} }} = &\frac{R^{5}}{\zeta ( 2 ) \zeta ( 4 ) \zeta ( 5 ) \zeta ( 6 ) \zeta ( 8 )}& w_{\para{P}_{1} } =&  w_{6}w_{5}w_{4}w_{3}w_{2}w_{4}w_{5}w_{6}w_{5}w_{4}w_{3}w_{2}w_{4}w_{5}w_{4}w_{3}w_{2}w_{4}w_{3}w_{2} \\ 
 A_{w_{\para{P}_{2} }} = &\frac{R^{5}}{\zeta ( 2 ) \zeta ( 3 ) \zeta ( 4 ) \zeta ( 5 ) \zeta ( 6 )}& w_{\para{P}_{2} } =&  w_{1}w_{3}w_{4}w_{5}w_{6}w_{1}w_{3}w_{4}w_{5}w_{1}w_{3}w_{4}w_{1}w_{3}w_{1} \\ 
 A_{w_{\para{P}_{3} }} = &\frac{R^{5}}{\zeta ( 2 )^{2} \zeta ( 3 ) \zeta ( 4 ) \zeta ( 5 )}& w_{\para{P}_{3} } =&  w_{2}w_{4}w_{5}w_{6}w_{2}w_{4}w_{5}w_{2}w_{4}w_{2}w_{1} \\ 
 A_{w_{\para{P}_{4} }} = &\frac{R^{5}}{\zeta ( 2 )^{3} \zeta ( 3 )^{2}}& w_{\para{P}_{4} } =&  w_{5}w_{6}w_{5}w_{1}w_{3}w_{2}w_{1} \\ 
 A_{w_{\para{P}_{5} }} = &\frac{R^{5}}{\zeta ( 2 )^{2} \zeta ( 3 ) \zeta ( 4 ) \zeta ( 5 )}& w_{\para{P}_{5} } =&  w_{6}w_{3}w_{4}w_{1}w_{3}w_{2}w_{4}w_{1}w_{3}w_{2}w_{1} \\ 
 A_{w_{\para{P}_{6} }} = &\frac{R^{5}}{\zeta ( 2 ) \zeta ( 4 ) \zeta ( 5 ) \zeta ( 6 ) \zeta ( 8 )}& w_{\para{P}_{6} } =&  w_{2}w_{4}w_{5}w_{3}w_{4}w_{1}w_{3}w_{2}w_{4}w_{5}w_{3}w_{4}w_{1}w_{3}w_{2}w_{4}w_{1}w_{3}w_{2}w_{1} \\ 
  \end{align*} 
 \end{prop} 
 \begin{longtable}{|c|c|c|c|c|c|c|c|c|c|c|} 
 \hline$ \para{P}_{i} $ & $ s $ & $ \para{P}_{j} $ &  $ t $ & $w$ & $ \frac{h_3(\chi_t)}{h_3(\chi_s)} $ & $d_{\para{P}_i}(\chi_s)$ & $d_{\para{P}_j}(\chi_t)$ & $d$ & $\epsilon_p$ & $\epsilon_q$ \\ \hline 
 $ \para{P}_{2} $ & $ \frac{5}{22} $ & $ \para{P}_{1} $ &  $ 0 $ & $w_{1}w_{3}w_{4}w_{5}w_{6}$ & $ \frac{ \zeta ( 2 ) } { \zeta ( 6 ) } $ & $1$ & $0$ & $1$ & $11$ & $1$ \\ \hline 
 $ \para{P}_{2} $ & $ \frac{5}{22} $ & $ \para{P}_{6} $ &  $ 0 $ & $w_{6}w_{5}w_{4}w_{3}w_{1}$ & $ \frac{ \zeta ( 2 ) } { \zeta ( 6 ) } $ & $1$ & $0$ & $1$ & $11$ & $1$ \\ \hline 
 $ \para{P}_{3} $ & $ \frac{7}{18} $ & $ \para{P}_{1} $ &  $ \frac{1}{3} $ & $w_{1}$ & $ 1 $ & $1$ & $0$ & $1$ & $9$ & $1$ \\ \hline 
 $ \para{P}_{3} $ & $ \frac{1}{6} $ & $ \para{P}_{2} $ &  $ \frac{1}{22} $ & $w_{2}w_{4}w_{5}w_{6}$ & $ \frac{ \zeta ( 2 ) } { \zeta ( 5 ) } $ & $2$ & $1$ & $1$ & $162$ & $22$ \\ \hline 
 $ \para{P}_{3} $ & $ \frac{5}{18} $ & $ \para{P}_{6} $ &  $ \frac{1}{12} $ & $w_{6}w_{5}w_{4}w_{2}$ & $ \frac{ \zeta ( 2 ) } { \zeta ( 5 ) } $ & $1$ & $0$ & $1$ & $9$ & $1$ \\ \hline 
 $ \para{P}_{4} $ & $ \frac{5}{14} $ & $ \para{P}_{1} $ &  $ \frac{1}{4} $ & $w_{1}w_{3}$ & $ \frac{ \zeta ( 2 ) } { \zeta ( 3 ) } $ & $2$ & $1$ & $1$ & $49$ & $12$ \\ \hline 
 $ \para{P}_{4} $ & $ \frac{5}{14} $ & $ \para{P}_{2} $ &  $ \frac{7}{22} $ & $w_{2}$ & $ 1 $ & $2$ & $1$ & $1$ & $49$ & $11$ \\ \hline 
 $ \para{P}_{4} $ & $ \frac{3}{14} $ & $ \para{P}_{3} $ &  $ \frac{1}{6} $ & $w_{3}w_{1}$ & $ \frac{ \zeta ( 2 ) } { \zeta ( 3 ) } $ & $3$ & $2$ & $1$ & $686$ & $162$ \\ \hline 
 $ \para{P}_{4} $ & $ \frac{3}{14} $ & $ \para{P}_{5} $ &  $ \frac{1}{6} $ & $w_{5}w_{6}$ & $ \frac{ \zeta ( 2 ) } { \zeta ( 3 ) } $ & $3$ & $2$ & $1$ & $686$ & $162$ \\ \hline 
 $ \para{P}_{4} $ & $ \frac{5}{14} $ & $ \para{P}_{6} $ &  $ \frac{1}{4} $ & $w_{6}w_{5}$ & $ \frac{ \zeta ( 2 ) } { \zeta ( 3 ) } $ & $2$ & $1$ & $1$ & $49$ & $12$ \\ \hline 
 $ \para{P}_{5} $ & $ \frac{5}{18} $ & $ \para{P}_{1} $ &  $ \frac{1}{12} $ & $w_{1}w_{3}w_{4}w_{2}$ & $ \frac{ \zeta ( 2 ) } { \zeta ( 5 ) } $ & $1$ & $0$ & $1$ & $9$ & $1$ \\ \hline 
 $ \para{P}_{5} $ & $ \frac{1}{6} $ & $ \para{P}_{2} $ &  $ \frac{1}{22} $ & $w_{2}w_{4}w_{3}w_{1}$ & $ \frac{ \zeta ( 2 ) } { \zeta ( 5 ) } $ & $2$ & $1$ & $1$ & $162$ & $22$ \\ \hline 
 $ \para{P}_{5} $ & $ \frac{7}{18} $ & $ \para{P}_{6} $ &  $ \frac{1}{3} $ & $w_{6}$ & $ 1 $ & $1$ & $0$ & $1$ & $9$ & $1$ \\ \hline 
 
  \end{longtable}
\begin{thm}  \label{E6_id}
 Let $f^{0} \in I_{\para{P}_i}(s)$ be the normalized spherical section then: 
\begin{enumerate} 
\item 
$\leadingterm{ -1 }^{\para{P}_{2}} (f^0,\frac{5}{22},g)=\frac{1}{11} \times \frac{ \text{R} \zeta ( 3 ) } { \zeta ( 6 ) \zeta ( 8 ) } \times \leadingterm{ 0 }^{\para{P}_{1}} (f^0,0,g). 
 $\item 
$\leadingterm{ -1 }^{\para{P}_{2}} (f^0,\frac{5}{22},g)=\frac{1}{11} \times \frac{ \text{R} \zeta ( 3 ) } { \zeta ( 6 ) \zeta ( 8 ) } \times \leadingterm{ 0 }^{\para{P}_{6}} (f^0,0,g). 
 $\item 
$\leadingterm{ -1 }^{\para{P}_{3}} (f^0,\frac{7}{18},g)=\frac{1}{9} \times \frac{ \text{R} \zeta ( 3 ) } { \zeta ( 6 ) \zeta ( 8 ) } \times \leadingterm{ 0 }^{\para{P}_{1}} (f^0,\frac{1}{3},g).
 $\item 
$\leadingterm{ -2 }^{\para{P}_{3}} (f^0,\frac{1}{6},g)=\frac{11}{81} \times \frac{ \text{R} \zeta ( 2 ) } { \zeta ( 5 ) \zeta ( 6 ) } \times \leadingterm{ -1 }^{\para{P}_{2}} (f^0,\frac{1}{22},g).
 $\item 
$\leadingterm{ -1 }^{\para{P}_{3}} (f^0,\frac{5}{18},g)=\frac{1}{9} \times \frac{ \text{R} \zeta ( 2 ) \zeta ( 3 ) } { \zeta ( 5 ) \zeta ( 6 ) \zeta ( 8 ) } \times \leadingterm{ 0 }^{\para{P}_{6}} (f^0,\frac{1}{12},g).
 $\item 
$\leadingterm{ -2 }^{\para{P}_{4}} (f^0,\frac{5}{14},g)=\frac{12}{49} \times \frac{ \text{R} \zeta ( 2 ) ^{ 2 } \zeta ( 3 ) } { \zeta ( 4 ) \zeta ( 5 ) \zeta ( 6 ) \zeta ( 8 ) } \times \leadingterm{ -1 }^{\para{P}_{1}} (f^0,\frac{1}{4},g).
 $\item 
$\leadingterm{ -2 }^{\para{P}_{4}} (f^0,\frac{5}{14},g)=\frac{11}{49} \times \frac{ \text{R} \zeta ( 2 ) \zeta ( 3 ) } { \zeta ( 4 ) \zeta ( 5 ) \zeta ( 6 ) } \times \leadingterm{ -1 }^{\para{P}_{2}} (f^0,\frac{7}{22},g).
 $\item 
$\leadingterm{ -3 }^{\para{P}_{4}} (f^0,\frac{3}{14},g)=\frac{81}{343} \times \frac{ \text{R} \zeta ( 2 ) } { \zeta ( 4 ) \zeta ( 5 ) } \times \leadingterm{ -2 }^{\para{P}_{3}} (f^0,\frac{1}{6},g).
 $\item 
$\leadingterm{ -3 }^{\para{P}_{4}} (f^0,\frac{3}{14},g)=\frac{81}{343} \times \frac{ \text{R} \zeta ( 2 ) } { \zeta ( 4 ) \zeta ( 5 ) } \times \leadingterm{ -2 }^{\para{P}_{5}} (f^0,\frac{1}{6},g).
 $\item 
$\leadingterm{ -2 }^{\para{P}_{4}} (f^0,\frac{5}{14},g)=\frac{12}{49} \times \frac{ \text{R} \zeta ( 2 ) ^{ 2 } \zeta ( 3 ) } { \zeta ( 4 ) \zeta ( 5 ) \zeta ( 6 ) \zeta ( 8 ) } \times \leadingterm{ -1 }^{\para{P}_{6}} (f^0,\frac{1}{4},g).
 $\item 
$\leadingterm{ -1 }^{\para{P}_{5}} (f^0,\frac{5}{18},g)=\frac{1}{9} \times \frac{ \text{R} \zeta ( 2 ) \zeta ( 3 ) } { \zeta ( 5 ) \zeta ( 6 ) \zeta ( 8 ) } \times \leadingterm{ 0 }^{\para{P}_{1}} (f^0,\frac{1}{12},g).
 $\item 
$\leadingterm{ -2 }^{\para{P}_{5}} (f^0,\frac{1}{6},g)=\frac{11}{81} \times \frac{ \text{R} \zeta ( 2 ) } { \zeta ( 5 ) \zeta ( 6 ) } \times \leadingterm{ -1 }^{\para{P}_{2}} (f^0,\frac{1}{22},g).
 $\item 
$\leadingterm{ -1 }^{\para{P}_{5}} (f^0,\frac{7}{18},g)=\frac{1}{9} \times \frac{ \text{R} \zeta ( 3 ) } { \zeta ( 6 ) \zeta ( 8 ) } \times \leadingterm{ 0 }^{\para{P}_{6}} (f^0,\frac{1}{3},g).
 $\end{enumerate} 
 \end{thm}
However in that case we are able to find another positive admissible data
which are not special and cannot be derived from the list above.
\begin{align*}
(\para{P}_4,\frac{1}{7},\para{P}_3,0,w_{3}w_{4}w_{1}w_{3}w_{2}w_{4}w_{5}w_{6}) \\
(\para{P}_4,\frac{1}{7},\para{P}_5,0,w_{5}w_{6}w_{4}w_{5}w_{2}w_{4}w_{3}w_{1}
)
\end{align*}
The following Theorem together with Theorem (\ref{E6_id}) give the complete list of the identities corresponds to positive admissible data.
 \begin{thm}  
 Let $f^{0} \in I_{\para{P}_i}(s)$ be the normalized spherical section then: 
\begin{enumerate} 
\item 
$\leadingterm{ -1 }^{\para{P}_{1}} (f^0,\frac{1}{4},g)=\frac{11}{12} \times \frac{ \zeta ( 8 ) } { \zeta ( 2 ) } \times \leadingterm{ -1 }^{\para{P}_{2}} (f^0,\frac{7}{22},g).
 $\item 
$\leadingterm{ 0 }^{\para{P}_{1}} (f^0,0,g)=1 \times \leadingterm{ 0 }^{\para{P}_{6}} (f^0,0,g).
 $\item 
$\leadingterm{ -1 }^{\para{P}_{1}} (f^0,\frac{1}{4},g)=1 \times  \leadingterm{ -1 }^{\para{P}_{6}} (f^0,\frac{1}{4},g).
 $\item 
$\leadingterm{ -3 }^{\para{P}_{4}} (f^0,\frac{3}{14},g)=\frac{11}{343} \times \frac{ R^{ 2 } \zeta ( 2 ) ^{ 2 } } { \zeta ( 4 ) \zeta ( 5 ) ^{ 2 } \zeta ( 6 ) } \times \leadingterm{ -1 }^{\para{P}_{2}} (f^0,\frac{1}{22},g).
 $\item 
$\leadingterm{ -1 }^{\para{P}_{2}} (f^0,\frac{7}{22},g)=\frac{12}{11} \times \frac{ \zeta ( 2 ) } { \zeta ( 8 ) } \times \leadingterm{ -1 }^{\para{P}_{6}} (f^0,\frac{1}{4},g).
 $\item 
$\leadingterm{ -1 }^{\para{P}_{4}} (f^0,\frac{1}{7},g)=\frac{1}{14} \times \frac{ R \zeta ( 2 ) \zeta ( \frac{3}{2} ) \zeta ( \frac{1}{2} ) } { \zeta ( 4 ) \zeta ( 5 ) \zeta ( \frac{7}{2} ) \zeta ( \frac{9}{2} ) } \times \leadingterm{ 0 }^{\para{P}_{3}} (f^0,0,g).
 $\item 
$\leadingterm{ 0 }^{\para{P}_{3}} (f^0,0,g)=1 \times  \leadingterm{ 0 }^{\para{P}_{5}} (f^0,0,g).
 $\item 
$\leadingterm{ -2 }^{\para{P}_{3}} (f^0,\frac{1}{6},g)=1  \times \leadingterm{ -2 }^{\para{P}_{5}} (f^0,\frac{1}{6},g).
 $\item 
$\leadingterm{ -1 }^{\para{P}_{4}} (f^0,\frac{1}{7},g)=\frac{1}{14} \times \frac{ R \zeta ( 2 ) \zeta ( \frac{3}{2} ) \zeta ( \frac{1}{2} ) } { \zeta ( 4 ) \zeta ( 5 ) \zeta ( \frac{7}{2} ) \zeta ( \frac{9}{2} ) } \times \leadingterm{ 0 }^{\para{P}_{5}} (f^0,0,g).
 $
\end{enumerate} 
 \end{thm}
\begin{landscape}
 \begin{longtable}{|c|c|c|c|c|c|c|c|c|c|c|} 
\hline$ \para{P}_{i} $ & $ s $ & $ \para{P}_{j} $ &  $ t $ & $w$ & $ \frac{h_3(\chi_t)}{h_3(\chi_s)} $ & $d_{\para{P}_i}(\chi_s)$ & $d_{\para{P}_j}(\chi_t)$ & $d$ & $\epsilon_p$ & $\epsilon_q$ \\ \hline 
$ \para{P}_{1} $ & $ \frac{1}{4} $ & $ \para{P}_{2} $ &  $ \frac{7}{22} $ & $w_{3}w_{2}w_{1}$ & $ \frac{ \zeta ( 3 ) } { \zeta ( 2 ) } $ & $1$ & $1$ & $0$ & $12$ & $11$ \\ \hline 
$ \para{P}_{1} $ & $ 0 $ & $ \para{P}_{6} $ &  $ 0 $ & $w_{5}w_{6}w_{4}w_{5}w_{3}w_{4}w_{1}w_{3}$ & $ 1 $ & $0$ & $0$ & $0$ & $1$ & $1$ \\ \hline 
$ \para{P}_{1} $ & $ \frac{1}{4} $ & $ \para{P}_{6} $ &  $ \frac{1}{4} $ & $w_{6}w_{5}w_{3}w_{1}$ & $ 1 $ & $1$ & $1$ & $0$ & $12$ & $12$ \\ \hline 
$ \para{P}_{4} $ & $ \frac{3}{14} $ & $ \para{P}_{2} $ &  $ \frac{1}{22} $ & $w_{2}w_{4}w_{5}w_{6}w_{3}w_{1}$ & $ \frac{ \zeta ( 2 ) ^{ 2 } } { \zeta ( 3 ) \zeta ( 5 ) } $ & $3$ & $1$ & $2$ & $686$ & $22$ \\ \hline 
$ \para{P}_{2} $ & $ \frac{7}{22} $ & $ \para{P}_{6} $ &  $ \frac{1}{4} $ & $w_{6}w_{5}w_{2}$ & $ \frac{ \zeta ( 2 ) } { \zeta ( 3 ) } $ & $1$ & $1$ & $0$ & $11$ & $12$ \\ \hline 
$ \para{P}_{4} $ & $ \frac{1}{7} $ & $ \para{P}_{3} $ &  $ 0 $ & $w_{3}w_{4}w_{1}w_{3}w_{2}w_{4}w_{5}w_{6}$ & $ \frac{ \zeta ( \frac{3}{2} ) \zeta ( \frac{1}{2} ) \zeta ( 2 ) } { \zeta ( \frac{7}{2} ) \zeta ( 3 ) \zeta ( \frac{9}{2} ) } $ & $1$ & $0$ & $1$ & $14$ & $1$ \\ \hline 
$ \para{P}_{3} $ & $ 0 $ & $ \para{P}_{5} $ &  $ 0 $ & $w_{4}w_{5}w_{6}w_{2}w_{4}w_{5}w_{3}w_{4}w_{1}w_{3}w_{2}w_{4}$ & $ 1 $ & $0$ & $0$ & $0$ & $1$ & $1$ \\ \hline 
$ \para{P}_{3} $ & $ \frac{1}{6} $ & $ \para{P}_{5} $ &  $ \frac{1}{6} $ & $w_{5}w_{6}w_{1}w_{3}$ & $ 1 $ & $2$ & $2$ & $0$ & $162$ & $162$ \\ \hline 
$ \para{P}_{4} $ & $ \frac{1}{7} $ & $ \para{P}_{5} $ &  $ 0 $ & $w_{5}w_{6}w_{4}w_{5}w_{2}w_{4}w_{3}w_{1}$ & $ \frac{ \zeta ( \frac{3}{2} ) \zeta ( 2 ) \zeta ( \frac{1}{2} ) } { \zeta ( \frac{7}{2} ) \zeta ( \frac{9}{2} ) \zeta ( 3 ) } $ & $1$ & $0$ & $1$ & $14$ & $1$ \\ \hline 

 \end{longtable}
\end{landscape}

\newpage
Therefore we get the follows:
\begin{figure}[H]
\begin{center}
\begin{tikzpicture}[thin,scale=2]
\tikzset{
  text style/.style={
    sloped, 
    text=black,
    font=\tiny,
    above
  }
}
\matrix (m) [matrix of math nodes, row sep=1em, column sep=3em]
    { 		
        & |[name=e12]| \text{I}_{\para{P}_1}(0) &    &   |[name=e22]| \text{I}_{\para{P}_1}(\frac{1}{4})
      \\
      |[name=e11]| \text{I}_{\para{P}_2}(\frac{5}{22})&     & |[name=e21]| \text{I}_{\para{P}_4}(\frac{5}{14})&   |[name=e23]| \text{I}_{\para{P}_2}(\frac{7}{22})
     \\
        & |[name=e13]| \text{I}_{\para{P}_6}(0) &  &  |[name=e24]| \text{I}_{\para{P}_6}(\frac{1}{4})
     \\
      |[name=e31]| \text{I}_{\para{P}_3}(\frac{7}{18})&   |[name=e32]| \text{I}_{\para{P}_1}(\frac{1}{3}) &   |[name=e41]| \text{I}_{\para{P}_3}(\frac{5}{18})&   |[name=e42]| \text{I}_{\para{P}_6}(\frac{1}{12})
     \\ 
      |[name=e51]| \text{I}_{\para{P}_5}(\frac{5}{18})&   |[name=e52]| \text{I}_{\para{P}_1}(\frac{1}{12}) &   |[name=e61]| \text{I}_{\para{P}_5}(\frac{7}{18})&   |[name=e62]| \text{I}_{\para{P}_6}(\frac{1}{3})
\\
	  &  &|[name=e72]| \text{I}_{\para{P}_3}(\frac{1}{6})   &  &  
      \\
	 & |[name=e71]| \text{I}_{\para{P}_4}(\frac{3}{14}) &     &  |[name=e74]| \text{I}_{\para{P}_2}(\frac{1}{22}) &  
      \\
    &    &|[name=e73]| \text{I}_{\para{P}_5}(\frac{1}{6}) &  &  &   \\
    &  &  |[name=e82]| \text{I}_{\para{P}_5}(0)\\
    &    |[name=e81]| \text{I}_{\para{P}_4}(\frac{1}{7})\\
    &  &  |[name=e83]| \text{I}_{\para{P}_3}(0)\\
      
    \\};
\draw[->]      
		(e11)  edge [-]
		node[text style,above]{}  (e12)
		(e11)  edge [-]
		node[text style,above]{}  (e13)
		(e21)  edge [-]
		node[text style,above]{}  (e22)
		(e21)  edge [-]
		node[text style,above]{}  (e23)
		(e21)  edge [-]
		node[text style,above]{}  (e24)
		(e31)  edge [-]
		node[text style,above]{}  (e32)
		(e41)  edge [-]
		node[text style,above]{}  (e42)
		(e51)  edge [-]
		node[text style,above]{}  (e52)
		(e61)  edge [-]
		node[text style,above]{}  (e62)
		
		(e71)  edge [-]
		node[text style,above]{}  (e72)
		(e71)  edge [-]
		node[text style,above]{}  (e73)
		(e72)  edge [-]
		node[text style,above]{}  (e74)
		(e73)  edge [-]
		node[text style,above]{}  (e74)
		(e81)  edge [-]
		node[text style,above]{}  (e83)
		(e81)  edge [-]
		node[text style,above]{}  (e82)
            ;
         
\end{tikzpicture}
\label{digram:E6s}
\end{center}
\end{figure} 

\newpage
\lhead{Chapter \ref{ch3:ids}. \emph{Identities- Group of type $E_{7}$}}
\section{Group of type $E_{7}$} \label{section:E7}
\begin{prop}\label{E7_mod_char} 
The matrix $B$ for $G=E_7$ is
 $$B= \left(\begin{array}{rrrrrrr}
17 & 7 & 2 & 3 & 5 & 8 & 12 \\
10 & 14 & 6 & 2 & 5 & 8 & 11 \\
10 & 7 & 11 & 3 & 5 & 7 & 9 \\
7 & 7 & 6 & 8 & 5 & 6 & 7 \\
8 & 7 & 6 & 4 & 10 & 8 & 9 \\
9 & 7 & 6 & 4 & 3 & 13 & 12 \\
10 & 7 & 6 & 4 & 3 & 2 & 18
\end{array}\right) $$ 
 \end{prop} \ 
 \begin{prop} 
 The special admissible  data are as follows: 
 \begin{figure}[H] \begin{center} 
 
\begin{tabular}{|c|c|c|c|c|c|c|c|} \hline
$[s_0,t_0]$ & $ \para{P}_{1}$ & $ \para{P}_{2}$ & $ \para{P}_{3}$ & $ \para{P}_{4}$ & $ \para{P}_{5}$ & $ \para{P}_{6}$ & $ \para{P}_{7}$ \\ \hline
$\para{P}_{1}$ &  & $\left[\frac{3}{34}, \frac{3}{14}\right]$ & $\left[\frac{13}{34}, \frac{9}{22}\right]$ & $\left[\frac{11}{34}, \frac{3}{8}\right]$ & $\left[\frac{7}{34}, \frac{3}{10}\right]$ & $\left[\frac{1}{34}, \frac{5}{26}\right]$ & $\left[\frac{7}{34}, \frac{1}{18}\right]$ \\ \hline
$\para{P}_{2}$ & $\left[\frac{3}{14}, \frac{3}{34}\right]$ &  & $\left[\frac{1}{14}, \frac{3}{22}\right]$ & $\left[\frac{5}{14}, \frac{3}{8}\right]$ & $\left[\frac{1}{7}, \frac{1}{5}\right]$ & $\left[\frac{1}{14}, \frac{1}{26}\right]$ & $\left[\frac{2}{7}, \frac{1}{9}\right]$ \\ \hline
$\para{P}_{3}$ & $\left[\frac{9}{22}, \frac{13}{34}\right]$ & $\left[\frac{3}{22}, \frac{1}{14}\right]$ &  & $\left[\frac{5}{22}, \frac{1}{4}\right]$ & $\left[\frac{1}{22}, \frac{1}{10}\right]$ & $\left[\frac{3}{22}, \frac{1}{26}\right]$ & $\left[\frac{7}{22}, \frac{1}{6}\right]$ \\ \hline
$\para{P}_{4}$ & $\left[\frac{3}{8}, \frac{11}{34}\right]$ & $\left[\frac{3}{8}, \frac{5}{14}\right]$ & $\left[\frac{1}{4}, \frac{5}{22}\right]$ &  & $\left[\frac{1}{8}, \frac{1}{10}\right]$ & $\left[\frac{1}{4}, \frac{5}{26}\right]$ & $\left[\frac{3}{8}, \frac{5}{18}\right]$ \\ \hline
$\para{P}_{5}$ & $\left[\frac{3}{10}, \frac{7}{34}\right]$ & $\left[\frac{1}{5}, \frac{1}{7}\right]$ & $\left[\frac{1}{10}, \frac{1}{22}\right]$ & $\left[\frac{1}{10}, \frac{1}{8}\right]$ &  & $\left[\frac{3}{10}, \frac{7}{26}\right]$ & $\left[\frac{2}{5}, \frac{1}{3}\right]$ \\ \hline
$\para{P}_{6}$ & $\left[\frac{5}{26}, \frac{1}{34}\right]$ & $\left[\frac{1}{26}, \frac{1}{14}\right]$ & $\left[\frac{1}{26}, \frac{3}{22}\right]$ & $\left[\frac{5}{26}, \frac{1}{4}\right]$ & $\left[\frac{7}{26}, \frac{3}{10}\right]$ &  & $\left[\frac{11}{26}, \frac{7}{18}\right]$ \\ \hline
$\para{P}_{7}$ & $\left[\frac{1}{18}, \frac{7}{34}\right]$ & $\left[\frac{1}{9}, \frac{2}{7}\right]$ & $\left[\frac{1}{6}, \frac{7}{22}\right]$ & $\left[\frac{5}{18}, \frac{3}{8}\right]$ & $\left[\frac{1}{3}, \frac{2}{5}\right]$ & $\left[\frac{7}{18}, \frac{11}{26}\right]$ &  \\ \hline
\end{tabular} 
  \end{center} 
 \end{figure} 
 \end{prop} \begin{prop} It holds that : 
 \begin{align*} 
 A_{w_{\para{P}_{1} }} = &\frac{R^{6}}{\zeta ( 2 ) \zeta ( 4 ) \zeta ( 6 )^{2} \zeta ( 8 ) \zeta ( 10 )}& w_{\para{P}_{1} } =&  w_{7}w_{6}w_{5}w_{4}w_{3}w_{2}w_{4}w_{5}w_{6}w_{7}w_{6}w_{5}w_{4}w_{3}
 \\
  & & &
 w_{2}w_{4}w_{5}w_{6}w_{5}w_{4}w_{3}w_{2}w_{4}w_{5}w_{4}w_{3}w_{2}w_{4}w_{3}w_{2} \\ 
 A_{w_{\para{P}_{2} }} = &\frac{R^{6}}{\zeta ( 2 ) \zeta ( 3 ) \zeta ( 4 ) \zeta ( 5 ) \zeta ( 6 ) \zeta ( 7 )}& w_{\para{P}_{2} } =&  w_{1}w_{3}w_{4}w_{5}w_{6}w_{7}w_{1}w_{3}w_{4}w_{5}w_{6}w_{1}w_{3}
 w_{4}w_{5}w_{1}w_{3}w_{4}w_{1}
 \\
  & & &
 w_{3}w_{1} \\ 
 A_{w_{\para{P}_{3} }} = &\frac{R^{6}}{\zeta ( 2 )^{2} \zeta ( 3 ) \zeta ( 4 ) \zeta ( 5 ) \zeta ( 6 )}& w_{\para{P}_{3} } =&  w_{2}w_{4}w_{5}w_{6}w_{7}w_{2}w_{4}w_{5}w_{6}w_{2}w_{4}w_{5}w_{2}w_{4}w_{2}w_{1} \\ 
 A_{w_{\para{P}_{4} }} = &\frac{R^{6}}{\zeta ( 2 )^{3} \zeta ( 3 )^{2} \zeta ( 4 )}& w_{\para{P}_{4} } =&  w_{5}w_{6}w_{7}w_{5}w_{6}w_{5}w_{1}w_{3}w_{2}w_{1} \\ 
 A_{w_{\para{P}_{5} }} = &\frac{R^{6}}{\zeta ( 2 )^{2} \zeta ( 3 )^{2} \zeta ( 4 ) \zeta ( 5 )}& w_{\para{P}_{5} } =&  w_{6}w_{7}w_{6}w_{3}w_{4}w_{1}w_{3}w_{2}w_{4}w_{1}w_{3}w_{2}w_{1} \\ 
 A_{w_{\para{P}_{6} }} = &\frac{R^{6}}{\zeta ( 2 )^{2} \zeta ( 4 ) \zeta ( 5 ) \zeta ( 6 ) \zeta ( 8 )}& w_{\para{P}_{6} } =&  w_{7}w_{2}w_{4}w_{5}w_{3}w_{4}w_{1}w_{3}w_{2}w_{4}w_{5}w_{3}w_{4}w_{1}w_{3}w_{2}w_{4}w_{1}w_{3}
 \\
  & & &
 w_{2}w_{1} \\ 
 A_{w_{\para{P}_{7} }} = &\frac{R^{6}}{\zeta ( 2 ) \zeta ( 5 ) \zeta ( 6 ) \zeta ( 8 ) \zeta ( 12 ) \zeta ( 9 )}& w_{\para{P}_{7} } =&  w_{1}w_{3}w_{4}w_{5}w_{6}w_{2}w_{4}w_{5}w_{3}w_{4}w_{1}w_{3}w_{2}w_{4}w_{5}w_{6}w_{2}w_{4}w_{5}w_{3}\\
 & & &
 w_{4}w_{1}w_{3}w_{2}w_{4}w_{5}w_{3}w_{4}w_{1}w_{3}w_{2}w_{4}w_{1}w_{3}w_{2}w_{1} \\ 
  \end{align*} 
 \end{prop} 
 \begin{longtable}{|c|c|c|c|c|c|c|c|c|c|c|} 
\hline$ \para{P}_{i} $ & $ s $ & $ \para{P}_{j} $ &  $ t $ & $w$ & $ \frac{h_3(\chi_t)}{h_3(\chi_s)} $ & $d_{\para{P}_i}(\chi_s)$ & $d_{\para{P}_j}(\chi_t)$ & $d$ & $\epsilon_p$ & $\epsilon_q$ \\ \hline 
$ \para{P}_{1} $ & $ \frac{7}{34} $ & $ \para{P}_{7} $ &  $ \frac{1}{18} $ & $w_{7}w_{6}w_{5}w_{3}w_{4}w_{1}w_{3}w_{2}w_{4}$ & $ \frac{ \zeta ( 2 ) \zeta ( 5 ) } { \zeta ( 4 ) \zeta ( 8 ) } $ & $1$ & $1$ & $0$ & $17$ & $18$ \\ \hline 
$ \para{P}_{2} $ & $ \frac{3}{14} $ & $ \para{P}_{1} $ &  $ \frac{3}{34} $ & $w_{1}w_{3}w_{4}w_{5}w_{6}w_{7}$ & $ \frac{ \zeta ( 2 ) } { \zeta ( 7 ) } $ & $1$ & $0$ & $1$ & $14$ & $1$ \\ \hline 
$ \para{P}_{2} $ & $ \frac{1}{14} $ & $ \para{P}_{6} $ &  $ \frac{1}{26} $ & $w_{5}w_{6}w_{7}w_{4}w_{5}w_{6}w_{2}w_{4}w_{5}$ & $ \frac{ \zeta ( 2 ) } { \zeta ( 5 ) } $ & $1$ & $1$ & $0$ & $14$ & $26$ \\ \hline 
$ \para{P}_{2} $ & $ \frac{2}{7} $ & $ \para{P}_{7} $ &  $ \frac{1}{9} $ & $w_{7}w_{6}w_{5}w_{4}w_{3}w_{1}$ & $ \frac{ \zeta ( 2 ) } { \zeta ( 7 ) } $ & $1$ & $0$ & $1$ & $14$ & $1$ \\ \hline 
$ \para{P}_{3} $ & $ \frac{9}{22} $ & $ \para{P}_{1} $ &  $ \frac{13}{34} $ & $w_{1}$ & $ 1 $ & $1$ & $0$ & $1$ & $11$ & $1$ \\ \hline 
$ \para{P}_{3} $ & $ \frac{3}{22} $ & $ \para{P}_{2} $ &  $ \frac{1}{14} $ & $w_{2}w_{4}w_{5}w_{6}w_{7}$ & $ \frac{ \zeta ( 2 ) } { \zeta ( 6 ) } $ & $2$ & $1$ & $1$ & $242$ & $14$ \\ \hline 
$ \para{P}_{3} $ & $ \frac{3}{22} $ & $ \para{P}_{6} $ &  $ \frac{1}{26} $ & $w_{6}w_{7}w_{5}w_{6}w_{4}w_{5}w_{2}w_{4}$ & $ \frac{ \zeta ( 2 ) ^{ 2 } } { \zeta ( 5 ) \zeta ( 6 ) } $ & $2$ & $1$ & $1$ & $242$ & $26$ \\ \hline 
$ \para{P}_{3} $ & $ \frac{7}{22} $ & $ \para{P}_{7} $ &  $ \frac{1}{6} $ & $w_{7}w_{6}w_{5}w_{4}w_{2}$ & $ \frac{ \zeta ( 2 ) } { \zeta ( 6 ) } $ & $1$ & $0$ & $1$ & $11$ & $1$ \\ \hline 
$ \para{P}_{4} $ & $ \frac{3}{8} $ & $ \para{P}_{1} $ &  $ \frac{11}{34} $ & $w_{1}w_{3}$ & $ \frac{ \zeta ( 2 ) } { \zeta ( 3 ) } $ & $2$ & $1$ & $1$ & $64$ & $17$ \\ \hline 
$ \para{P}_{4} $ & $ \frac{3}{8} $ & $ \para{P}_{2} $ &  $ \frac{5}{14} $ & $w_{2}$ & $ 1 $ & $2$ & $1$ & $1$ & $64$ & $14$ \\ \hline 
$ \para{P}_{4} $ & $ \frac{1}{4} $ & $ \para{P}_{3} $ &  $ \frac{5}{22} $ & $w_{3}w_{1}$ & $ \frac{ \zeta ( 2 ) } { \zeta ( 3 ) } $ & $3$ & $2$ & $1$ & $1024$ & $242$ \\ \hline 
$ \para{P}_{4} $ & $ \frac{1}{8} $ & $ \para{P}_{5} $ &  $ \frac{1}{10} $ & $w_{5}w_{6}w_{7}$ & $ \frac{ \zeta ( 2 ) } { \zeta ( 4 ) } $ & $4$ & $3$ & $1$ & $49152$ & $6000$ \\ \hline 
$ \para{P}_{4} $ & $ \frac{1}{4} $ & $ \para{P}_{6} $ &  $ \frac{5}{26} $ & $w_{6}w_{7}w_{5}w_{6}$ & $ \frac{ \zeta ( 2 ) ^{ 2 } } { \zeta ( 3 ) \zeta ( 4 ) } $ & $3$ & $2$ & $1$ & $1024$ & $338$ \\ \hline 
$ \para{P}_{4} $ & $ \frac{3}{8} $ & $ \para{P}_{7} $ &  $ \frac{5}{18} $ & $w_{7}w_{6}w_{5}$ & $ \frac{ \zeta ( 2 ) } { \zeta ( 4 ) } $ & $2$ & $1$ & $1$ & $64$ & $18$ \\ \hline 
$ \para{P}_{5} $ & $ \frac{3}{10} $ & $ \para{P}_{1} $ &  $ \frac{7}{34} $ & $w_{1}w_{3}w_{4}w_{2}$ & $ \frac{ \zeta ( 2 ) } { \zeta ( 5 ) } $ & $2$ & $1$ & $1$ & $100$ & $17$ \\ \hline 
$ \para{P}_{5} $ & $ \frac{1}{5} $ & $ \para{P}_{2} $ &  $ \frac{1}{7} $ & $w_{2}w_{4}w_{3}w_{1}$ & $ \frac{ \zeta ( 2 ) } { \zeta ( 5 ) } $ & $2$ & $1$ & $1$ & $200$ & $28$ \\ \hline 
$ \para{P}_{5} $ & $ \frac{1}{10} $ & $ \para{P}_{3} $ &  $ \frac{1}{22} $ & $w_{3}w_{4}w_{1}w_{3}w_{2}w_{4}$ & $ \frac{ \zeta ( 2 ) ^{ 2 } } { \zeta ( 4 ) \zeta ( 5 ) } $ & $3$ & $2$ & $1$ & $6000$ & $726$ \\ \hline 
$ \para{P}_{5} $ & $ \frac{3}{10} $ & $ \para{P}_{6} $ &  $ \frac{7}{26} $ & $w_{6}w_{7}$ & $ \frac{ \zeta ( 2 ) } { \zeta ( 3 ) } $ & $2$ & $1$ & $1$ & $100$ & $13$ \\ \hline 
$ \para{P}_{5} $ & $ \frac{2}{5} $ & $ \para{P}_{7} $ &  $ \frac{1}{3} $ & $w_{7}w_{6}$ & $ \frac{ \zeta ( 2 ) } { \zeta ( 3 ) } $ & $1$ & $0$ & $1$ & $10$ & $1$ \\ \hline 
$ \para{P}_{6} $ & $ \frac{5}{26} $ & $ \para{P}_{1} $ &  $ \frac{1}{34} $ & $w_{1}w_{3}w_{4}w_{5}w_{2}w_{4}w_{3}w_{1}$ & $ \frac{ \zeta ( 4 ) \zeta ( 2 ) } { \zeta ( 5 ) \zeta ( 8 ) } $ & $2$ & $1$ & $1$ & $338$ & $34$ \\ \hline 
$ \para{P}_{6} $ & $ \frac{11}{26} $ & $ \para{P}_{7} $ &  $ \frac{7}{18} $ & $w_{7}$ & $ 1 $ & $1$ & $0$ & $1$ & $13$ & $1$ \\ \hline 

 \end{longtable}
 \begin{thm}  \label{E7_id}
 Let $f^{0} \in I_{\para{P}_i}(s)$ be the normalized spherical section then: 
\begin{enumerate} 
\item 
$\leadingterm{ -1 }^{\para{P}_{1}} (f^0,\frac{7}{34},g)=\frac{18}{17} \times \frac{ \zeta ( 2 ) \zeta ( 6 ) \zeta ( 10 ) } { \zeta ( 8 ) \zeta ( 12 ) \zeta ( 9 ) } \times \leadingterm{ -1 }^{\para{P}_{7}} (f^0,\frac{1}{18},g).
 $\item 
$\leadingterm{ -1 }^{\para{P}_{2}} (f^0,\frac{3}{14},g)=\frac{1}{14} \times \frac{ R \zeta ( 3 ) \zeta ( 5 ) } { \zeta ( 6 ) \zeta ( 8 ) \zeta ( 10 ) } \times \leadingterm{ 0 }^{\para{P}_{1}} (f^0,\frac{3}{34},g).
 $\item 
$\leadingterm{ -1 }^{\para{P}_{2}} (f^0,\frac{1}{14},g)=\frac{13}{7} \times \frac{ \zeta ( 3 ) \zeta ( 7 ) } { \zeta ( 5 ) \zeta ( 8 ) } \times \leadingterm{ -1 }^{\para{P}_{6}} (f^0,\frac{1}{26},g).
 $\item 
$\leadingterm{ -1 }^{\para{P}_{2}} (f^0,\frac{2}{7},g)=\frac{1}{14} \times \frac{ R \zeta ( 3 ) \zeta ( 4 ) } { \zeta ( 8 ) \zeta ( 12 ) \zeta ( 9 ) } \times \leadingterm{ 0 }^{\para{P}_{7}} (f^0,\frac{1}{9},g).
 $\item 
$\leadingterm{ -1 }^{\para{P}_{3}} (f^0,\frac{9}{22},g)=\frac{1}{11} \times \frac{ R \zeta ( 3 ) \zeta ( 5 ) } { \zeta ( 6 ) \zeta ( 8 ) \zeta ( 10 ) } \times \leadingterm{ 0 }^{\para{P}_{1}} (f^0,\frac{13}{34},g).
 $\item 
$\leadingterm{ -2 }^{\para{P}_{3}} (f^0,\frac{3}{22},g)=\frac{7}{121} \times \frac{ R \zeta ( 2 ) } { \zeta ( 6 ) \zeta ( 7 ) } \times \leadingterm{ -1 }^{\para{P}_{2}} (f^0,\frac{1}{14},g).
 $\item 
$\leadingterm{ -2 }^{\para{P}_{3}} (f^0,\frac{3}{22},g)=\frac{13}{121} \times \frac{ R \zeta ( 2 ) \zeta ( 3 ) } { \zeta ( 5 ) \zeta ( 6 ) \zeta ( 8 ) } \times \leadingterm{ -1 }^{\para{P}_{6}} (f^0,\frac{1}{26},g).
 $\item 
$\leadingterm{ -1 }^{\para{P}_{3}} (f^0,\frac{7}{22},g)=\frac{1}{11} \times \frac{ R \zeta ( 2 ) \zeta ( 3 ) \zeta ( 4 ) } { \zeta ( 6 ) \zeta ( 8 ) \zeta ( 12 ) \zeta ( 9 ) } \times \leadingterm{ 0 }^{\para{P}_{7}} (f^0,\frac{1}{6},g).
 $\item 
$\leadingterm{ -2 }^{\para{P}_{4}} (f^0,\frac{3}{8},g)=\frac{17}{64} \times \frac{ R \zeta ( 2 ) ^{ 2 } \zeta ( 3 ) } { \zeta ( 6 ) ^{ 2 } \zeta ( 8 ) \zeta ( 10 ) } \times \leadingterm{ -1 }^{\para{P}_{1}} (f^0,\frac{11}{34},g).
 $\item 
$\leadingterm{ -2 }^{\para{P}_{4}} (f^0,\frac{3}{8},g)=\frac{7}{32} \times \frac{ R \zeta ( 2 ) \zeta ( 3 ) } { \zeta ( 5 ) \zeta ( 6 ) \zeta ( 7 ) } \times \leadingterm{ -1 }^{\para{P}_{2}} (f^0,\frac{5}{14},g).
 $\item 
$\leadingterm{ -3 }^{\para{P}_{4}} (f^0,\frac{1}{4},g)=\frac{121}{512} \times \frac{ R \zeta ( 2 ) } { \zeta ( 5 ) \zeta ( 6 ) } \times \leadingterm{ -2 }^{\para{P}_{3}} (f^0,\frac{5}{22},g).
 $\item 
$\leadingterm{ -4 }^{\para{P}_{4}} (f^0,\frac{1}{8},g)=\frac{125}{1024} \times \frac{ R \zeta ( 2 ) } { \zeta ( 4 ) \zeta ( 5 ) } \times \leadingterm{ -3 }^{\para{P}_{5}} (f^0,\frac{1}{10},g).
 $\item 
$\leadingterm{ -3 }^{\para{P}_{4}} (f^0,\frac{1}{4},g)=\frac{169}{512} \times \frac{ R \zeta ( 2 ) ^{ 2 } \zeta ( 3 ) } { \zeta ( 4 ) \zeta ( 5 ) \zeta ( 6 ) \zeta ( 8 ) } \times \leadingterm{ -2 }^{\para{P}_{6}} (f^0,\frac{5}{26},g).
 $\item 
$\leadingterm{ -2 }^{\para{P}_{4}} (f^0,\frac{3}{8},g)=\frac{9}{32} \times \frac{ R \zeta ( 2 ) ^{ 2 } \zeta ( 3 ) ^{ 2 } } { \zeta ( 5 ) \zeta ( 6 ) \zeta ( 8 ) \zeta ( 12 ) \zeta ( 9 ) } \times \leadingterm{ -1 }^{\para{P}_{7}} (f^0,\frac{5}{18},g).
 $\item 
$\leadingterm{ -2 }^{\para{P}_{5}} (f^0,\frac{3}{10},g)=\frac{17}{100} \times \frac{ R \zeta ( 2 ) \zeta ( 3 ) ^{ 2 } } { \zeta ( 6 ) ^{ 2 } \zeta ( 8 ) \zeta ( 10 ) } \times \leadingterm{ -1 }^{\para{P}_{1}} (f^0,\frac{7}{34},g).
 $\item 
$\leadingterm{ -2 }^{\para{P}_{5}} (f^0,\frac{1}{5},g)=\frac{7}{50} \times \frac{ R \zeta ( 2 ) \zeta ( 3 ) } { \zeta ( 5 ) \zeta ( 6 ) \zeta ( 7 ) } \times \leadingterm{ -1 }^{\para{P}_{2}} (f^0,\frac{1}{7},g).
 $\item 
$\leadingterm{ -3 }^{\para{P}_{5}} (f^0,\frac{1}{10},g)=\frac{121}{1000} \times \frac{ R \zeta ( 2 ) \zeta ( 3 ) } { \zeta ( 4 ) \zeta ( 5 ) \zeta ( 6 ) } \times \leadingterm{ -2 }^{\para{P}_{3}} (f^0,\frac{1}{22},g).
 $\item 
$\leadingterm{ -2 }^{\para{P}_{5}} (f^0,\frac{3}{10},g)=\frac{13}{100} \times \frac{ R \zeta ( 3 ) } { \zeta ( 6 ) \zeta ( 8 ) } \times \leadingterm{ -1 }^{\para{P}_{6}} (f^0,\frac{7}{26},g).
 $\item 
$\leadingterm{ -1 }^{\para{P}_{5}} (f^0,\frac{2}{5},g)=\frac{1}{10} \times \frac{ R \zeta ( 2 ) \zeta ( 3 ) \zeta ( 4 ) } { \zeta ( 6 ) \zeta ( 8 ) \zeta ( 12 ) \zeta ( 9 ) } \times \leadingterm{ 0 }^{\para{P}_{7}} (f^0,\frac{1}{3},g).
 $\item 
$\leadingterm{ -2 }^{\para{P}_{6}} (f^0,\frac{5}{26},g)=\frac{17}{169} \times \frac{ R \zeta ( 2 ) \zeta ( 4 ) } { \zeta ( 6 ) \zeta ( 8 ) \zeta ( 10 ) } \times \leadingterm{ -1 }^{\para{P}_{1}} (f^0,\frac{1}{34},g).
 $\item 
$\leadingterm{ -1 }^{\para{P}_{6}} (f^0,\frac{11}{26},g)=\frac{1}{13} \times \frac{ R \zeta ( 4 ) } { \zeta ( 12 ) \zeta ( 9 ) } \times \leadingterm{ 0 }^{\para{P}_{7}} (f^0,\frac{7}{18},g).
 $\end{enumerate} 
 \end{thm}

 \begin{landscape} 

 \begin{longtable}{|c|c|c|c|c|c|c|c|c|c|c|} 
\hline$ \para{P}_{i} $ & $ s $ & $ \para{P}_{j} $ &  $ t $ & $w$ & $ \frac{h_3(\chi_t)}{h_3(\chi_s)} $ & $d_{\para{P}_i}(\chi_s)$ & $d_{\para{P}_j}(\chi_t)$ & $d$ & $\epsilon_p$ & $\epsilon_q$ \\ \hline 
$ \para{P}_{3} $ & $ \frac{5}{22} $ & $ \para{P}_{1} $ &  $ \frac{1}{34} $ & $w_{1}w_{3}w_{4}w_{5}w_{6}w_{7}w_{2}w_{4}w_{5}w_{6}$ & $ \frac{ \zeta ( 2 ) ^{ 2 } } { \zeta ( 5 ) \zeta ( 8 ) } $ & $2$ & $1$ & $1$ & $242$ & $34$ \\ \hline 
$ \para{P}_{1} $ & $ \frac{11}{34} $ & $ \para{P}_{2} $ &  $ \frac{5}{14} $ & $w_{3}w_{2}w_{1}$ & $ \frac{ \zeta ( 3 ) } { \zeta ( 2 ) } $ & $1$ & $1$ & $0$ & $17$ & $14$ \\ \hline 
$ \para{P}_{4} $ & $ \frac{1}{4} $ & $ \para{P}_{1} $ &  $ \frac{1}{34} $ & $w_{1}w_{3}w_{4}w_{5}w_{6}w_{7}w_{2}w_{4}w_{5}w_{6}w_{3}w_{1}$ & $ \frac{ \zeta ( 2 ) ^{ 3 } } { \zeta ( 3 ) \zeta ( 5 ) \zeta ( 8 ) } $ & $3$ & $1$ & $2$ & $1024$ & $34$ \\ \hline 
$ \para{P}_{1} $ & $ \frac{7}{34} $ & $ \para{P}_{6} $ &  $ \frac{7}{26} $ & $w_{6}w_{7}w_{2}w_{4}w_{3}w_{1}$ & $ \frac{ \zeta ( 5 ) } { \zeta ( 3 ) } $ & $1$ & $1$ & $0$ & $17$ & $13$ \\ \hline 
$ \para{P}_{1} $ & $ \frac{11}{34} $ & $ \para{P}_{7} $ &  $ \frac{5}{18} $ & $w_{7}w_{6}w_{5}w_{3}w_{1}$ & $ \frac{ \zeta ( 3 ) } { \zeta ( 4 ) } $ & $1$ & $1$ & $0$ & $17$ & $18$ \\ \hline 
$ \para{P}_{5} $ & $ \frac{3}{20} $ & $ \para{P}_{2} $ &  $ \frac{1}{28} $ & $w_{2}w_{4}w_{5}w_{3}w_{4}w_{1}w_{3}w_{2}w_{4}w_{5}w_{6}w_{7}$ & $ \frac{ \zeta ( \frac{3}{2} ) ^{ 2 } \zeta ( 2 ) } { \zeta ( \frac{9}{2} ) \zeta ( 3 ) \zeta ( \frac{13}{2} ) } $ & $1$ & $0$ & $1$ & $20$ & $1$ \\ \hline 
$ \para{P}_{2} $ & $ \frac{5}{14} $ & $ \para{P}_{7} $ &  $ \frac{5}{18} $ & $w_{7}w_{6}w_{5}w_{2}$ & $ \frac{ \zeta ( 2 ) } { \zeta ( 4 ) } $ & $1$ & $1$ & $0$ & $14$ & $18$ \\ \hline 
$ \para{P}_{4} $ & $ \frac{1}{8} $ & $ \para{P}_{3} $ &  $ \frac{1}{22} $ & $w_{3}w_{4}w_{1}w_{3}w_{2}w_{4}w_{5}w_{6}w_{7}$ & $ \frac{ \zeta ( 2 ) ^{ 3 } } { \zeta ( 4 ) ^{ 2 } \zeta ( 5 ) } $ & $4$ & $2$ & $2$ & $49152$ & $726$ \\ \hline 
$ \para{P}_{3} $ & $ \frac{5}{22} $ & $ \para{P}_{6} $ &  $ \frac{5}{26} $ & $w_{6}w_{7}w_{5}w_{6}w_{1}w_{3}$ & $ \frac{ \zeta ( 2 ) } { \zeta ( 4 ) } $ & $2$ & $2$ & $0$ & $242$ & $338$ \\ \hline 
$ \para{P}_{4} $ & $ \frac{1}{12} $ & $ \para{P}_{5} $ &  $ \frac{1}{30} $ & $w_{5}w_{6}w_{7}w_{4}w_{5}w_{6}w_{3}w_{4}w_{5}w_{1}w_{3}w_{4}w_{2}$ & $ \frac{ \zeta ( \frac{7}{3} ) \zeta ( \frac{4}{3} ) \zeta ( \frac{2}{3} ) } { \zeta ( \frac{8}{3} ) \zeta ( \frac{11}{3} ) \zeta ( \frac{14}{3} ) } $ & $1$ & $0$ & $1$ & $24$ & $1$ \\ \hline 
$ \para{P}_{4} $ & $ \frac{3}{16} $ & $ \para{P}_{6} $ &  $ \frac{1}{13} $ & $w_{6}w_{7}w_{5}w_{6}w_{4}w_{5}w_{2}w_{4}w_{3}w_{1}$ & $ \frac{ \zeta ( \frac{3}{2} ) \zeta ( 2 ) \zeta ( \frac{1}{2} ) } { \zeta ( \frac{9}{2} ) \zeta ( \frac{11}{2} ) \zeta ( 3 ) } $ & $1$ & $0$ & $1$ & $16$ & $1$ \\ \hline 
$ \para{P}_{5} $ & $ \frac{3}{10} $ & $ \para{P}_{7} $ &  $ \frac{1}{18} $ & $w_{7}w_{6}w_{5}w_{4}w_{3}w_{2}w_{4}w_{1}w_{3}$ & $ \frac{ \zeta ( 2 ) ^{ 2 } } { \zeta ( 4 ) \zeta ( 8 ) } $ & $2$ & $1$ & $1$ & $100$ & $18$ \\ \hline 
$ \para{P}_{6} $ & $ \frac{11}{26} $ & $ \para{P}_{7} $ &  $ \frac{7}{18} $ & $w_{7}$ & $ 1 $ & $1$ & $0$ & $1$ & $13$ & $1$ \\ \hline 

 \end{longtable}
 \end{landscape}
The following Theorem together with Theorem (\ref{E7_id}) give the complete list of the identities corresponds to positive admissible data.
 \begin{thm}  
 Let $f^{0} \in I_{\para{P}_i}(s)$ be the normalized spherical section then: 
\begin{enumerate} 
\item 
$\leadingterm{ -2 }^{\para{P}_{3}} (f^0,\frac{5}{22},g)=\frac{17}{121} \times \frac{ R \zeta ( 2 ) ^{ 2 } \zeta ( 3 ) } { \zeta ( 6 ) \zeta ( 8 ) ^{ 2 } \zeta ( 10 ) } \times \leadingterm{ -1 }^{\para{P}_{1}} (f^0,\frac{1}{34},g).
 $\item 
$\leadingterm{ -1 }^{\para{P}_{1}} (f^0,\frac{11}{34},g)=\frac{14}{17} \times \frac{ \zeta ( 6 ) \zeta ( 8 ) \zeta ( 10 ) } { \zeta ( 2 ) \zeta ( 5 ) \zeta ( 7 ) } \times \leadingterm{ -1 }^{\para{P}_{2}} (f^0,\frac{5}{14},g).
 $\item 
$\leadingterm{ -3 }^{\para{P}_{4}} (f^0,\frac{1}{4},g)=\frac{17}{512} \times \frac{ R^{ 2 } \zeta ( 2 ) ^{ 3 } \zeta ( 3 ) } { \zeta ( 6 ) ^{ 2 } \zeta ( 8 ) ^{ 2 } \zeta ( 10 ) \zeta ( 5 ) } \times \leadingterm{ -1 }^{\para{P}_{1}} (f^0,\frac{1}{34},g).
 $\item 
$\leadingterm{ -1 }^{\para{P}_{1}} (f^0,\frac{7}{34},g)=\frac{13}{17} \times \frac{ \zeta ( 6 ) \zeta ( 10 ) } { \zeta ( 2 ) \zeta ( 3 ) } \times \leadingterm{ -1 }^{\para{P}_{6}} (f^0,\frac{7}{26},g).
 $\item 
$\leadingterm{ -1 }^{\para{P}_{1}} (f^0,\frac{11}{34},g)=\frac{18}{17} \times \frac{ \zeta ( 6 ) \zeta ( 10 ) \zeta ( 3 ) } { \zeta ( 5 ) \zeta ( 12 ) \zeta ( 9 ) } \times \leadingterm{ -1 }^{\para{P}_{7}} (f^0,\frac{5}{18},g).
 $\item 
$\leadingterm{ -1 }^{\para{P}_{5}} (f^0,\frac{3}{20},g)=\frac{1}{20} \times \frac{ R \zeta ( 2 ) \zeta ( \frac{3}{2} ) ^{ 2 } } { \zeta ( 6 ) \zeta ( 7 ) \zeta ( \frac{9}{2} ) \zeta ( \frac{13}{2} ) } \times \leadingterm{ 0 }^{\para{P}_{2}} (f^0,\frac{1}{28},g).
 $\item 
$\leadingterm{ -1 }^{\para{P}_{2}} (f^0,\frac{5}{14},g)=\frac{9}{7} \times \frac{ \zeta ( 2 ) \zeta ( 3 ) \zeta ( 7 ) } { \zeta ( 8 ) \zeta ( 12 ) \zeta ( 9 ) } \times \leadingterm{ -1 }^{\para{P}_{7}} (f^0,\frac{5}{18},g).
 $\item 
$\leadingterm{ -4 }^{\para{P}_{4}} (f^0,\frac{1}{8},g)=\frac{121}{8192} \times \frac{ R^{ 2 } \zeta ( 2 ) ^{ 2 } \zeta ( 3 ) } { \zeta ( 4 ) ^{ 2 } \zeta ( 5 ) ^{ 2 } \zeta ( 6 ) } \times \leadingterm{ -2 }^{\para{P}_{3}} (f^0,\frac{1}{22},g).
 $\item 
$\leadingterm{ -2 }^{\para{P}_{3}} (f^0,\frac{5}{22},g)=\frac{169}{121} \times \frac{ \zeta ( 2 ) \zeta ( 3 ) } { \zeta ( 4 ) \zeta ( 8 ) } \times \leadingterm{ -2 }^{\para{P}_{6}} (f^0,\frac{5}{26},g).
 $\item 
$\leadingterm{ -1 }^{\para{P}_{4}} (f^0,\frac{1}{12},g)=\frac{1}{24} \times \frac{ R \zeta ( \frac{7}{3} ) \zeta ( \frac{4}{3} ) \zeta ( \frac{2}{3} ) } { \zeta ( 5 ) \zeta ( \frac{8}{3} ) \zeta ( \frac{11}{3} ) \zeta ( \frac{14}{3} ) } \times \leadingterm{ 0 }^{\para{P}_{5}} (f^0,\frac{1}{30},g).
 $\item 
$\leadingterm{ -1 }^{\para{P}_{4}} (f^0,\frac{3}{16},g)=\frac{1}{16} \times \frac{ R \zeta ( 2 ) \zeta ( 3 ) \zeta ( \frac{3}{2} ) \zeta ( \frac{1}{2} ) } { \zeta ( 5 ) \zeta ( 6 ) \zeta ( 8 ) \zeta ( \frac{9}{2} ) \zeta ( \frac{11}{2} ) } \times \leadingterm{ 0 }^{\para{P}_{6}} (f^0,\frac{1}{13},g).
 $\item 
$\leadingterm{ -2 }^{\para{P}_{5}} (f^0,\frac{3}{10},g)=\frac{9}{50} \times \frac{ R \zeta ( 2 ) ^{ 2 } \zeta ( 3 ) ^{ 2 } } { \zeta ( 6 ) \zeta ( 8 ) ^{ 2 } \zeta ( 12 ) \zeta ( 9 ) } \times \leadingterm{ -1 }^{\para{P}_{7}} (f^0,\frac{1}{18},g).
 $\item 
$\leadingterm{ -1 }^{\para{P}_{6}} (f^0,\frac{11}{26},g)=\frac{1}{13} \times \frac{ R \zeta ( 4 ) } { \zeta ( 12 ) \zeta ( 9 ) } \times \leadingterm{ 0 }^{\para{P}_{7}} (f^0,\frac{7}{18},g).
 $\end{enumerate} 
 \end{thm}
\begin{remk}
Observe that the following  positive admissible data that are not special. 
\begin{align*}
&(\para{P}_5,\frac{3}{10},\para{P}_7,\frac{1}{18},
w_{7}w_{6}w_{5}w_{4}w_{3}w_{2}w_{4}w_{1}w_{3})
\\
&(\para{P}_5,\frac{3}{20},\para{P}_2,\frac{1}{28},
w_{2}w_{4}w_{5}w_{3}w_{4}w_{1}w_{3}w_{2}w_{4}w_{5}w_{6}w_{7}
)\\
&(\para{P}_4,\frac{3}{16},\para{P}_6,\frac{1}{13},
w_{6}w_{7}w_{5}w_{6}w_{4}w_{5}w_{2}w_{4}w_{3}w_{1}
) \\
&(\para{P}_4,\frac{1}{12},\para{P}_5,\frac{1}{30},
w_{5}w_{6}w_{7}w_{4}w_{5}w_{6}w_{3}w_{4}w_{5}w_{1}w_{3}w_{4}w_{2}
)
\end{align*}
\end{remk}
Therefore we get the following list:

\begin{figure}[H]
\begin{center}
\begin{tikzpicture}[thin,scale=2]
\tikzset{
  text style/.style={
    sloped, 
    text=black,
    font=\tiny,
    above
  }
}
\matrix (m) [matrix of math nodes, row sep=1em, column sep=3em]
    { 		
       |[name=p2314]| \text{I}_{\para{P}_2}(\frac{3}{14})
        &  |[name=p1334]| \text{I}_{\para{P}_1}(\frac{3}{34})
       & &
		|[name=p227]| \text{I}_{\para{P}_2}(\frac{2}{7})
       &  |[name=p719]| \text{I}_{\para{P}_7}(\frac{1}{9})
      \\
      |[name=p515]| \text{I}_{\para{P}_5}(\frac{1}{5})
              &  |[name=p217]| \text{I}_{\para{P}_2}(\frac{1}{7})
             & &
      		|[name=p525]| \text{I}_{\para{P}_5}(\frac{2}{5})
             &  |[name=p713]| \text{I}_{\para{P}_7}(\frac{1}{3})
            \\
      |[name=p3922]| \text{I}_{\para{P}_3}(\frac{9}{22})
              &  |[name=p11334]| \text{I}_{\para{P}_1}(\frac{13}{34})
             & 
|[name=p418]| \text{I}_{\para{P}_4}(\frac{1}{8})
&
|[name=p5110]| \text{I}_{\para{P}_5}(\frac{1}{10})
&
|[name=p3122]| \text{I}_{\para{P}_3}(\frac{1}{22})

            \\
      |[name=p3722]| \text{I}_{\para{P}_3}(\frac{7}{22})
              &  |[name=p716]| \text{I}_{\para{P}_7}(\frac{1}{6})
             & &
      		|[name=p61126]| \text{I}_{\para{P}_6}(\frac{11}{26})
             &  |[name=p7718]| \text{I}_{\para{P}_7}(\frac{7}{18})
            \\
|[name=p2128]| \text{I}_{\para{P}_2}(\frac{1}{28})      &
|[name=p5320]| \text{I}_{\para{P}_5}(\frac{3}{20}) & &
|[name=p4316]| \text{I}_{\para{P}_4}(\frac{3}{16})
&
|[name=p6113]| \text{I}_{\para{P}_6}(\frac{1}{13})
            \\
|[name=p4112]| \text{I}_{\para{P}_4}(\frac{1}{12})
&
|[name=p5130]| \text{I}_{\para{P}_5}(\frac{1}{30})            
\\
      & |[name=p1734]| \text{I}_{\para{P}_1}(\frac{7}{34})
      & & &  |[name=p2114]| \text{I}_{\para{P}_2}(\frac{1}{14})
      \\
	|[name=p5310]| \text{I}_{\para{P}_5}(\frac{3}{10})
                    &  
	|[name=p7118]|                    \text{I}_{\para{P}_7}(\frac{1}{18})
                   & &  
            		|[name=p3322]| \text{I}_{\para{P}_3}(\frac{3}{22})
            		&
\\
      & |[name=p6726]| \text{I}_{\para{P}_6}(\frac{7}{26})
      & & &  |[name=p6126]| \text{I}_{\para{P}_6}(\frac{1}{26})
      \\
& 
|[name=p11134]| \text{I}_{\para{P}_1}(\frac{11}{34})
& & 
|[name=p6526]| \text{I}_{\para{P}_6}(\frac{5}{26})
&
|[name=p1134]| \text{I}_{\para{P}_1}(\frac{1}{34})\\
\\
|[name=p438]| \text{I}_{\para{P}_4}(\frac{3}{8})
      & 
      |[name=p2514]| \text{I}_{\para{P}_2}(\frac{5}{14})
      
      &  
|[name=p414]| \text{I}_{\para{P}_4}(\frac{1}{4})
      \\
&
|[name=p7518]| \text{I}_{\para{P}_7}(\frac{5}{18})
& & 
|[name=p3522]| \text{I}_{\para{P}_3}(\frac{5}{22})
      &
\\
& & & 

    \\};

\draw[->]      
(p2314)  edge [-]
node[text style,above]{}  (p1334)

(p227)  edge [-]
node[text style,above]{}  (p719)

(p3922)  edge [-]
node[text style,above]{}  (p11334)

(p414)  edge [-]
node[text style,above]{}  (p3522)

(p515)  edge [-]
node[text style,above]{}  (p217)

(p525)  edge [-]
node[text style,above]{}  (p713)

(p61126)  edge [-]
node[text style,above]{}  (p7718)

(p3722)  edge [-]
node[text style,above]{}  (p716)

(p3322)  edge [-]
node[text style,above]{}  (p2114)

(p3322)  edge [-]
node[text style,above]{}  (p6126)

(p5310)  edge [-]
node[text style,above]{}  (p1734)

(p5310)  edge [-]
node[text style,above]{}  (p6726)

(p438)  edge [-]
node[text style,above]{}  (p11134)
(p438)  edge [-]
node[text style,above]{}  (p2514)
(p438)  edge [-]
node[text style,above]{}  (p7518)

(p418)  edge [-]
node[text style,above]{}  (p5110)

(p5110)  edge [-]
node[text style,above]{}  (p3122)

(p414)  edge [-]
node[text style,above]{}  (p6526)
(p6526)  edge [-]
node[text style,above]{}  (p1134)

(p5310)  edge [-]
node[text style,above]{}  (p7118)

(p2128)  edge [-]
node[text style,above]{}  (p5320)

(p4316)  edge [-]
node[text style,above]{}  (p6113)
(p4112)  edge [-]
node[text style,above]{}  (p5130)
;
\end{tikzpicture}
\label{digram:E7s}
\end{center}
\end{figure}
\appendix 
\newpage
\chapter{Appendix 1}
\label{Ap1:Code_example}
\lhead{Appendix A. \emph{Standard output}}
\begin{defn}
For every $a \in \mathbb{C} \setminus \{0,1\}$ we write the Taylor series of $\zeta(s)$ around $s=a$ as follows
$$ \zeta(s) = \sum_{j=0}^{\infty} \zeta(a)_{j}(s-a)^{j}.$$ \\
For $a=1$ we write 
$$\zeta(s) = \sum_{j=-1}^{\infty} c_{j}(s-1)^{j}.$$
\end{defn}

\section{Standard output}
 \begin{thm} 
Let $G=F_{4}$, and let $\para{P}=\para{P}_{1}$. Then $E_{\para{P}}(f^{0},g,s)$ admits a pole of order $ 1 $ at $s= \frac{1}{4} $. 
 Moreover, the leading term $\leadingterm{-1 }(f^{0},g, \frac{1}{4} )$ is in $L^{2}$.
 \end{thm} 
  \begin{proof} 
 As we can see from Table the exponent $ \left[-5, -9, -12, -5\right] $   contributes a pole of order $1$ and it cannot be canceled. Hence it remains to
  show that the exponent $ \left[-5, -9, -13, -7\right] $ does not contribute a pole of order $2$.
 \\
 For the exp. $ \left[-5, -9, -13, -7\right] $ we do the follows:  \\
 First of all, we take only the operators that contribute an order great or equal to $1$. In that case,  all the elements in the exponent contribute a pole of order $2$.
 Let: 
 \begin{align*} 
 y1&=\zeta( 16 s - 3 ) & 
  y2&=\zeta( 8 s - 2 ) &
 y3&=\zeta( 8 s ) &  
  y4&=\zeta( 16 s + 4 )  \\
  y5&=\zeta( 8 s + 1 ) & 
  y6&=\zeta( 8 s + 4 ) &
 y7&=\zeta( 8 s - 1 ) & 
  \end{align*} 
 Therefore after doing a common denominator we get the following expression $$ \frac{y_{1} y_{2} y_{3} + y_{1} y_{3} y_{7}}{y_{4} y_{5} y_{6}} .$$ Since the denominator is holomorphic and non zero we may ignore it. For every $y_i$ we write its Laurent expansion around  $ \frac{1}{4} $ : \begin{align*} 
 y1&= \frac{\frac{1}{16} c_{-1}}{(s- \frac{1}{4} )} + c_{ 0 } + 16 c_{ 1 }(s- \frac{1}{4} ) + 256 c_{ 2 }(s- \frac{1}{4} )^{2} + 4096 c_{ 3 }(s- \frac{1}{4} )^{3}
 + \dots
 \\ 
 y2&= \frac{-\frac{1}{8} c_{-1}}{(s- \frac{1}{4} )} + c_{ 0 } - 8 c_{ 1 }(s- \frac{1}{4} ) + 64 c_{ 2 }(s- \frac{1}{4} )^{2} - 512 c_{ 3 }(s- \frac{1}{4} )^{3}
 + \dots
 \\ 
 y3&= \zeta( 2 )_{ 0 } + 8 \zeta( 2 )_{ 1 }(s- \frac{1}{4} ) + 64 \zeta( 2 )_{ 2 }(s- \frac{1}{4} )^{2} + 512 \zeta( 2 )_{ 3 }(s- \frac{1}{4} )^{3}
+ \dots
 \\ 
 y4&= \zeta( 8 )_{ 0 } + 16 \zeta( 8 )_{ 1 }(s- \frac{1}{4} ) + 256 \zeta( 8 )_{ 2 }(s- \frac{1}{4} )^{2} + 4096 \zeta( 8 )_{ 3 }(s- \frac{1}{4} )^{3}
 + \dots
 \\ 
 y5&= \zeta( 3 )_{ 0 } + 8 \zeta( 3 )_{ 1 }(s- \frac{1}{4} ) + 64 \zeta( 3 )_{ 2 }(s- \frac{1}{4} )^{2} + 512 \zeta( 3 )_{ 3 }(s- \frac{1}{4} )^{3}
 + \dots
 \\ 
 y6&= \zeta( 6 )_{ 0 } + 8 \zeta( 6 )_{ 1 }(s- \frac{1}{4} ) + 64 \zeta( 6 )_{ 2 }(s- \frac{1}{4} )^{2} + 512 \zeta( 6 )_{ 3 }(s- \frac{1}{4} )^{3}
 + \dots
 \\ 
 y7&= \frac{\frac{1}{8} c_{-1}}{(s- \frac{1}{4} )} + c_{ 0 } + 8 c_{ 1 }(s- \frac{1}{4} ) + 64 c_{ 2 }(s- \frac{1}{4} )^{2} + 512 c_{ 3 }(s- \frac{1}{4} )^{3} + \dots
 \end{align*} 
 For the summand: $ y_{1} y_{2} y_{3} $ we get : \\ 
 $$\frac{-\frac{1}{128} c_{-1}^{2} \zeta( 2 )_{ 0 }}{(s- \frac{1}{4} )^{2}} + \frac{-\frac{1}{16} c_{-1} c_{ 0 } \zeta( 2 )_{ 0 } - \frac{1}{16} c_{-1}^{2} \zeta( 2 )_{ 1 }}{(s- \frac{1}{4} )} + O(1)$$  
 For the summand: $ y_{1} y_{3} y_{7} $ we get : 
 $$\frac{\frac{1}{128} c_{-1}^{2} \zeta( 2 )_{ 0 }}{(s- \frac{1}{4} )^{2}} + \frac{\frac{3}{16} c_{-1} c_{ 0 } \zeta( 2 )_{ 0 } + \frac{1}{16} c_{-1}^{2} \zeta( 2 )_{ 1 }}{(s- \frac{1}{4} )} + O(1)$$  
 In conclusion the final sum is: $$\frac{\frac{1}{8} c_{-1} c_{ 0 } \zeta( 2 )_{ 0 }}{(s- \frac{1}{4} )} + O(1).$$
  Hence, this exponent contributes a pole of at most  order $ 1 $. The leading term is in $L^{2}$
 by Langlands' criterion
 .
 \end{proof} 
 
 \begin{landscape} 
 \begin{longtable}{|c|c|c|c|c|} 
\hline pole & order & operator & factor & exp \\ 
 \hline 
 $ \frac{1}{4} $ & $ 2 $ & $ w_{2}w_{3}w_{4}w_{2}w_{3}w_{1}w_{2}w_{3}w_{4}w_{1}w_{2}w_{3}w_{2}w_{1} $ &$\frac{\zeta( 16 s -3 ) \zeta( 8 s -2 ) \zeta( 8 s)}{\zeta( 16 s+ 4 ) \zeta( 8 s+ 1 ) \zeta( 8 s+ 4 )}$&$ \left[-5, -9, -13, -7\right] $ \\ 
 $ \frac{1}{4} $ & $ 2 $ & $ w_{3}w_{4}w_{2}w_{3}w_{1}w_{2}w_{3}w_{4}w_{1}w_{2}w_{3}w_{2}w_{1} $ &$\frac{\zeta( 16 s -3 ) \zeta( 8 s) \zeta( 8 s -1 )}{\zeta( 16 s+ 4 ) \zeta( 8 s+ 1 ) \zeta( 8 s+ 4 )}$&$ \left[-5, -9, -13, -7\right] $ \\ 
 \hline $ \frac{1}{4} $ & $ 1 $ & $ w_{2}w_{3}w_{1}w_{2}w_{3}w_{4}w_{1}w_{2}w_{3}w_{2}w_{1} $ &$\frac{\zeta( 8 s) \zeta( 8 s -1 ) \zeta( 16 s -1 )}{\zeta( 16 s+ 4 ) \zeta( 8 s+ 1 ) \zeta( 8 s+ 4 )}$&$ \left[-5, -9, -12, -5\right] $ \\ 
 \hline $ \frac{1}{4} $ & $ 1 $ & $ w_{1}w_{2}w_{3}w_{4}w_{2}w_{3}w_{1}w_{2}w_{3}w_{4}w_{1}w_{2}w_{3}w_{2}w_{1} $ &$\frac{\zeta( 16 s -3 ) \zeta( 8 s -3 ) \zeta( 8 s)}{\zeta( 16 s+ 4 ) \zeta( 8 s+ 1 ) \zeta( 8 s+ 4 )}$&$ \left[-4, -9, -13, -7\right] $ \\ 
 \hline $ \frac{1}{4} $ & $ 1 $ & $ w_{4}w_{2}w_{3}w_{1}w_{2}w_{3}w_{4}w_{1}w_{2}w_{3}w_{2}w_{1} $ &$\frac{\zeta( 8 s) \zeta( 8 s -1 ) \zeta( 16 s -2 )}{\zeta( 16 s+ 4 ) \zeta( 8 s+ 1 ) \zeta( 8 s+ 4 )}$&$ \left[-5, -9, -12, -7\right] $ \\ 
 \hline $ \frac{1}{4} $ & $ 0 $ & $ 1 $ &$ 1 $ & $ \left[4, 3, 3, 1\right] $ \\ 
 \hline $ \frac{1}{4} $ & $ 0 $ & $ w_{1} $ &$\frac{\zeta( 8 s+ 3 )}{\zeta( 8 s+ 4 )}$&$ \left[-1, 3, 3, 1\right] $ \\ 
 \hline $ \frac{1}{4} $ & $ 0 $ & $ w_{2}w_{1} $ &$\frac{\zeta( 8 s+ 2 )}{\zeta( 8 s+ 4 )}$&$ \left[-1, -1, 3, 1\right] $ \\ 
 \hline $ \frac{1}{4} $ & $ 0 $ & $ w_{3}w_{2}w_{1} $ &$\frac{\zeta( 16 s+ 3 ) \zeta( 8 s+ 2 )}{\zeta( 16 s+ 4 ) \zeta( 8 s+ 4 )}$&$ \left[-1, -1, -4, 1\right] $ \\ 
 \hline $ \frac{1}{4} $ & $ 0 $ & $ w_{2}w_{3}w_{2}w_{1} $ &$\frac{\zeta( 16 s+ 3 ) \zeta( 8 s+ 1 )}{\zeta( 16 s+ 4 ) \zeta( 8 s+ 4 )}$&$ \left[-1, -4, -4, 1\right] $ \\ 
 \hline $ \frac{1}{4} $ & $ 0 $ & $ w_{1}w_{2}w_{3}w_{2}w_{1} $ &$\frac{\zeta( 16 s+ 3 ) \zeta( 8 s)}{\zeta( 16 s+ 4 ) \zeta( 8 s+ 4 )}$&$ \left[-3, -4, -4, 1\right] $ \\ 
 \hline $ \frac{1}{4} $ & $ 0 $ & $ w_{4}w_{3}w_{2}w_{1} $ &$\frac{\zeta( 8 s+ 2 ) \zeta( 16 s+ 2 )}{\zeta( 16 s+ 4 ) \zeta( 8 s+ 4 )}$&$ \left[-1, -1, -4, -5\right] $ \\ 
 \hline $ \frac{1}{4} $ & $ 0 $ & $ w_{4}w_{2}w_{3}w_{2}w_{1} $ &$\frac{\zeta( 8 s+ 1 ) \zeta( 16 s+ 2 )}{\zeta( 16 s+ 4 ) \zeta( 8 s+ 4 )}$&$ \left[-1, -4, -4, -5\right] $ \\ 
 \hline $ \frac{1}{4} $ & $ 0 $ & $ w_{3}w_{4}w_{2}w_{3}w_{2}w_{1} $ &$\frac{\zeta( 8 s+ 1 ) \zeta( 16 s+ 1 )}{\zeta( 16 s+ 4 ) \zeta( 8 s+ 4 )}$&$ \left[-1, -4, -9, -5\right] $ \\ 
 \hline $ \frac{1}{4} $ & $ 0 $ & $ w_{2}w_{3}w_{4}w_{2}w_{3}w_{2}w_{1} $ &$\frac{\zeta( 8 s) \zeta( 16 s+ 1 )}{\zeta( 16 s+ 4 ) \zeta( 8 s+ 4 )}$&$ \left[-1, -6, -9, -5\right] $ \\ 
 \hline $ \frac{1}{4} $ & $ 0 $ & $ w_{4}w_{1}w_{2}w_{3}w_{2}w_{1} $ &$\frac{\zeta( 8 s) \zeta( 16 s+ 2 )}{\zeta( 16 s+ 4 ) \zeta( 8 s+ 4 )}$&$ \left[-3, -4, -4, -5\right] $ \\ 
 \hline $ \frac{1}{4} $ & $ 0 $ & $ w_{3}w_{4}w_{1}w_{2}w_{3}w_{2}w_{1} $ &$\frac{\zeta( 8 s) \zeta( 16 s+ 1 )}{\zeta( 16 s+ 4 ) \zeta( 8 s+ 4 )}$&$ \left[-3, -4, -9, -5\right] $ \\ 
 \hline $ \frac{1}{4} $ & $ 0 $ & $ w_{2}w_{3}w_{4}w_{1}w_{2}w_{3}w_{2}w_{1} $ &$\frac{\zeta( 16 s) \zeta( 8 s)}{\zeta( 16 s+ 4 ) \zeta( 8 s+ 4 )}$&$ \left[-3, -8, -9, -5\right] $ \\ 
 \hline $ \frac{1}{4} $ & $ 0 $ & $ w_{3}w_{2}w_{3}w_{4}w_{1}w_{2}w_{3}w_{2}w_{1} $ &$\frac{\zeta( 8 s) \zeta( 16 s -1 )}{\zeta( 16 s+ 4 ) \zeta( 8 s+ 4 )}$&$ \left[-3, -8, -12, -5\right] $ \\ 
 \hline $ \frac{1}{4} $ & $ 0 $ & $ w_{1}w_{2}w_{3}w_{4}w_{2}w_{3}w_{2}w_{1} $ &$\frac{\zeta( 16 s) \zeta( 8 s)}{\zeta( 16 s+ 4 ) \zeta( 8 s+ 4 )}$&$ \left[-5, -6, -9, -5\right] $ \\ 
 \hline $ \frac{1}{4} $ & $ 0 $ & $ w_{1}w_{2}w_{3}w_{4}w_{1}w_{2}w_{3}w_{2}w_{1} $ &$\frac{\zeta( 16 s) \zeta( 8 s)^{2}}{\zeta( 16 s+ 4 ) \zeta( 8 s+ 1 ) \zeta( 8 s+ 4 )}$&$ \left[-5, -8, -9, -5\right] $ \\ 
 \hline $ \frac{1}{4} $ & $ 0 $ & $ w_{3}w_{1}w_{2}w_{3}w_{4}w_{1}w_{2}w_{3}w_{2}w_{1} $ &$\frac{\zeta( 8 s)^{2} \zeta( 16 s -1 )}{\zeta( 16 s+ 4 ) \zeta( 8 s+ 1 ) \zeta( 8 s+ 4 )}$&$ \left[-5, -8, -12, -5\right] $ \\ 
 \hline $ \frac{1}{4} $ & $ 0 $ & $ w_{4}w_{3}w_{2}w_{3}w_{4}w_{1}w_{2}w_{3}w_{2}w_{1} $ &$\frac{\zeta( 8 s) \zeta( 16 s -2 )}{\zeta( 16 s+ 4 ) \zeta( 8 s+ 4 )}$&$ \left[-3, -8, -12, -7\right] $ \\ 
 \hline $ \frac{1}{4} $ & $ 0 $ & $ w_{4}w_{3}w_{1}w_{2}w_{3}w_{4}w_{1}w_{2}w_{3}w_{2}w_{1} $ &$\frac{\zeta( 8 s)^{2} \zeta( 16 s -2 )}{\zeta( 16 s+ 4 ) \zeta( 8 s+ 1 ) \zeta( 8 s+ 4 )}$&$ \left[-5, -8, -12, -7\right] $ \\ 
 \hline \end{longtable}
 
 \end{landscape}

\backmatter


\label{Bibliography}

\lhead{\emph{Bibliography}} 

\bibliographystyle{unsrtnat} 

\bibliography{Bibliography} 

\end{document}